\documentclass[11pt]{amsart}
\usepackage{amssymb,amsmath,latexsym,amsthm}
\usepackage[varg]{pxfonts}
\usepackage{hyperref}

\usepackage[utf8]{inputenc}

\usepackage[T1]{fontenc}

\usepackage{placeins}
\usepackage{bbm}
\usepackage{subcaption}
\usepackage{pgfplots}
\usepackage{xcolor}
\usepackage{enumitem}
\pgfplotsset{compat=1.15}
\usepackage{mathrsfs}
\usepackage{bold-extra}
\usetikzlibrary{arrows}
\usepackage{mathtools}
\usepackage{multirow}
\usepackage{cleveref}
\usepackage{bold-extra}
\usepackage{diagbox}
\usepackage{tikz}
\usepackage{float}
\usetikzlibrary{positioning}
\usetikzlibrary{calc}

\newcommand{\m}[1]{\mathbf{\uppercase{#1}}}
\newcommand{\n}{\scriptscriptstyle{\mathsf{N}}}
\newcommand{\co}{\scriptscriptstyle{\mathsf{C}}}

\newtheorem{thm}{Theorem}
\newtheorem{cor}[thm]{Corollary}
\newtheorem{lem}[thm]{Lemma}
\newtheorem{prp}[thm]{Proposition}

\theoremstyle{definition}
\newtheorem{dfn}[thm]{Definition}
\newtheorem{exmp}{Example}[section]
\theoremstyle{remark}
\newtheorem*{rem}{Remark}
\theoremstyle{plain}
\newtheorem{claim}[thm]{Claim}

\def\T{\mathbf{T}}

\begin{document}

\title{On 2-generated minimal Taylor algebras of size 4}
\author[affil1]{Zarathustra Brady}
\email{notzeb@gmail.com}
\author[affil2]{Petar Đapi\'c}
\email{petarn@dmi.uns.ac.rs}
\author[affil2]{Vuka\v sin Đinovi\' c}
\email{vukasindjinovic04@gmail.com}
\author[affil2]{Petar Markovi\'c}
\email{pera@dmi.uns.ac.rs}
\author[affil3]{Aleksandar Proki\'{c}}
\email{aprokic@uns.ac.rs}
\author[affil2]{Veljko Tolji\' c}
\email{veljko.toljic@dmi.uns.ac.rs}
\author[affil2]{Vlado Uljarevi\'c}
\email{vlado.uljarevic@dmi.uns.ac.rs}
\address[affil1]{No academic affiliation}
\address[affil2]{Department of Mathematics and Informatics\\ 
University of Novi Sad\\ Serbia}
\address[affil3]{Faculty of Technical Sciences\\ 
University of Novi Sad\\ Serbia}
\newcommand{\AuthorNames}{Z. Brady et al.}
\newcommand{\FilMSC}{Primary 08B05; Secondary 08A05, 08A45}
\newcommand{\FilKeywords}{minimal Taylor algebras}
\newcommand{\FilCommunicated}{Miroslav \'Ciri\'c}
\newcommand{\FilSupport}{P. Đapi\'c, P. Markovi\'c and V. Uljarevi\'c were supported by the Ministry of Education, Science and Technological Development of the Republic of Serbia (Grants No. 451-03-137/2025-03/200125 and 451-03-136/2025-03/200125). Aleksandar Proki\'{c} was supported by the Ministry of Education, Science and Technological Development of the Republic of Serbia (Grant No. 451-03-137/2025-03/200156.)}

\begin{abstract}
We prove that any 2-generated minimal Taylor algebra on a domain of size 4 is not simple. In addition, we find all such algebras up to isomorphism and term-equivalence.
\end{abstract}

\maketitle

\section{Introduction}
The study of Constraint Satisfaction Problem (CSP) and its complexity led to new directions in which Universal Algebra developed. In particular, the regularity of structures in classical approach to Universal Algebra came from having many term operations. The approach which seems most natural to CSP, however, is to squeeze regularity out of having the least amount of term operations among those algebras which have at least one term operation of some fixed type. This approach was formalized in \cite{dreamteam}, where the authors defined minimal Taylor algebras, algebras which have a Taylor term, but no proper subclone of their term clone contains a Taylor operation. The increased understanding of minimal Taylor algebras directly translates to better understanding of polynomial-time algorithms for solving CSP and simpler proofs of the main result in the complexity of CSP, the Dichotomy Theorem, a famously complicated result.

The colored edge theory invented by A. Bulatov in his proof of the Dichotomy Theorem defines edges on an algebra according to the types of 2-generated subalgebras. Problem 4.2.2 from \cite{Zebnotes} asks whether there exists a 2-generated simple minimal Taylor algebra with at least three elements which is not affine. If the answer is negative, then Bulatov's theory would be greatly simplified, since all large 2-generated algebras would be either affine, or have a nontrivial congruence. The nontrivial congruence would reduce the proofs of various theorems to the factor algebras and congruence classes, all of which would be smaller algebras for which the theorem would inductively hold. Indeed, many theorems in Bulatov's theory have as the hardest case the simple 2-generated algebras. If one could prove that all of those are just the three 2-element minimal Taylor algebras, plus the simple affine ones, it would shorten and trivialize several proofs in Bulatov's theory. 

After the preliminaries in Section 2, in Section 3 we consider the 4-element case of the above problem. It was proved that there are no 3-element 2-generated simple non-affine minimal Taylor algebras in \cite{jankovec}, since all 24 3-element minimal Taylor algebras were found in that thesis. In Section 3 we extend this result to 4-element 2-generated minimal Taylor algebras. Sadly, some parts of our proof don't extend to all minimal Taylor algebras so we are unable thus far to solve the abovementioned problem beyond the 4-element case.

In Section 4 we applied the result of Section 3 to find all 4-element 2-generated minimal Taylor algebras using a case analysis based on a maximal congruence such an algebra might have. Indeed, the factor modulo such a congruence must be a 2-element algebra or the simple affine algebra $\mathbb{Z}_3^{\mathrm{aff}}$, and we thus break down our search into four cases. In the end we find that there are 18 4-element 2-generated minimal Taylor algebras, up to term equivalence and isomorphism, and describe them all. 

Our case-analysis first assumes that there exists a homomorphism onto a 2-element semilattice, and find seven such 2-generated minimal Taylor algebras (Subsection 4.1). Next we assume that there exists a homomorphism onto the two-element majority algebra, but no homomorphism onto the 2-element semilattice, and find three such minimal Taylor algebras (Subsection 4.2). Next we assume that there is a homomorphism onto the two-element affine algebra, but neither of the previous two cases holds, and find three such minimal Taylor algebras (Subsection 4.3). Finally, we find the all six 2-generated minimal Taylor algebras which map onto the 3-element affine algebra, and find that one of them also maps onto the 2-element semilattice, while the remaining five round out our classification (Subsection 4.4).

We did not finish the classification of 4-element minimal Taylor algebras. The number of 4-generated (conservative) 4-element minimal Taylor algebras is not too hard to find, using a computer and combinatorics. Indeed, using Corollary 4.4.16 of \cite{Zebnotes} and computer, it was found that there are exactly 520 4-element conservative minimal Taylor algebras, up to isomorphism and term equivalence (for 5-element conservative minimal Taylor algebras this number grows to 2686891). However, the 3-generated 4-element case seems to be pretty wild, so classifying the 4-element minimal Taylor algebras remains open.

\section{Preliminaries}
This paper is in the area of Universal Algebra, so it is assumed that the reader is familiar with the terminology and basic results of that area. The readers who need help with basic Universal Algebra are advised to look up the missing parts in \cite{burris-sank}, or \cite{alvin1}, \cite{alvin2} and \cite{alvin3}. The investigation of minimal Taylor algebras was motivated by computational complexity of CSP, so we will refer a lot to the manuscript \cite{Zebnotes} by the first author of this paper, where almost all knowledge on CSP is collected, and several new facts are proven, and we will select the variants of definitions and results we use like defined in \cite{Zebnotes}. When the result was first proven elsewhere, we will also cite the original source as the alternative.

All algebras in the present paper are assumed to be \emph{idempotent}, which means each of them has all its singletons as subuniverses. 

A set of functions $\mathcal{C}$ on some nonempty set $A$ is called a \emph{clone} if it contains all \emph{projections} $\pi_i^k\colon A^k\rightarrow A$ which satisfy $\pi_i^k(x_1,\dots,x_k)=x_i$, and is closed under composition, the operation which maps $k$-ary function $f$ and functions $g_1,\dots,g_k$ of arity $l$ to the function $$(f\circ (g_1,\dots,g_k))\colon (x_1,\dots,x_l)\mapsto f(g_1(x_1,\dots,x_l),\dots, g_k(x_1,\dots,x_l)).$$
Given an algebra $\m a$, by a {\emph{clone of $\m a$}}, denoted $\mathrm{Clo}(\m a)$, we call the set of all term operations of $\m a$, that is the set of operations generated by the set of basic operations of $\m a$. We denote the collection of all operations of arity $n$ from $\mathrm{Clo}(\m a)$ by $\mathrm{Clo}_{n}(\m a)$. Two algebras, say $\m a$ and $\m b$, with the same domain, are \emph{term-equivalent} if $\mathrm{Clo}(\m a)=\mathrm{Clo}(\m b)$.

An algebra $\m a$ is a \emph{Taylor algebra} if it has an idempotent term operation $f$ which satisfies a system of identities of the form $$f(\dots,\ast,x,\ast,\dots )\approx f(\dots,\ast,y,\ast,\dots ),$$ for each coordinate, where $\ast$'s are substituted somehow with $x$'s and $y$'s. Such an operation $f$ is called a \emph{Taylor operation}.
A clone $\mathcal{C}$ on a finite domain is called \emph{minimal Taylor clone} if it contains a Taylor operation and every proper subclone of $\mathcal{C}$ is not Taylor. A finite algebra $\m a$ is a \emph{minimal Taylor algebra} if $\mathrm{Clo}(\m a)$ is a minimal Taylor clone. Equivalently, the algebra is a minimal Taylor algebra if it is Taylor, and it has no proper reduct that is also Taylor. The next proposition clarifies the existence of a minimal Taylor clone.

\begin{prp}[Proposition 5.2 of \cite{dreamteam}, Proposition 4.2.2 of \cite{Zebnotes}]
Every Taylor clone on a finite domain contains a minimal Taylor clone.
\end{prp}

It is not hard to show that the class of minimal Taylor algebras is closed under the basic constructions.

\begin{prp}[Proposition 5.4 of \cite{dreamteam}, Theorem 4.2.4 of \cite{Zebnotes}]\label{HSPfin}
Any subalgebra, finite power, or quotient of a minimal Taylor algebra is a minimal Taylor algebra.
\end{prp}

Given an nonempty set $A$ and a ternary function $f$ on it, we say that $f$ is a \emph{majority operation} on $A$ if it satisfies $f(a,a,b)=f(a,b,a)=f(b,a,a)=a$ for any $a,b\in A$. In the case of a two-element set $A=\{a,b\}$, there is only one majority operation and we denote it by $\mathrm{maj}_{a,b}$, or just $\mathrm{maj}$, while the algebra $(\{a,b\};\mathrm{maj})$ is the \emph{two-element majority algebra}. A \emph{Mal'cev operation} is a ternary operation $p$ satisfying $p(b,a,a)=p(a,a,b)=b$ for any $a,b\in A$. An interesting example of such an operation is the operation $x-y+z$ on $A$, which is sometimes denoted by $\mathrm{aff}$, where $+$ and $-$ are operations from an abelian group on $A$. The algebra $(A;x-y+z)$ is called the \emph{affine Mal'cev algebra} of the abelian group $(A;+)$. If that abelian group is $\mathbb{Z}_n$ for some $n\in\mathbb{N}$, the affine Mal'cev algebra of it is denoted $\mathbb{Z}_n^{\mathrm{aff}}$. The operation $p$ of the two-element affine Mal'cev algebra $\mathbb{Z}_2^{\mathrm{aff}}$ satisfies the so-called \emph{minority condition}, that is $p(b,a,a)=p(a,b,a)=p(a,a,b)=b$ for any $a,b$.  A \emph{meet-semilattice operation}, or just a \emph{semilattice operation}, is a binary operation $\wedge$ which is idempotent, commutative, and associative. An algebra $(A;\wedge)$, where $\wedge$ is semilattice operation, is called a \emph{semilattice}. An ordered set $(A,\leqslant)$ that has a greatest lower bound for every pair of elements can be observed as a semilattice and vice versa; we define $a\wedge b$ to be the infimum of elements $a,b\in A$. Conversely, given a semilattice $(A;\wedge)$, we get the order $\leqslant$ by setting $a\leqslant b$ if and only if $a\wedge b=a$. In this case, we say that $a$ \emph{absorbs} $b$.  Although the lattice of clones on a two-element domain, usually referred to as Post's lattice, is countable (see \cite{post}), there are only three minimal Taylor algebras on a two-element domain, up to isomorphism and term-equivalence. These are the 2-element semilattice, the majority algebra, and the affine Mal'cev algebra.

Algebras with a semilattice or a majority operation are examples of bounded width algebras, while any affine Mal'cev algebra of a nontrivial abelian group is not. An algebra $\m a$ is \emph{abelian} if there exists a congruence $\theta$ of $\m a \times \m a$ such that the diagonal relation $\Delta _{A}=\{(a,a)\colon a\in A\}$ is a $\theta$-class. In general, the notion of abelian is tightly connected to bounded width. Namely, the bounded width theorem from \cite{bkbw} characterizes the bounded width algebras as the algebras such that no quotient of a subalgebra of a power of it is a nontrivial abelian algebra. In Taylor algebras, according to  \cite{hobby-mckenzie}, abelian algebras are precisely those which are \emph{affine modules}, the algebras whose term operations are exactly idempotent term operations of a module over a unit ring. Furthermore, we have the following simplification for minimal Taylor algebras.

\begin{prp}\label{mintayloraff}[see Theorem 3.12.8, Proposition 1.9.6 and Proposition 4.2.14 of \cite{Zebnotes}]
If $\m a$ is an abelian minimal Taylor algebra, then $\m a$ is term-equivalent to the affine Mal'cev algebra of some abelian group $(A;+)$.
\end{prp}

\subsection{Minimal Taylor algebras on domain of size 3}\label{s3elem}
Here we summarize results concerning three-element minimal Taylor algebras from \cite{Zebnotes} since we shall recall these algebras on numerous occasions later. Clones of those algebras are later studied in \cite{jankovec}, from which we adopt the notation. We present them the same as in \cite{jankovec}. Before that, to do this efficiently, we introduce a few notions. 
    

\begin{dfn}
Let $f$ be a $k$-ary operation on the set $A$. We say that $f$ is
\begin{itemize}
    \item \emph{symmetric} if the result does not depend on the order of arguments;
    \item \emph{cyclic} if $f(a_1,a_2,\dots,a_k)=f(a_2,\dots,a_k,a_1)$ for any $a_1,\dots,a_k\in A$, i.e., the result does not depend on the cyclic permutations of arguments;
    \item \emph{conservative} if $f(a_1,\dots,a_k)\in\{a_1,\dots,a_k\}$ for any $a_1,\dots,a_k\in A$.
\end{itemize}
An algebra $\m a$ is said to be \emph{conservative} if all the basic operations of $\m a$ are conservative. Clearly, an algebra $\m a$ is conservative if and only if every subset of it is a subalgebra.
\end{dfn}
We are ready to give the list of minimal Taylor algebras on the domain $\{0,1,2\}$.
There are 24 such algebras up to isomorphism and term-equivalence, among which 5 are not conservative. A single basic operation generates a clone for each, and a generating operation is binary or ternary.

\subsubsection*{Nonconservative algebras}
The first three of the nonconservative algebras (\Cref{table:1}) have a single basic operation $g$, which is symmetric. The fourth algebra ${\m T}_4^{\n}$ is the semilattice algebra with binary operation $\wedge$ determined by the ordering $0 \leqslant 1$, $0 \leqslant 2$. The algebra ${\m T}_5^{\n}$ is actually the affine Mal'cev algebra $\mathbb{Z}^{\mathrm{aff}}_3$. We note that $\mathrm{min}^n_{a\leqslant b}$ in the table denotes the $n$-ary operation on $\{a,b\}$ which returns the minimum of the arguments with respect to the ordering $a\leqslant b$.
\begin{table}[h!]
\centering
 \begin{tabular}{|c||c|c|c|c|c|}
\hline
    Algebra & $g{\restriction_{\{0,1\}}}$ & $g{\restriction_{\{0,2\}}}$ & $g(1,1,2)$ & $g(1,2,2)$ & $g(0,1,2)$ \\
\hline\hline
     $\T^{\n}_1$  & maj& $\text{min}^3_{0\leqslant 2}$ & 1 & 0 & 0\\
\hline
    $\T^{\n}_2$  & aff & $\text{min}^3_{0\leqslant 2}$ & 0 & 1 & 1\\
\hline
    $\T^{\n}_3$  & maj & aff & 2 & 0 & 2\\
\hline
    $\T^{\n}_4$  & $\text{min}^3_{0\leqslant 1}$ & $\text{min}^3_{0\leqslant 2}$ & 0 & 0 & 0\\
\hline
    $\T^{\n}_5$  & \multicolumn{5}{|c|}{$\mathbb{Z}^{\mathrm{aff}}_3$}\\
\hline
\end{tabular}  
\caption{Nonconservative algebras.}
\label{table:1}
\end{table}

\subsubsection*{Conservative algebras with a commutative binary operation }
    As listed in the table below, there are precisely two algebras of this kind. The first one, $\T^{\scriptscriptstyle{\mathsf{S}}}_1$, is sometimes referred to as \emph{rock-paper-scissors algebra}.
\begin{table}[!ht]
\centering
 \begin{tabular}{|c||c|c|c|}
\hline
    Algebra & $t{\restriction_{\{0,1\}}}$ & $t{\restriction_{\{1,2\}}}$ & $t{\restriction_{\{0,2\}}}$ \\
\hline\hline
     $\T^{\scriptscriptstyle{\mathsf{S}}}_1$  & $\text{min}^2_{0\leqslant 1}$ & $\text{min}^2_{1\leqslant 2}$ & $\text{min}^2_{2\leqslant 0}$ \\
\hline
    $\T^{\scriptscriptstyle{\mathsf{S}}}_2$  & $\text{min}^2_{0\leqslant 1}$ & $\text{min}^2_{1\leqslant 2}$ & $\text{min}^2_{0\leqslant 2}$ \\
 \hline
\end{tabular}  
\caption{Conservative algebras with commutative binary term.}
\label{table:2}
\end{table}
\FloatBarrier

\subsubsection*{Conservative algebras without commutative binary or cyclic ternary operation}
There are two such algebras: a \emph{dual discriminator algebra} $\m t_1^{\scriptscriptstyle{\mathsf{P}}}$ and a simple nonabelian Mal'cev algebra $\m t_2^{\scriptscriptstyle{\mathsf{P}}}$ (\Cref{table:3}). 

\begin{table}[hbt!]
\centering
\begin{tabular}{|c||c|c|c|c|}
\hline
    Algebra & $f{\restriction_{\{0,1\}}}$ & $f{\restriction_{\{1,2\}}}$ & $f{\restriction_{\{0,2\}}}$ & $f(x,y,z)\:\text{ if }\: \{x,y,z\}=\{0,1,2\}$ \\
\hline\hline
     $\T^{\scriptscriptstyle{\mathsf{P}}}_1$  & maj & maj & maj & \multirow{2}{*}{$x$}\\

\cline{1-4}
    $\T^{\scriptscriptstyle{\mathsf{P}}}_2$  & aff & aff & aff & \\
 \hline
\end{tabular}  
\caption{Conservative algebras without symmetric binary or cyclic ternary term.}
\label{table:3}
\end{table}

\subsubsection*{Conservative algebras with a ternary cyclic operation}
The following list includes a basic cyclic operation for each of the 15 conservative algebras with such an operation. 
\begin{table}[!ht]
\centering
 \begin{tabular}{|c||c|c|c|c|c|}
\hline
    Algebra & $g{\restriction_{\{0,1\}}}$ & $g{\restriction_{\{1,2\}}}$ & $g{\restriction_{\{0,2\}}}$ & $g(0,1,2)$ & $g(0,2,1)$ \\
\hline\hline
     $\T^{\co}_1$  & $\text{min}^3_{0\leqslant 1}$ & $\text{min}^3_{1\leqslant 2}$ & maj & 0 & 0\\
\hline
    $\T^{\co}_2$  & $\text{min}^3_{0\leqslant 1}$ & maj & $\text{min}^3_{0\leqslant 2}$ & 0 & 0\\
\hline
    $\T^{\co}_3$  & maj & $\text{min}^3_{1\leqslant 2}$ & $\text{min}^3_{0\leqslant 2}$ & 0 & 1\\
\hline
    $\T^{\co}_4$  & $\text{min}^3_{0\leqslant 1}$ & maj & maj & 0 & 0\\
\hline
    $\T^{\co}_5$  & $\text{min}^3_{0\leqslant 1}$ & aff & $\text{min}^3_{0\leqslant 2}$ & 0 & 0\\
\hline
    $\T^{\co}_6$  & $\text{min}^3_{0\leqslant 1}$ & $\text{min}^3_{1\leqslant 2}$ & aff & 0 & 0\\
\hline
    $\T^{\co}_7$  & aff & $\text{min}^3_{1\leqslant 2}$ & $\text{min}^3_{0\leqslant 2}$ & 0 & 1\\
\hline
    $\T^{\co}_8$  & $\text{min}^3_{0\leqslant 1}$ & aff & aff & 2 & 2\\
\hline
    $\T^{\co}_9$  & $\text{min}^3_{0\leqslant 1}$ & aff & maj & 0 & 0\\
\hline
    $\T^{\co}_{10}$  & $\text{min}^3_{0\leqslant 1}$ & maj & aff & 2 & 2\\
\hline
    $\T^{\co}_{11}$  & maj & aff & maj & 1 & 2\\
\hline
    $\T^{\co}_{12}$  & aff & maj & aff & 0 & 0\\
\hline
    $\T^{\co}_{13}$  & aff & aff & aff & 0 & 0\\
\hline
    $\T^{\co}_{14}$  & maj & maj & maj & 0 & 0\\
\hline
    $\T^{\co}_{15}$  & maj &  maj & maj & 1 & 2\\
\hline
\end{tabular}  
\caption{Conservative algebras with cyclic ternary term.}
\label{table:4}
\end{table}
\FloatBarrier



\subsection{On edges and absorption within minimal Taylor algebras}

In this part of the paper, we highlight some important notions and concepts from \cite{dreamteam}, where three approaches to the CSP were unified through minimal Taylor algebras. 

\begin{dfn}[Definition 4.3.1 of \cite{Zebnotes}]
Let $\m A$ be an algebra and $a,b\in A$.
\begin{itemize}
\item ({\em weak semilattice edge}) $(a, b)$ is a weak semilattice edge if there is a proper congruence $\theta$ on $\mathrm{Sg}\{a,b\}$ and a binary term $t$ such that $t(a/\theta, b/\theta) = t(b/\theta, a/\theta) = b/\theta$.
\item ({\em weak majority edge}) $\{a,b\}$ is a weak majority edge if there is a proper congruence $\theta$ on $\mathrm{Sg}\{a,b\}$ and a term $m\in \mathrm{Clo}_3({\m A})$ which acts as the majority operation on $\{a/\theta, b/\theta \}$.
\item ({\em weak affine edge}) $\{a,b\}$ is a weak affine edge if there is a proper congruence $\theta$ on $\mathrm{Sg}\{a,b\}$ and a term operation $p\in \mathrm{Clo}_3({\m A})$ such that $(\mathrm{Sg}\{a,b\}/\theta;p)$ is an affine Mal'cev algebra with respect to some abelian group $(\mathrm{Sg}\{a,b\}/\theta;+)$.
\end{itemize}
The congruence $\theta$ in the definition of the weak edges is called the {\em witness} of the edge. We say that
\begin{itemize}
\item ({\em semilattice edge}) $(a, b)$ is a semilattice edge if there exists a binary term $t$ such that $t(a, b) = t(b, a) = b$.
\item ({\em majority and affine edges}) We drop the label ``weak" from the majority or affine edge $\{a,b\}$ if the witness is a maximal congruence of $\mathrm{Sg}\{a,b\}$ and, moreover, if for every $a',b' \in A$ such that $(a,a'), (b,b')\in \theta$, we have $\mathrm{Sg}\{a',b'\}=\mathrm{Sg}\{a,b\}$. By a {\em strong} affine edge we will call an affine edge whose witnessing congruence is the identity.
\end{itemize}
\end{dfn}
Note that semilattice edges are directed, while majority and affine are not. One can notice that in Definition 3.1 of \cite{dreamteam}, the more general notion of so-called \emph{abelian edge} is defined instead of the affine edge. However, since we work with minimal Taylor algebras in this paper, where these two notions coincide (see Proposition~\ref{mintayloraff}), we use the term ``affine". 

We recall the following 

\begin{prp}[Proposition 5.3 of \cite{dreamteam}, Theorem 4.2.4 of \cite{Zebnotes}]\label{subuniverse}
Let $\m A$ be a minimal Taylor algebra and $B\subseteq A$ be closed under an operation $f\in \mathrm{Clo}({\m A})$ such that $B$ together with the restriction of $f$ to $B$ forms a Taylor algebra. Then $B$ is a subuniverse of ${\m A}$.
\end{prp}

Applying Proposition~\ref{subuniverse}, together with Proposition~\ref{HSPfin}, to weak edges and edges, we obtain

\begin{prp}\label{edgesaresubalgs}
Let $\m a$ be a minimal Taylor algebra and $a,b\in A$. If $(a,b)$ is a weak semilattice edge, or $\{a,b\}$ is a weak majority edge, with the witnessing congruence $\theta$, then $a/\theta\cup b/\theta$ is a subuniverse of $\m a$. In particular, if $(a,b)$ is a semilattice edge, then $\{a,b\}$ is a subuniverse of $\m a$.
\end{prp}

The following result is expressed in terminology of directed graphs (uses weak connectedness), so we can replace each undirected weak edge (majority or affine) with two directed edges, one in each direction. Another way to parse it is to replace weak connectedness with usual connectedness in the graph in which the orientation of weak semilattice edges is forgotten.

\begin{thm}[Theorem 5 of \cite{BulatovGraph1}, Theorem 4.3.8 of \cite{Zebnotes}]\label{connectthm}
A finite idempotent algebra $\m a$ has a Taylor term iff the graph of weak edges of every algebra $\m b\leq \m a$ is weakly connected.
\end{thm}

In \Cref{3elemgraf}, graphs of the first three nonconservative examples among minimal Taylor algebras on a three-element domain are given. The straight lines in the illustration represent edges which are not weak, while those that are weak are drawn bent. We adopt this convention throughout the paper. Additionally, lines are labeled $\text{s}$, $\text{m}$ or $\text{a}$, depending on the type. In a few cases, we would like to distinguish whether an affine edge comes from $\mathbb{Z}_3^{\mathrm{aff}}$, so those edges will be labeled $\text{a3}$. 
\begin{figure*}[h]
	\begin{subfigure}{0.33\textwidth}
        \centering
		\begin{tikzpicture}
			\tikzstyle{every node}=[draw,circle,fill=white,minimum size=3pt,inner sep=0pt]
			
			\node (0) at (0,{1*sqrt(3)}) [label=above:\strut{$0$}] {};
			\node (1) at (-1,0) [label=below:\strut{$1$}] {};
			\node (2) at (1,0) [label=below:\strut{$2$}] {};
			
			\path[->,line width=1.2pt,>=stealth, shorten <=0.7mm, shorten >=0.7mm]
			(2) edge node[draw=none,midway,above right,inner sep=0.8pt] {\scriptsize{s}} (0);
			\path[-,line width=1.2pt,dashed, shorten <=0.7mm, shorten >=0.7mm]
			(0) edge node[draw=none,midway,above left, inner sep=0.8pt] {\scriptsize{m}} (1);
			\path[-,line width=0.5pt,dashed, shorten <=0.7mm, shorten >=0.7mm]
			(2) edge[bend right=10] node[draw=none,midway,above] {\scriptsize{m}} (1);
			\path[->,line width=0.5pt,>=stealth, shorten <=0.7mm, shorten >=0.7mm]
			(2) edge[bend left=10] node[draw=none,midway,below,inner sep=0.8pt] {\scriptsize{s}} (1);

		\end{tikzpicture}
		\caption*{$\m t_1^{\n}$}
	\end{subfigure}%
	\hfill
	\begin{subfigure}{0.33\textwidth}
        \centering
		\begin{tikzpicture}
			\tikzstyle{every node}=[draw,circle,fill=white,minimum size=3pt,inner sep=0pt]
			
			\node (0) at (0,{1*sqrt(3)}) [label=above:\strut{$0$}] {};
			\node (1) at (-1,0) [label=below:\strut{$1$}] {};
			\node (2) at (1,0) [label=below:\strut{$2$}] {};
			
			\path[->,line width=1.2pt,>=stealth, shorten <=0.7mm, shorten >=0.7mm]
			(2) edge node[draw=none,midway,above right,inner sep=0.8pt] {\scriptsize{s}} (0);
			\path[-,line width=1.2pt,densely dotted, shorten <=0.7mm, shorten >=0.7mm]
			(0) edge node[draw=none,midway,above left, inner sep=0.8pt] {\scriptsize{a}} (1);
			\path[-,line width=0.5pt,densely dotted, shorten <=0.7mm, shorten >=0.7mm]
			(2) edge[bend right=10] node[draw=none,midway,above] {\scriptsize{a}} (1);
			\path[->,line width=0.5pt,>=stealth, shorten <=0.7mm, shorten >=0.7mm]
			(2) edge[bend left=10] node[draw=none,midway,below,inner sep=0.8pt] {\scriptsize{s}} (1);
			
		\end{tikzpicture}
		\caption*{$\m t_2^{\n}$}
	\end{subfigure}%
	\hfill
	\begin{subfigure}{0.33\textwidth}
        \centering
		\begin{tikzpicture}
			\tikzstyle{every node}=[draw,circle,fill=white,minimum size=3pt,inner sep=0pt]
			
			\node (0) at (0,{1*sqrt(3)}) [label=above:\strut{$0$}] {};
			\node (1) at (-1,0) [label=below:\strut{$1$}] {};
			\node (2) at (1,0) [label=below:\strut{$2$}] {};
			
			\path[-,line width=1.2pt,>=stealth,densely dotted, shorten <=0.7mm, shorten >=0.7mm]
			(2) edge node[draw=none,midway,above right,inner sep=0.8pt] {\scriptsize{a}} (0);
			\path[-,line width=1.2pt,dashed, shorten <=0.7mm, shorten >=0.7mm]
			(0) edge node[draw=none,midway,above left, inner sep=0.8pt] {\scriptsize{m}} (1);
			\path[-,line width=0.5pt,densely dotted, shorten <=0.7mm, shorten >=0.7mm]
			(2) edge[bend left=10] node[draw=none,midway,below] {\scriptsize{a}} (1);
			
		\end{tikzpicture}
		\caption*{$\m t_3^{\n}$}
	\end{subfigure}
    \caption{Directed graph of $\m t_i^{\n}$ for $i=1,2,3$.}
    \label{3elemgraf}
\end{figure*}
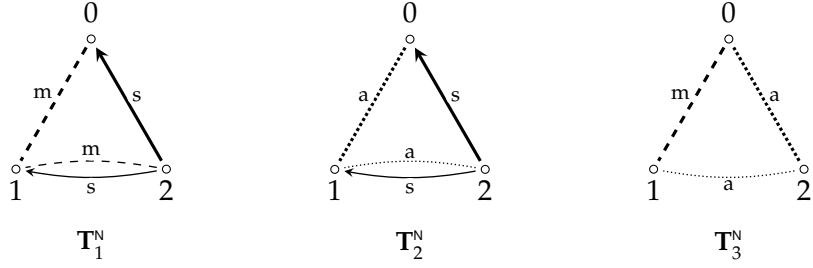	

Next we recall the fundamental notion of the absorption theory.
\begin{dfn}[Definition 3.3 of \cite{dreamteam}]\label{absdef}
Let ${\m A}$ be an algebra and $B\subseteq A$. We call $B$ an $n$-{\em absorbing set} of ${\m A}$ if there is a term operation $t\in \mathrm{Clo}_n({\m A})$ such that $t({\bf a})\in B$ whenever ${\bf a}\in A^{n}$ and $|\{i : a_i\in B\}|\geqslant n-1$. If, additionally, $B$ is a subuniverse of $\m a$, we write $B\trianglelefteq_n {\m A}$.
\end{dfn}
In the previous definition, we also say that $B$ $n$-{\em absorbs} $\m a$ {\em by} $t$ or that $t$ {\em witnesses the absorption} $B\trianglelefteq_n {\m A}$. A set which is 2-absorbing (resp. 3-absorbing) is also said to be a \emph{binary} (resp. \emph{ternary}) \emph{absorbing set}.  

As insinuated in Definition~\ref{absdef}, an absorbing set does not have to be closed under all term operations. However, in minimal Taylor algebras, absorbing sets are automatically subuniverses due to the following result, which fails in general Taylor algebras.

\begin{thm}[Theorem 5.5 of \cite{dreamteam}]\label{abssetsubuniv}
Let $\m a$ be a minimal Taylor algebra and $B$ an absorbing set of $\m a$. Then $B$ is a subuniverse of $\m a$.
\end{thm}




Several fruitful properties of binary and ternary absorbing sets in minimal Taylor algebras are collected in the following statements. Besides being effective, some of those are rather unexpected, such as the one when the union of two subuniverses is a subuniverse.  
\begin{prp}[Propositions 5.9 and 5.12 of \cite{dreamteam}]\label{composition}
Let $\m a$ be a minimal Taylor algebra and $B,C\subseteq A$. The following hold:
\begin{enumerate}
\item[$\mathrm{(1)}$] If $B\trianglelefteq_2\m a$ and $C\leq \m a$, then $B\cup C\leq \m a$.
\item[$\mathrm{(2)}$] If $C\trianglelefteq_2 B\trianglelefteq_2 \m a$, then $C\trianglelefteq_2 \m a$.
\item[$\mathrm{(3)}$] If $B,C\trianglelefteq_3 \m a$, then $B\cup C\leq \m a$ and $B\cap C\trianglelefteq_3 \m a$.
\item[$\mathrm{(4)}$] If $C\trianglelefteq_3 B\trianglelefteq_3 \m a$, then $C\trianglelefteq_3 \m a$.
\end{enumerate}
\end{prp}

\begin{prp}[Propostion 4.2.10 of \cite{Zebnotes}]\label{bradybinaryprop}
Let $\m a$ be a minimal Taylor algebra and $B\trianglelefteq_2\m a$. Then the following hold:
\begin{enumerate}
    \item[$\mathrm{(1)}$] $\m B$ is strongly absorbing subalgebra of $\m a$, that is, any term $f\in\mathrm{Clo}_n({\m A})$ which depends on its first input satisfies $f(b,a_1,\dots,a_{n-1})\in B$ for any $b\in B$ and $a_1,\dots,a_{n-1}\in A$.

    \item[$\mathrm{(2)}$] There is a congruence $\theta_B$ on $\m a$ such that $B$ is a congruence class of $\theta_B$, and all other congruence classes of $\theta_B$ are singletons;
    \item[$\mathrm{(3)}$] For every $a\notin B$, $\m b\cup \{a\}$ is a subalgebra of $\m a$, and $(\m b\cup \{a\})/\theta_B $ is a two-element semilattice with absorbing element $\m b/\theta_B$;
\end{enumerate}
\end{prp}

The next two results connect absorption with edges.

\begin{prp}[Proposition 9.7 of \cite{dreamteam}]\label{smedge}
Let $\m A$ be a minimal Taylor algebra:
\begin{itemize}
\item[$\mathrm{(1)}$] If $B\trianglelefteq_2 {\m A}$ and $a\in A\setminus B$ there exists $b\in B$ such that $(a,b)$ is a semilattice edge.
\item[$\mathrm{(2)}$] If $B\trianglelefteq_3 {\m A}$ and $a\in A\setminus B$ there exists $b\in B$ such that:
\begin{itemize}
\item[$\mathrm{a)}$] $(a,b)$ is a semilattice edge, or
\item[$\mathrm{b)}$] $\{a,b\}$ is a weak majority edge and $\mathrm{Sg}\{a,b\}\cap B$ is one of the two equivalence classes of the congruence witnessing it.
\end{itemize}
\end{itemize}
\end{prp}

\begin{thm}[Theorem 5.21 of \cite{dreamteam}]\label{singletoncenter}
If $\m a$ is a minimal Taylor algebra and $a\in A$ satisfies that there is no outgoing semilattice edge, nor any weak affine edge, connecting $a$ and any other element, then $\{a\}\trianglelefteq_3 {\m A}$
\end{thm}

As stated in \cite{dreamteam}, three types of operations arising from edges can be unified throughout a single operation, but when restricted to minimal Taylor clones, the common witnessing operation witnesses all the binary and ternary absorption, as well.
\begin{thm}[Theorem 5.23 of \cite{dreamteam}]\label{compositionactslike}
Every minimal Taylor algebra ${\m A}$ has a ternary term operation $m$ such that if $(a,b)$ is a weak edge witnessed by $\theta$ in ${\m E}=\mathrm{Sg}\{a,b\}$, then:
\begin{itemize}
\item if $(a,b)$ is a weak semilattice edge, then $m(x,y,z)=x\wedge y\wedge z$ on ${\m E}/\theta$ (where $b/\theta$ is the absorbing element);
\item if $\{a,b\}$ is a weak majority edge, then $m$ is the majority operation on ${\m E}/\theta$ (which has two elements);
\item if $\{a,b\}$ is a weak affine edge, then $m(x,y,z)=x-y+z$ on ${\m E}/\theta$;
\item $m$ witnesses all $3$-absorptions $B\trianglelefteq_3 {\m A}$;
\item each of the binary operations $m(x,x,y)$, $m(x,y,x)$ and $m(y,x,x)$ witnesses all $2$-absorptions $B\trianglelefteq_2 {\m A}.$
\end{itemize}
\end{thm}
\begin{rem}
Reading the proof of the previous theorem in \cite{dreamteam}, one can see that $m(x,x,y)=m(y,x,x)$ can be derived for all $x,y$.
\end{rem}
Any ternary operation that witnesses edges in the minimal Taylor algebra $\m a$ generates the whole clone $\mathrm{Clo}(\m a)$ as shown in the next theorem.
\begin{thm}[Theorem 5.24 of \cite{dreamteam}]\label{gen}
If ${\m A}$ is a minimal Taylor algebra, then $\mathrm{Clo}({\m A})$ is generated by any operation $m\in \mathrm{Clo}({\m A})$ satisfying the first three items in \Cref{compositionactslike}.
\end{thm}

\begin{prp}[Proposition 9.14 of \cite{dreamteam}]\label{ternaryabscyclic}
Let $\m a$ be a minimal Taylor algebra and $\m b\trianglelefteq_3\m a$. It $t$ is an $n$-ary cyclic term, $m\geqslant n/2$ and $a_1,\dots,a_n\in A$ such that $a_1,\dots,a_m\in B$, then $t(a_1,\dots,a_n)\in B$. In particular, if $t$ is a ternary cyclic term and two of its entries are in $B$, then the result is in $B$.
\end{prp}

\begin{thm}[Theorem 6.5 of \cite{dreamteam}]\label{Malcevcriterionedges}
Let $\m a$ be a minimal Taylor algebra. $\m a$ has no semilattice or weak majority edges iff $\m a$ has a Mal'cev term.
\end{thm}

\section{Existence of a nontrivial congruence}

\thispagestyle{empty}
In this section, we prove that any minimal Taylor algebra generated by two elements on a four-element domain is not simple. Our groundwork relies on the following result.

\begin{thm}[Theorem 9.16 of \cite{dreamteam}]\label{center} 
If ${\m A}$ is a minimal Taylor algebra that is generated by two distinct elements $a,b \in A$, then either $\m A$ has a nontrivial abelian quotient, or at least one of $a,b$ is contained in a proper ternary absorbing subuniverse of $\m A$.
\end{thm}

\subsection{Auxiliary results: the general case}

\begin{dfn}\label{linkdef}
Let ${\m A}_1$ and ${\m A}_2$ be algebras and $R \leq_{sd}{\m A}_1 \times {\m A}_2$. For any $i\in \{1,2\}$, the $i$-th {\em link tolerance} of $R$, denoted by $tol_i R$ is defined by
\[
\begin{gathered}
tol_1 R:=\{(a_1,a_1')\in A_1^2:(\exists a_2\in A_2)\
(a_1,a_2)\in R\text{ and }(a_1',a_2)\in R\},\\
tol_2 R:=\{(a_2,a_2')\in A_2^2:(\exists a_1\in A_1)\
(a_1,a_2)\in R\text{ and }(a_1,a_2')\in R\}.\\
\end{gathered}
\]
The transitive closure of $tol_iR$ is the $i$-th {\em link congruence} of $R$, denoted by $lk_iR$. We say $R$ is {\em linked} if its link congruences are full.
\end{dfn}

The next theorem is currently an unpublished result proved by Zhuk, and it is a stronger version of Theorem 4.2 of \cite{BartoKozik}, and establishes the existence of cyclic terms in finite Taylor algebras.

\begin{thm}[Theorem 4.1.8 of \cite{Zebnotes}]\label{cyclic}
Suppose ${\m A}$ is a finite idempotent Taylor algebra and that $p$ is prime. Then one of the following is true:
\begin{itemize}
\item[$\mathrm{(1)}$] either ${\m A}$ has a cyclic term of arity $p$, or
\item[$\mathrm{(2)}$] there exist $B\in \mathsf{HS}({\m A})$ and some automorphism $\sigma\in \mathrm{Aut}({\m B})$ such that $\sigma ^p =1$ and $\sigma$ has no fixed points.
\end{itemize}
In particular, if $p>|A|$, then ${\m A}$ has a cyclic term of arity $p$.
\end{thm}

When a minimal Taylor algebra ${\m A}=\mathrm{Sg}\{a,b\}$ is simple, it is useful to consider the algebra $R_{ab}=\mathrm{Sg}\{(a,b),(b,a)\}\leq_{sd} {\m A}^2$. Then $R_{ab}$ is either the graph of automorphism $\varphi$ of ${\m A}$ such that $\varphi (a)=b$ and $\varphi (b)=a$, or the link congruence $lk_1 R_{ab}$ is not the identity, which, since $\m a$ is simple, means that $R_{ab}$ is linked. When we consider a linked subdirect power of a Taylor algebra, we may use the following result.

\begin{thm}[\emph{Loop Lemma}; Theorem 3.5 of \cite{BartoKozik}]\label{looplemma}
Let $\m a$ be a finite Taylor algebra and $R\leq_{sd}\m a^2$ is linked. Then $R$ contains a loop, that is $(a,a)\in R$ for some $a\in A$.
\end{thm}

In the rest of this subsection, we introduce three lemmas that can be used whenever $\mathrm{Sg}\{a,b\}$ has more than two elements.

\begin{dfn}
We say that a binary term $t$ {\em has a dominant first coordinate }if for any $3$-absorbing subuniverse $C$ and all $c\in C$ and $d\in A$, it is true that $t(c,d)\in C$. A dominant second coordinate of a binary term is defined analogously.
\end{dfn}

\begin{lem}\label{dominant}
Every binary term $t$ of minimal Taylor algebra has a dominant first or second coordinate.
\end{lem}

\begin{proof}
Each such term $t$ belongs to $\mathrm{Clo}(A;m)$, where $m$ is the term operation from \Cref{compositionactslike}. The proof is by induction on the complexity of the term $t$. If $t(x,y)$ is a projection or is equal to some term $m(z_1,z_2,z_3)$ where $\{z_1,z_2,z_3\}=\{x,y\}$, then we are done. If a term $t$ has higher complexity, then $t(x,y)=m(t_1(x,y)$, $t_2(x,y), t_3(x,y))$, where $t_1, t_2, t_3$ are binary terms. Based on the induction hypothesis, each of the terms $t_1,t_2,t_3$ has a dominant coordinate. Without loss of generality, we can assume that $t_1$ and $t_2$ have a dominant first coordinate, and then it is obvious that the term $t$ must have a dominant first coordinate.
\end{proof}

When considering a binary term $t(x,y)$, we can assume that the first coordinate is dominant because if it is the second, we can take the term $t'(x,y):= t(y,x)$ instead.

By \Cref{center}, every simple nonabelian minimal Taylor algebra ${\m A}=\mathrm{Sg}\{a,b\}$ has a proper $3$-absorbing subuniverse which contains one of the generators $a$ and $b$. In the next lemma, we assume without loss of generality that element $a$ is in some proper $3$-absorbing subuniverse. 

\begin{lem}\label{lemma1}
Let ${\m A}$ be a simple nonabelian minimal Taylor algebra of size at least $3$ generated by two distinct elements $a,b\in A$ and $a\in C_a\trianglelefteq_3 {\m A}$.
\begin{enumerate}
\item[a)] If $R_{ab}$ is the graph of automorphism $\varphi$ which swaps $a$ and $b$, and $g$ is a cyclic term of algebra ${\m A}$ of arity $n,$ then there exists a cyclic term $g'$ of the same arity such that 
$$g'(\{a,b\}^{n}\setminus\{(a,a,\dotso, a), (b,b.\dotso, b)\})\subseteq C_a\cap C_b,$$
where $C_b$ is the image of $C_a$ under $\varphi$, thus $b\in C_b\trianglelefteq_3 {\m A}$.
\item[b)] If $R_{ab}$ is linked and $g$ is a cyclic term of algebra ${\m A}$ of arity $n$, then there exists a cyclic term $g'$ of the same arity such that  
$$g'(\{a,b\}^{n}\setminus\{(b,b,\dotso, b)\})\subseteq C_a.$$
\end{enumerate}
\end{lem} 

\begin{proof}

First we shall prove part $a)$. Notice that $C_a\cap C_b\not= \emptyset$ holds, otherwise $\m a$ has $2$-majority quotient, which is impossible as it is simple and has more than two elements. Of course, $C_a\cap C_b$ is a 3-absorbing subuniverse. As $\varphi$ is an automorphism such that $\varphi (a)=b$ and $\varphi (b)=a$, then there exists a binary term $t$ such that $t(a,b), t(b,a)\in C_a\cap C_b$. Assume that the first coordinate is dominant in $t$. We define a new cyclic term $g'$ of the same arity as $g$ as follows
\begin{align*}
g'(x_1,\dotso, x_n) := g(&t(\dotso t(t(x_1,x_2),x_3),\dotso x_n), t(\dotso t(t(x_2,x_3),x_4),\dotso x_1),\dotso ,\\
& t(\dotso t(t(x_n,x_1),x_2),\dotso x_{n-1})).
\end{align*}
If $x_1,\dots,x_n\in \{a,b\}$ and not all $x_i$'s are equal, then $t(a,b)$ or $t(b,a)$ appears in (a simplification of) the expression $t(\dotso t(t(x_j,x_{j+1}),x_{j+2}),\dotso x_{j-1})$ for all $1\leqslant j\leqslant n$. Since the first coordinate is dominant and $t(a,b), t(b,a)\in C_a\cap C_b$, we have $$t(\dotso t(t(x_j,x_{j+1}),x_{j+2}),\dotso x_{j-1})\in C_a\cap C_b$$ for all $1\leqslant j\leqslant n$, and therefore $g'(x_1,\dotso ,x_n)\in C_a\cap C_b$.

In part $b)$, \Cref{looplemma} is useful. Take $R_{a,b}$ as $R$ of \Cref{looplemma}, so we get $c=t(a,b)=t(b,a)$, for some binary term $t$. The same composition of $g$ and $t$ as in part $a)$ gives us a cyclic term $g'$ which satisfies the lemma's statement.
\end{proof}

The next lemma will be used multiple times in the proof of the main theorem. The proof of the lemma relies on the following consequence of cyclic operations.

\begin{lem}\label{3elem}
Let ${\m A}$ be a minimal Taylor algebra of size at least $3$ generated by two distinct elements $a,b\in A$. Let $C\trianglelefteq_3 {\m A}$ and let $a,c\in C$ such that $c=t(a,b)=t(b,a)$ for some binary term $t$. If $\{a,c\}$ and $\{b,c\}$ are subuniverses, then $\mathrm{Sg}\{a,b\}=\{a,b,c\}$.
\end{lem}

\begin{proof}
First, notice that $c\not= a$, otherwise $\mathrm{Sg}\{a,b\}$ is a two-element semilattice, which contradicts the assumption on the size of $\m a$. The ternary term operation $m$ from \Cref{compositionactslike} generates the whole clone $\mathrm{Clo}({\m A})$, so we may assume that $m$ is a basic operation of $\m a$. Since the two-element subuniverses must be edges of one of the three types (and in the affine case that edge would be $\mathbb{Z}_2^{\mathrm{aff}}$), Theorem~\ref{compositionactslike} implies that on the subuniverses $\{a,c\}$ and $\{b,c\}$ the operation $m$ acts like a cyclic ternary operation. It follows that $(\{b,c\}; m)$ is a two-element semilattice or two-element majority algebra; it cannot be an affine algebra because $m(b,c,c)\in C$ as $m$ witnesses all 3-absorptions. Assume that the first coordinate is dominant in $t$. We define a new ternary term $m'$ by
$$m'(x,y,z) = m(t(t(x,y),z), t(t(y,z),x), t(t(z,x),y)),$$
 which shall be cyclic on $\{a,b,c\}$.
It is easy to verify that $m'$ is cyclic on subuniverses $\{a,c\}$ and $\{b,c\}$. Further,
\begin{align*}
m'(a,a,b)&=m(t(a,b), t(c,a), t(c,a))=m(c, t(c,a), t(c,a))\in \{a,c\},\\
m'(b,b,a)&=m(t(b,a), t(c,b), t(c,b))=m(c,c,c)=c,
\end{align*}
so it is also cyclic on $\{a,b\}$ as $m$ is cyclic on $\{a,c\}$.
Note that $t(t(b,c),a)$ is equal to $a$ or $c$, so each of the following elements:
\begin{align*}
    m'(a,b,c)&=m(t(c,c), t(t(b,c),a), t(t(c,a),b))=m(c, t(t(b,c),a), c),\\
    m'(b,c,a)&=m(t(t(b,c),a),t(t(c,a),b), t(c,c))=m(t(t(b,c),a), c, c),\\
    m'(c,a,b)&=m(t(t(c,a),b), t(c,c), t(t(b,c),a))=m(c,c,t(t(b,c),a)),\\
    m'(a,c,b)&=m(t(t(a,c),b), t(t(c,b),a), t(c,c))=m(c, t(c,a), c),\\
    m'(c,b,a)&=m(t(t(c,b),a), t(c,c), t(t(a,c),b))=m(t(c,a),c,c),\\
    m'(b,a,c)&=m(t(c,c), t(t(a,c),b), t(t(c,b),a))=m(c,c, t(c,a)),
\end{align*}
is in the set $\{a,c\}$. Therefore, the set $\{a,b,c\}$ is closed under the cyclic term $m'$, thus, by Proposition~\ref{subuniverse}, $\{a,b,c\}$ is a subuniverse of ${\m A}$ which implies $\mathrm{Sg}\{a,b\}=\{a,b,c\}.$
\end{proof} 

\begin{lem}\label{Aljahomework}
Algebra ${\m T}_i^{\n}$, $i\in \{1,2,3,4\}$, has a unique ternary cyclic term.
\end{lem}

\begin{proof} Suppose that algebra ${\m T}_i^{\n}$, $i\in \{1,2,3,4\}$, has universe $\{0,1,2\}$ and subuniverses $\{0,1\}$ and $\{0,2\}$ of the same types as corresponding subalgebras in the Table~\ref{table:1}. Let $q$ be a ternary cyclic term of algebra ${\m T}_1^{\n}$. Obviously, on subuniverses $\{0,1\}$ and $\{0,2\}$, the term $q$ acts like the term $t$ from the Table~\ref{table:1}.

Algebra ${\m T}_1^{\n}$ has two congruences:
\begin{itemize}
\item $\alpha=\{\{0,1\},\{2\}\}$, such that ${\m T}_1^{\n}/\alpha$ is term equivalent to two-element semilattice with absorbing element $\{0,1\}$.
\item $\beta=\{\{0,2\},\{1\}\}$, such that ${\m T}_1^{\n}/\beta$ is term equivalent to two-element majority algebra.
\end{itemize}
These congruences ensure that $q(1,1,2)=1$, $q(1,2,2)=0$ and $q(0,1,2)=q(0,2,1)=0$ hold. Hence, the term $q$ acts in the same way as the term $t$ from Table~\ref{table:1}.

Proof in the case algebra ${\m T}_2^{\n}$ is analogous, but then the quotient ${\m T}_2^{\n} /\beta$ is term equivalent to a two-element affine algebra.

Algebra ${\m T}_3^{\n}$ has congruence $\alpha=\{\{0,1\},\{2\}\}$, such that ${\m T}_3^{\n}/\alpha$ is term equivalent to two-element affine algebra. Also, subalgebra $\{0,2\}$ ternary absorbs ${\m T}_3^{\n}$. These facts ensure that $q(1,1,2)=2, q(1,2,2)=0,$ and $q(0,1,2)=q(0,2,1)=2$, thus, the term $q$ acts in the same way as the term $t$ from Table~\ref{table:1}.

Lastly, the algebra ${\m T}_4^{\n}$ is term equivalent to a two-generated three-element semilattice. Therefore, it necessarily has a unique ternary cyclic term.
\end{proof}

\subsection{The main theorem}
\begin{thm}\label{main-factor}
Let ${\m A}$ be a nonabelian minimal Taylor algebra on a domain of size four generated by two distinct elements $a,b\in A$. Then $\m a$ is not simple.
\end{thm}

\begin{proof} For the sake of contradiction, assume that $\m a$ is simple. Any non-affine weak edge has only two classes, and if there were an affine weak edge $\{a,b\}$, the simplicity of $\m a=\mathrm{Sg}\{a,b\}$ would imply that the witnessing congruence is the diagonal and $\m a$ is Abelian. Hence, there is no type of edge between $a$ and $b$. Let $A=\{a,b,c,d\}$ and let $m$ denote the ternary term operation from \Cref{compositionactslike}, which generates the clone $\mathrm{Clo}(\m a)$. We may assume that $m$ is the basic operation of $\m a$. Recall the relation $R_{ab}$, which is either a graph of an automorphism that swaps $a$ and $b$ or it is linked. Based on that, we naturally distinguish two cases.

\begin{enumerate}[wide, labelwidth=!,itemindent=!,labelindent=0pt]
\item[{\sc{\textbf{Case 1.}}}]  $R_{ab}$ is a graph of an automorphism $\varphi$ of algebra ${\m A}$ such that $\varphi (a)=b$ and $\varphi (b)=a$.

First, let us recall 3-absorbing subuniverses $C_a$ and $C_b$ from Lemma~\ref{lemma1}, thus $a\in C_a$ and $b\in C_b$. By $C$ we denote the intersection $ C_a\cap C_b$, which is also a non-empty 3-absorbing subuniverse according to the proof of the same lemma. Since any element of $A=\mathrm{Sg}\{a,b\}$ is equal to $t(a,b)$ for some binary term $t$, and thus must be in $C_a$ or $C_b$, depending on which coordinate of $t$ is dominant. Thus $A=C_a\cup C_b$. Since $C$ is non-empty and we are in the automorphism case, then $c,d\in C$, that is  $C=\{c,d\}$. 
\begin{enumerate}[wide, labelindent=0pt]
    \item[{\sc{\textbf{Subcase 1a.}}}] $\varphi (c)=c$ and $\varphi (d)=d$. 
    
    Hence, $(c,c),(d,d)\in R_{ab}$, so there exist binary terms $t_1$ and $t_2$ such that $t_1(a,b)=t_1(b,a)=c$ and $t_2(a,b)=t_2(b,a)=d$.

By examining three-element minimal Taylor algebras listed in the subsection \ref{s3elem}, we see that at least one of the sets $\{a,c\}$ and $\{a,d\}$ is a subuniverse. Let $\{a,c\}$ be a subuniverse. Hence $\{b,c\}=\varphi(\{a,c\})$ is also a subuniverse. Since $t_1(a,b)=t_1(b,a)=c$, we have $A=\mathrm{Sg}\{a,b\}=\{a,b,c\}$ by Lemma~\ref{3elem}, a contradiction.

    \item[{\sc{\textbf{Subcase 1b.}}}] $\varphi (c)=d$ and $\varphi (d)=c$.
    

It is already assumed that there is no weak edge between $a$ and $b$, so in view of Theorem~\ref{connectthm}, some weak edge must connect $a$ with $c$ or $d$. We claim that one such edge must be either affine or an outgoing semilattice edge, otherwise by Theorem~\ref{singletoncenter}, $\{a\}\trianglelefteq_3 \m a$. Then $\{b\}\trianglelefteq_3 \m a$ follows because of the automorphism, which together with the fourth item of Theorem~\ref{compositionactslike} means that the ternary term $m$ from that theorem acts on $\{a,b\}$ as the majority. This implies $A=\mathrm{Sg}\{a,b\}=\{a,b\}$, a contradiction. 

Consider the subalgebra $\mathbf{C}_a$ on domain $C_a=\{a,c,d\}$.

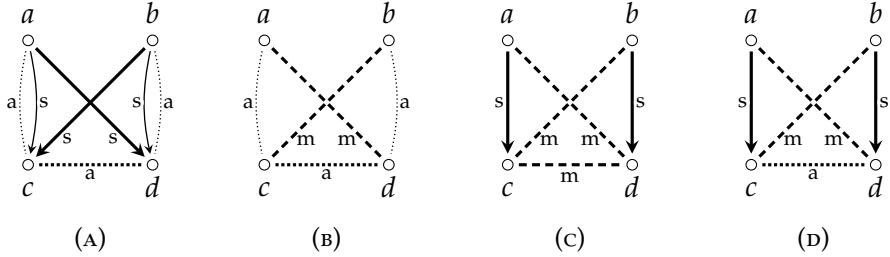
\begin{figure}[h]
	\begin{subfigure}{0.24\textwidth}
		\centering
		\begin{tikzpicture}
			\tikzstyle{every node}=[draw,circle,fill=white,minimum size=4pt,inner sep=0pt]
			
			\node (a) [label=above:\strut{$a$}] {};
			\node (b) [right=1.5cm of a,label=above:\strut{$b$}] {};
			\node (c) [below=1.5cm of a,label=below:\strut{$c$}] {};
			\node (d) [right=1.5cm of c,label=below:\strut{$d$}] {};
			
			
			\path[-,line width=0.5pt,densely dotted, shorten <=0.7mm, shorten >=0.7mm]
			(b) edge [bend left=12] node[draw=none,midway,right] {\scriptsize{a}} (d)
			(a) edge [bend right=12] node[draw=none,midway,left] {\scriptsize{a}} (c);
			\path[-,line width=1.2pt,densely dotted, shorten <=0.7mm, shorten >=0.7mm]
			(c) edge node[draw=none,midway,below] {\scriptsize{a}} (d);
			\path[->,>=stealth,line width=1.2pt, shorten <=0.7mm, shorten >=0.7mm]
			(a) edge [line width=0.5pt] [bend left=12] node[draw=none,midway,right] {\scriptsize{s}} (c)
			(b) edge node[draw=none,near end,below right] {\scriptsize{s}} (c)
			(b) edge [line width=0.5pt] [bend right=12] node[draw=none,midway,left] {\scriptsize{s}} (d)
			(a) edge node[draw=none,near end,below left] {\scriptsize{s}} (d);
			
		\end{tikzpicture}
		\caption{}
		\label{subcase1b1prva}
	\end{subfigure}%
	\hfill
	\begin{subfigure}{0.24\textwidth}
		\centering
		\begin{tikzpicture}
			\tikzstyle{every node}=[draw,circle,fill=white,minimum size=4pt,inner sep=0pt]
			
			\node (a) [label=above:\strut{$a$}] {};
			\node (b) [right=1.5cm of a,label=above:\strut{$b$}] {};
			\node (c) [below=1.5cm of a,label=below:\strut{$c$}] {};
			\node (d) [right=1.5cm of c,label=below:\strut{$d$}] {};
			
			
			\path[-,line width=0.5pt,densely dotted, shorten <=0.7mm, shorten >=0.7mm]
			(b) edge [bend left=12] node[draw=none,midway,right] {\scriptsize{a}} (d)
			(a) edge [bend right=12] node[draw=none,midway,left] {\scriptsize{a}} (c);
			\path[-,line width=1.2pt,densely dotted, shorten <=0.7mm, shorten >=0.7mm]
			(c) edge node[draw=none,midway,below] {\scriptsize{a}} (d);
			\path[-,line width=1.2pt,densely dashed, shorten <=0.7mm, shorten >=0.7mm]
			(b) edge node[draw=none,near end,below right] {\scriptsize{m}} (c)
			(a) edge node[draw=none,near end,below left] {\scriptsize{m}} (d);
		\end{tikzpicture}
		\caption{}
            \label{subcase1b1druga}
	\end{subfigure}
	\hfill
	\begin{subfigure}{0.24\textwidth}
		\centering
		\begin{tikzpicture}
			\tikzstyle{every node}=[draw,circle,fill=white,minimum size=4pt,inner sep=0pt]
			
			\node (a) [label=above:\strut{$a$}] {};
			\node (b) [right=1.5cm of a,label=above:\strut{$b$}] {};
			\node (c) [below=1.5cm of a,label=below:\strut{$c$}] {};
			\node (d) [right=1.5cm of c,label=below:\strut{$d$}] {};
			
			\path[->,line width=1.2pt,>=stealth, shorten <=0.7mm, shorten >=0.7mm]
                (a) edge node[draw=none,midway,left] {\scriptsize{s}} (c)
                (b) edge node[draw=none,midway,right] {\scriptsize{s}} (d);
			
			\path[-,line width=1.2pt,densely dashed, shorten <=0.7mm, shorten >=0.7mm]
			(c) edge node[draw=none,midway,below] {\scriptsize{m}} (d);
			\path[-,line width=1.2pt,densely dashed, shorten <=0.7mm, shorten >=0.7mm]
			(b) edge node[draw=none,near end,below right] {\scriptsize{m}} (c)
			(a) edge node[draw=none,near end,below left] {\scriptsize{m}} (d);
		\end{tikzpicture}
		\caption{}
            \label{subcase1b3prva}
	\end{subfigure}
	\hfill
	\begin{subfigure}{0.24\textwidth}
		\centering
		\begin{tikzpicture}
			\tikzstyle{every node}=[draw,circle,fill=white,minimum size=4pt,inner sep=0pt]
			
			\node (a) [label=above:\strut{$a$}] {};
			\node (b) [right=1.5cm of a,label=above:\strut{$b$}] {};
			\node (c) [below=1.5cm of a,label=below:\strut{$c$}] {};
			\node (d) [right=1.5cm of c,label=below:\strut{$d$}] {};
			
			\path[->,line width=1.2pt,>=stealth, shorten <=0.7mm, shorten >=0.7mm]
                (a) edge node[draw=none,midway,left] {\scriptsize{s}} (c)
                (b) edge node[draw=none,midway,right] {\scriptsize{s}} (d);
			
			\path[-,line width=1.2pt,densely dotted, shorten <=0.7mm, shorten >=0.7mm]
			(c) edge node[draw=none,midway,below] {\scriptsize{a}} (d);
			\path[-,line width=1.2pt,densely dashed, shorten <=0.7mm, shorten >=0.7mm]
			(b) edge node[draw=none,near end,below right] {\scriptsize{m}} (c)
			(a) edge node[draw=none,near end,below left] {\scriptsize{m}} (d);
		\end{tikzpicture}
		\caption{}
            \label{subcase1b3druga}
	\end{subfigure}
    \caption{\sc{Subcase 1b}}
\end{figure}

\begin{enumerate}[wide, labelwidth=!,labelindent=0pt,itemindent=!]
    \item[{\sc Subsubcase 1b1}.]
    There is a weak affine edge going from element $a$. 
     
Without loss of generality, we can assume that $\{a,c\}$ is a weak affine edge. The set $C=\{c,d\}$ $3$-absorbs the algebra ${\m A}$, hence $m(a,c,c)\in \{c,d\}$. If $\{a,c\}$ were a two-element subalgebra, it would contradict the assumption that this set is an affine edge. Therefore, $\mathbf{C}_a$ is not conservative and $C_a=\mathrm{Sg}\{a,c\}=\{a,c,d\}$. Since $\mathbf{T}_5^{\n}$ has no 3-absorbing subsets, only two possibilities are left: ${\m C}_a\cong {\m T}_2^{\n}$ or ${\m C}_a\cong {\m T}_3^{\n}$. 

If ${\m C}_a\cong {\m T}_2^{\n}$ (see \Cref{subcase1b1prva}), then $C$ is a $2$-absorbing set of ${\m C}_a$. Since the automorphism $\varphi$ maps ${\m C}_a$ to ${\m C}_b$, it follows that $C\trianglelefteq_2{\m C}_b$, thus $C\trianglelefteq_2{\m A}$. From the first part of Lemma~\ref{lemma1} it is easy to derive that $\{\{a\},\{b\},\{c,d\}\}$ is a factor of ${\m A}$ (term-equivalent to semilattice ${\m T}_4^{\n}$), a contradiction. 

Consider the case when ${\m C}_a\cong {\m T}_3^{\n}$ (where $a\mapsto 1$, $c\mapsto 2$, $d\mapsto 0$), and also ${\m C}_b\cong {\m T}_3^{\n}$ (see \Cref{subcase1b1druga}). The algebra ${\m A}$ has a ternary cyclic term $g$ by \Cref{cyclic}. Restricted to ${\m C}_a$ and ${\m C}_b$, the term $g$ acts as the corresponding operation from the \Cref{table:1} because ${\m T}_3^{\n}$ has an affine factor $\{\{0,1\},\{2\}\}$ and the ternary absorbing set $\{0,2\}$. Therefore, to completely determine $g$, we only need to find the values $g(a,a,b)$, $g(a,b,c)$, and $g(a,b,d)$, which are all in $C$.
Let $g'$ be a ternary cyclic term defined by $$g'(x,y,z) = g(g(x,x,y),g(y,y,z),g(z,z,x)).$$

\noindent Now, we have $g'(a,a,b)=c$ while $g'$ acts as $g$ on $C_a$ and $C_b$. Again, we define a new ternary cyclic term $g''$ by 
$$g''(x,y,z) = g'(g'(x,x,y),g'(y,y,z),g'(z,z,x)).$$
\noindent Now, we get
$$g''(a,b,c)=g'(g'(a,a,b),g'(b,b,c),g'(c,c,a))=g'(c,b,d)=d,$$
$$g''(a,b,d)=g'(g'(a,a,b),g'(b,b,d),g'(d,d,a))=g'(c,d,d)=c,$$
while $g''$ has the same values as $g'$ in all other cases. Hence, the term $g''$ is fully determined and must exist in algebra ${\m A}$ in this case. Furthermore, $g''$ generates $\mathrm{Clo}({\m A})$ since it is a cyclic term, and produces an affine quotient $\{\{a,d\},\{b,c\}\}$, which is a contradiction.

    \item[{\sc Subsubcase 1b2}.] $(a,c)$ and $(a,d)$ are weak semilattice edges.
    
    We have $\{c,d\}\trianglelefteq_2 \m c_a$ in all those cases. Using the automorphism's properties, we obtain $\{c,d\}\trianglelefteq_2\m c_b$, so $\{c,d\}\trianglelefteq_2\m a$. Then there is a congruence $\theta$ of $\m a$ with classes $\{a\}$, $\{b\}$ and $\{c,d\}$, and the quotient $\m a/\theta$ term-equivalent to $\m t_4^{\n}$, which is a contradiction.

{\sc Subsubcases 1b1} and {\sc 1b2} cover all cases where $\m c_a$ is nonconservative, according to the list of all 3-element minimal Taylor algebras and other properties we have uncovered about $\m c_a$, so we are left with

    \item[{\sc Subsubcase 1b3}.] $\{a,c\}$ and $\{a,d\}$ are subuniverses.

    As we have seen, it is not possible for both to be majority algebras, while neither is affine, as that was dealt with in {\sc Subsubcase 1b1}. It is impossible that both of them are two-element semilattices, as $C\trianglelefteq C_a$ implies that both would be going out of $a$, and we dealt with that case in {\sc Subsubcase 1b2}. For that reason, one of the subuniverses must be a two-element semilattice, and the other a two-element majority algebra. Therefore, let $\{a,c\}$ be a two-element semilattice with absorbing element $c$ and let $\{a,d\}$ be a two-element majority algebra. It follows that ${\m C}_a$ is a conservative minimal Taylor algebra, and the subalgebra on universe $C=\{c,d\}$ is either the two-element majority algebra or the two-element affine algebra (since the automorphism $\varphi$ flips $c$ and $d$, it can't be the semilattice).

Let ${\m C}$ be a two-element majority algebra as in \Cref{subcase1b3prva}. From \Cref{table:4} we see that ${\m C}_a\cong {\m T}_4^{\co}$ holds. Now, we can assume that algebra ${\m A}$ has a ternary cyclic term $g$ by \Cref{cyclic}, also we see that ${\m C}_a\cong {\m T}_4^{\co}$ has a majority quotient $\{\{a,c\},\{d\}\}$ and $C$ is a $3$-absorbing set of ${\m C}_a$. Thus, the cyclic operation $g$ acts in the same way as the corresponding cyclic operation from \Cref{table:4} for the algebra ${\m T}_4^{\co}$ $(a\mapsto 1, c\mapsto 0, d\mapsto 2)$. Similar to before, we introduce new cyclic operations
$$g'(x,y,z) := g(g(x,x,y),g(y,y,z),g(z,z,x)),$$
$$g''(x,y,z) := g'(g'(x,x,y),g'(y,y,z),g'(z,z,x)).$$
Since $g(a,a,b)\not= g(b,b,a)$, we get $g'(a,a,b)=g''(a,a,b)=c$, $g'(b,b,a)=g''(b,b,a)=d$ and
$$g''(a,b,c)=g'(g'(a,a,b),g'(b,b,c),g'(c,c,a))=g'(c,b,c)=c,$$
$$g''(a,b,d)=g'(g'(a,a,b),g'(b,b,d),g'(d,d,a))=g'(c,d,d)=d.$$
From automorphism we have $g''(a,d,b)=d$ and $g''(a,c,b)=c$, while  $g''$ has the same values as the term $g$ on $C_a$ and $C_b$. Hence, this algebra ${\m A}$ has a majority quotient $\{\{a,c\},\{b,d\}\}$, a contradiction.

Consider now the case when ${\m C}$ is a two-element affine algebra (\Cref{subcase1b3druga}). According to \Cref{table:4}, ${\m C}_a\cong {\m T}_{10}^{\co}$ must hold. Additionally, we may assume that algebra ${\m A}$ has a ternary cyclic term $g$ by \Cref{cyclic} and $C$ is a $3$-absorbing set of ${\m A}$. Unlike in the previous case, now we cannot assume that $g$ acts in the same way as the corresponding term from \Cref{table:4} because ${\m T}_{10}^{\co}$ has at least two ternary cyclic terms (one can get $g'$ from $g$ in the usual way and $g'(0,1,2)=g'(0,2,1)=0\neq 2=g(0,1,2)=g(0,2,1))$. Obviously, on two-element subalgebras, all ternary cyclic terms act similarly. Suppose first that $g(a,a,b)=d$ and $g(b,b,a)=c$ hold. Let $g'$ be a ternary cyclic term defined by
$$g'(x,y,z) = g(g(x,x,g(x,y,z)),g(y,y,g(y,z,x)),g(z,z,g(z,x,y))).$$ We have $$g'(a,a,b)=g(g(a,a,d),g(a,a,d),g(b,b,d))=g(a,a,d)=a$$ and $g'(b,b,a)=b$. Since $g'$ is cyclic, we get $\mathrm{Sg}\{a,b\}=\{a,b\}$, a contradiction. If $g(a,a,b)=c$ and $g(b,b,a)=d$, we define a new ternary cyclic term $g''$ by
\begin{align*}
g''(x,y,z) = g(&g(x,g(y,y,z),g(y,y,z)),g(y,g(z,z,x),g(z,z,x)),\\ 
& g(z,g(x,x,y),g(x,x,y))).
\end{align*}
Now, we get $g''(a,a,b)=d$ and $g''(b,b,a)=c$, and we can apply the previous reasoning on $g''$.
\end{enumerate}

\end{enumerate}

\item[{\sc{\textbf{Case 2.}}}] $R_{ab}$ is linked.

We shall use the Loop Lemma in this case. In a similar way as in Lemma~\ref{lemma1} we get $R_{a,b}\cap \Delta_A\neq \emptyset$. It is impossible that $(a,a)\in R_{a,b}$ or $(b,b)\in R_{a,b}$ holds; otherwise, if, for example, $(a,a)\in R_{a,b}$, we have $a=s(a,b)=s(b,a)$ for some binary term $s$, so $(b,a)$ is a semilattice edge, a contradiction. Therefore, let us suppose, without loss of generality, that $c=t(a,b)=t(b,a)$ holds for some binary term $t$. Recall $C_a$ from Lemma~\ref{3elem}, the $3$-absorbing set of ${\m A}$ such that $a, c\in C_a$ and $b\notin C_a$. We prove this case by discussing the number of elements in $C_a$.

\begin{enumerate}[wide, labelwidth=!,labelindent=0pt,itemindent=!]
    \item[{\sc{\textbf{Subcase 2a.}}}] $C_a=\{a,c\}.$

    Observe that $\{a,c\}$ must be an affine or $(a,c)$ a semilattice edge in this case. Indeed, if $\{a,c\}$ were a majority edge or $(c,a)$ a semilattice edge, we would have $\{a\}\trianglelefteq_3 C_a \trianglelefteq_3 \m a$. Thus by Proposition~\ref{smedge}, $(b,a)$ would be a semilattice edge or a majority edge, a contradiction. From Proposition~\ref{smedge} it follows that $(b,c)$ must be a semilattice edge or $\{b,c\}$ a weak majority edge. The first case would imply that both $\{a,c\}$ and $\{b,c\}$ are subuniverses, which by Lemma \ref{3elem} leads to $A=\{a,b,c\}$, a contradiction. Hence $\{b,c\}$ is a weak majority edge. Moreover, $a\notin \mathrm{Sg}\{b,c\}$, or the congruence witnessing the weak majority edge $\{b,c\}$ would be a nontrivial congruence, and $\mathrm{Sg}\{b,c\}=\{b,c,d\}$.

\begin{enumerate}[wide, labelwidth=!,labelindent=0pt,itemindent=!]
    \item[\sc{Subsubcase 2a1}.] $\{a,c\}$ is an affine edge.


From $\mathrm{Sg}\{b,c\}=\{b,c,d\}$ and Proposition~\ref{smedge} follows $m(b,c,c)=c$. Since $(a,b), (c,c)\in R_{ab}$, we have that
$$
\begin{bmatrix}
a \\
c
\end{bmatrix}=
m\left(
\begin{bmatrix}
a \\
b
\end{bmatrix}, 
\begin{bmatrix}
c \\
c
\end{bmatrix},
\begin{bmatrix}
c \\
c
\end{bmatrix}
\right)\in R_{ab},
$$
and by symmetry $(c,a)\in R_{ab}$. Therefore,
$$
\begin{bmatrix}
a \\
a
\end{bmatrix}=
m\left(
\begin{bmatrix}
a \\
c
\end{bmatrix}, 
\begin{bmatrix}
c \\
c
\end{bmatrix},
\begin{bmatrix}
c \\
a
\end{bmatrix}
\right)\in R_{ab},
$$
which is a contradiction.

\item[\sc{Subsubcase 2a2}.] $(a,c)$ is a semilattice edge.

By examining the nonconservative three-element algebras, we get the only one possibility -- $\mathrm{Sg}\{b,c\}\cong {\m T}_1^{\n}$ (where $c\mapsto 1, b\mapsto 2, d\mapsto 0$). However, in that case, it can be easily checked that $\{a,c,d\}$ binary absorbs $\m a$ by $t'$, where $t'(x,y):=t(t(x,y),x)$ if we assume that $t$ has the first coordinate as the dominant one. Hence, $(b,a)$ is a weak semilattice edge, a contradiction.
\end{enumerate}

    \item[{\sc{\textbf{Subcase 2b.}}}] $C_a=\{a,c,d\}$.

     A weak affine edge or a semilattice edge is going from $b$ to $c$ or $d$; otherwise, the set $\{b\}$ would be a $3$-absorbing set, so there would exist a weak edge between $a$ and $b$ by Proposition~\ref{smedge}. 

\begin{enumerate}[wide, labelwidth=!,labelindent=0pt,itemindent=!]
    \item[{\sc Subsubcase 2b1}.] $\{b,c\}$ or $\{b,d\}$ is a weak affine edge.

If $\{b,c\}$ is a weak affine edge, then an analogous argument as at the start of {\sc Subsubcase} 1{\sc{b}}1 proves $\mathrm{Sg}\{b,c\}=\{b,c,d\}$. Similarly, if $\{b,d\}$ is a weak affine edge, then $\mathrm{Sg}\{b,d\}=\{b,c,d\}$. In both cases, we obtain that $\{c,d\}$ must be an affine subalgebra, also there are only two possibilities for $\{b,c,d\}$; it is either ${\m T}_2^{\n}$ or ${\m T}_3^{\n}$. If $\{b,c,d\}\cong{\m T}_2^{\n}$, then $\{c,d\}\trianglelefteq_2\{b,c,d\}$, consequently $C_a\trianglelefteq_2\m a$, and we would have that $(b,a)$ must be a weak semilattice edge, a contradiction. So we may assume that $\{b,c,d\}\cong{\m T}_3^{\n}$ holds. \\

\begin{claim}
$(d,d)\in R_{ab}$. 
\end{claim}

\begin{proof}
From $R_{ab}$ being a symmetric and linked relation, we know that $c$ is connected to $c$ and at least one other element, thus $\{c\}\subsetneq R_{ab}[c]$. There are three cases:
\begin{itemize}[leftmargin=*]
\item $d\in R_{ab}[c].$ Then
$$
\begin{bmatrix}
d \\
d
\end{bmatrix}=
m\left(
\begin{bmatrix}
c \\
d
\end{bmatrix}, 
\begin{bmatrix}
c \\
c
\end{bmatrix},
\begin{bmatrix}
d \\
c
\end{bmatrix}
\right)\in R_{ab}.
$$

\item $R_{ab}[c]=\{a,c\}$, which must then be a subuniverse of $\m a$. Since $(a,b),(c,a)\in R_{ab}$, we have that $R_{a,b}[\{a,c\}]= A$, so $(a,d)\in R_{ab}$ follows. If $m(a,c,c)=c$, or $m(c,a,a)=c$, we obtain $(c,d)\in R_{ab}$ (contradicting $R_{ab}[c]=\{a,c\}$):
$$
\begin{bmatrix}
c \\
d
\end{bmatrix}=
m\left(
\begin{bmatrix}
a \\
d
\end{bmatrix}, 
\begin{bmatrix}
c \\
c
\end{bmatrix},
\begin{bmatrix}
c \\
c
\end{bmatrix}
\right) \text{, or} 
\begin{bmatrix}
c \\
d
\end{bmatrix}=
m\left(
\begin{bmatrix}
c \\
c
\end{bmatrix}, 
\begin{bmatrix}
a \\
c
\end{bmatrix},
\begin{bmatrix}
a \\
d
\end{bmatrix}
\right)\text{, respectively.}
$$
On the other hand, if $a=m(a,c,c)=m(c,a,a)$, then $\{a,c\}$ is a two-element semilattice, and we have that
$$
\begin{bmatrix}
a \\
a
\end{bmatrix}=
m\left(
\begin{bmatrix}
a \\
c
\end{bmatrix}, 
\begin{bmatrix}
a \\
c
\end{bmatrix},
\begin{bmatrix}
c \\
a
\end{bmatrix}
\right)\in R_{ab}.$$
However, that is impossible because then $\{b,a\}$ is a two-element semilattice. 

\item $R_{ab}[c]=\{b,c\}.$ We handle this case similarly to the previous one. $R_{ab}[c]=\{b,c\}$ implies that $\{b,c\}$ is a subuniverse of $\m a$. As $(b,a),(c,b)\in R_{ab}$, we have that $R_{ab}[\{b,c\}]=A$. As $(c,d)\notin R_{ab}$, we obtain $(b,d)\in R_{ab}.$ Since $\{a,c,d\}\trianglelefteq_3 A$, by intersecting with the subuniverse $\{b,c,d\}$ we obtain $\{c,d\}\trianglelefteq_3\{b,c,d\}$. Hence $m(b,c,c)\in \{c,d\}$, and since $\{b,c\}$ is a subuniverse, then $m(b,c,c)=c$. The contradiction $(c,d)\in R_{ab}$ follows from
$$
\begin{bmatrix}
c \\
d
\end{bmatrix}=
m\left(
\begin{bmatrix}
b \\
d
\end{bmatrix}, 
\begin{bmatrix}
c \\
c
\end{bmatrix},
\begin{bmatrix}
c \\
c
\end{bmatrix}
\right)\in R_{ab}.
$$
\end{itemize} 
\end{proof}


One of the sets $\{b,c\}$ and $\{b,d\}$ is a subuniverse, and since we just proved $(d,d)\in R_{ab}$, we know all the properties of $c$ (that $t(a,b)=t(b,a)=c$) also hold for $d$ and some other binary term. So, without loss of generality, we can assume that $\{b,c\}$ is a subuniverse. From Lemma~\ref{3elem} we know that $\mathrm{Sg}\{a,c\}=\{a,c\}$ is impossible, so we have $\mathrm{Sg}\{a,c\}=\{a,c,d\}=C_a$. Since $\{c,d\}$ is a two-element affine subalgebra, then the only possibilities are that $\m C_a\cong {\m T}_2^{\n}$ or $\m C_a\cong {\m T}_3^{\n}$.

 If $\m C_a\cong{\m T}_2^{\n}$, the subuniverse $\{b,c,d\}$ is $2$-absorbing as $t$ witnesses the absorption, so $(a,b)$ is a weak semilattice edge, which is a contradiction. Let us now assume $\m C_a\cong {\m T}_3^{\n}$ $(c\mapsto 2, a\mapsto 1, d\mapsto 0)$. We shall prove that ${\m A}$ has an affine quotient $\{\{a,d\},\{b,c\}\}$. By \Cref{cyclic}, the algebra ${\m A}$ has a ternary cyclic term $g$. We know that both $\{a,c,d\}$ and $\{b,c,d\}$ are isomorphic to ${\m T}_3^{\n}$, which has an affine quotient and a ternary absorbing subuniverse such that every ternary cyclic term must act in the same way. Define a new ternary cyclic term $g'$
$$g'(x,y,z) := g(t(x,y),t(y,z),t(z,x)).$$
We get $g'(a,a,b)=g(a,c,c)=d$ and $g'(b,b,a)=g(b,c,c)=c$. Again, define a new ternary cyclic term $g''$
$$g''(x,y,z) := g'(g'(x,x,y),g'(y,y,z),g'(z,z,x)).$$
Now, $g''(a,a,b)=g'(a,d,c)=c$ and $g''(b,b,a)=g'(b,c,d)=d$. Finally, let $g'''$ be a ternary cyclic such that
$$g'''(x,y,z) = g''(g''(x,x,y),g''(y,y,z),g''(z,z,x)).$$
Then we have 
\begin{align*}
g'''(a,a,b)&=g''(g''(a,a,a),g''(a,a,b),g''(b,b,a))=g''(a,c,d)=c,\\
g'''(b,b,a)&=g''(g''(b,b,b),g''(b,b,a),g''(a,a,b))=g''(b,d,c)=d,\\
g'''(a,b,c)&=g''(g''(a,a,b),g''(b,b,c),g''(c,c,a))=g''(c,b,d)=d,\\
g'''(a,c,b)&=g''(g''(a,a,c),g''(c,c,b),g''(b,b,a))=g''(c,c,d)=d,\\
g'''(a,b,d)&=g''(g''(a,a,b),g''(b,b,d),g''(d,d,a))=g''(c,d,d)=c,\\
g'''(a,d,b)&=g''(g''(a,a,d),g''(d,d,b),g''(b,b,a))=g''(a,c,d)=c.
\end{align*}

It follows that $\{\{a,d\},\{b,c\}\}$ is an affine quotient, which is in contradiction with $c=t(a,b)=t(b,a)$.

\item[\sc{Subsubcase 2b2}.] There is a semilattice edge going from $b$ to $C_a$.

If $(b,c)$ and $(b,d)$ are weak semilattice edges, then $C_a$ is a $2$-absorbing set of ${\m A}$ and $(b,a)$ is a weak semilattice edge. Additionally, if we avoid {\sc Subsubcase 2b1}, by Proposition~\ref{smedge} we have that exactly one of $(b,c)$ and $(b,d)$ is a semilattice edge.

We shall prove that $\{b,c\}$ and $\{b,d\}$ are both subalgebras --- one of them a semilattice and other one a majority subalgebra. Suppose that $(b,c)$ is a semilattice edge and that $(b,d)$ is not. We claim that $\{b,d\}$ is a two-element majority algebra. Indeed, if $\mathrm{Sg}\{b,d\}=\{b,c,d\}$, since $(b,d)$ is not a weak semilattice edge, then $\{c,d\}$ is not a $2$-absorbing set for $\mathrm{Sg}\{b,d\}$, but by examining nonconservative algebras from Section~\ref{s3elem} one can confirm there is no such a three-element minimal Taylor algebra. Suppose now $a\in \mathrm{Sg}\{b,d\}$, thus $\mathrm{Sg}\{b,d\}=A$. Hence, $r(d,b)=c$ follows, for some binary term $r$. It can be easily checked that, if $t$ has the first coordinate dominant, the binary term $t(t(x,y),r(y,x))$ witnesses absorption $C_a\trianglelefteq_2\m a$, so $(b,a)$ must be a weak semilattice edge, a contradiction. Therefore, $\{b,d\}$ must be a majority subalgebra.

If $\mathrm{Sg}\{a,c\}=\{a,c\}$, by Lemma~\ref{3elem} we have $\mathrm{Sg}\{a,b\}=\{a,b,c\}$, which is impossible. Therefore, $\mathrm{Sg}\{a,c\}=C_a$ holds, so $\m c_a$ is a nonconservative three-element minimal Taylor algebra, and there are five cases:
\begin{itemize}[leftmargin=*]
\item $\m C_a\cong {\m T}_1^{\n}$. Since an affine edge or a semilattice edge must go out from $a$, it must be $a\mapsto 2$, $c\mapsto 1$, $d\mapsto 0$. But now it is easy to check that $t$ witnesses the absorption $\{b,c,d\}\trianglelefteq_2\m a$, so $(a,b)$ is a semilattice edge, a contradiction.

\item $\m C_a\cong {\m T}_2^{\n}$. If the element $a$ corresponds to $2$, we will show that $\{b,c,d\}$ is a $2$-absorbing set of ${\m A}$ (and consequently, $(a,b)$ is an semilattice edge). Note that $\{0,1\}\trianglelefteq_2 \m T_2^{\n}$, so $\{c,d\}\trianglelefteq_2 \{a,c,d\}$ and since $t(a,b)=t(b,a)=c$, then $\{b,c,d\}\trianglelefteq_2\m a$ is witnessed by $t$, a contradiction.

The remaining option is $a\mapsto 1, c\mapsto 2, d\mapsto 0$. But then the set $\{a,d\}$ is a $3$-absorbing set of $C_a$, and since $C_a$ is a $3$-absorbing of ${\m A}$, then $\{a,d\}\trianglelefteq_3{\m A}$. This is impossible since by Lemma~\ref{dominant} either $t(a,b)$ or $t(b,a)$ should be in $\{a,d\}$, and we know that $t(a,b)=t(b,a)=c\notin\{a,d\}$.

\item $\m C_a\cong {\m T}_3^{\n}$. Hence, $d\mapsto 0$ follows in this case. If $a\mapsto 2$ and $c\mapsto 1$, then $\{a,d\}$ is a $3$-absorbing set of $C_a$ and thus $\{a,d\}\trianglelefteq_3{\m A}$. This is a contradiction, just like in the previous case.
 Therefore, we can suppose $a\mapsto 1, c\mapsto 2$. Then $\{c,d\}$ is an affine subalgebra. Note that, from \Cref{table:1}, we conclude that $m(a,c,c)=d$, while $m(b,c,c)=c$ regardless of whether $\{b,c\}$ is the two-element semilattice or the majority algebra. Hence,
$$
\begin{gathered}
\begin{bmatrix}
d \\
c
\end{bmatrix}=
m\left(
\begin{bmatrix}
a \\
b
\end{bmatrix}, 
\begin{bmatrix}
c \\
c
\end{bmatrix},
\begin{bmatrix}
c \\
c
\end{bmatrix}
\right)\in R_{ab}\text{ and }
\begin{bmatrix}
c \\
d
\end{bmatrix}=
m\left(
\begin{bmatrix}
b \\
a
\end{bmatrix}, 
\begin{bmatrix}
c \\
c
\end{bmatrix},
\begin{bmatrix}
c \\
c
\end{bmatrix}
\right)\in R_{ab}\text{, thus }\\
\begin{bmatrix}
d \\
d
\end{bmatrix}=
m\left(
\begin{bmatrix}
d \\
c
\end{bmatrix}, 
\begin{bmatrix}
c \\
c
\end{bmatrix},
\begin{bmatrix}
c \\
d
\end{bmatrix}
\right)\in R_{ab}.
\end{gathered}
$$
Hence $q(a,b)=q(b,a)=d$ for some binary term $q$, and as $\{a,d\}$ and $\{b,d\}$ are both subuniverses, Lemma~\ref{3elem} guarantees that $\mathrm{Sg}\{a,b\}=\{a,b,d\}$, a contradiction.

\item $\m C_a\cong {\m T}_4^{\n}$. Then $\{c,d\}$ 2-absorbs $\{a,c,d\}$ and we prove that $\{b,c,d\}$ $2$-absorbs ${\m A}$ by the same argument as in the case $\m C_a\cong {\m T}_2^{\n}$. So $(a,b)$ is a weak semilattice edge, also a contradiction.

\item $\m C_a\cong {\m T}_5^{\n}(=\mathbb{Z}_3^{\text{aff}})$. On the set $C_a$, $m$ acts like a Mal'cev operation, and $C_a$ is generated by any two distinct elements. Regarding connectivity, $R_{ab}$ connects $c$ to some element other than itself. We distinguish three cases. 

Firstly, if $a\in R_{ab}[\{c\}]$, then
$$
\begin{bmatrix}
a \\
a
\end{bmatrix}=
m\left(
\begin{bmatrix}
c \\
a
\end{bmatrix}, 
\begin{bmatrix}
c \\
c
\end{bmatrix},
\begin{bmatrix}
a \\
c
\end{bmatrix}
\right)\in R_{ab},
$$
but this is impossible because, otherwise, $\{b,a\}$ is a two-element semilattice.

Secondly, if $d\in R_{ab}[\{c\}]$, since $\mathrm{Sg}\{c,d\}=\{a,c,d\}$, then $(c,a)\in R_{ab}$, and we are in the first case, which is proved to be impossible.

Finally, if $R_{ab}[\{c\}]=\{c,b\}$, then since $(a,b),(b,c)\in R_{ab}$, we have $R_{ab}[\{c,b\}]=A$. Now, we obtain that $(b,d)\in R_{ab}$ and one of two pairs 

$$
m\left(
\begin{bmatrix}
b \\
d
\end{bmatrix}, 
\begin{bmatrix}
c \\
b
\end{bmatrix},
\begin{bmatrix}
c \\
b
\end{bmatrix}
\right)
\text{and}\;
m\left(
\begin{bmatrix}
b \\
d
\end{bmatrix}, 
\begin{bmatrix}
b \\
d
\end{bmatrix},
\begin{bmatrix}
c \\
b
\end{bmatrix}
\right)
$$
is equal to $(c,d)$, but this contradicts $R_{ab}[\{c\}]=\{c,b\}$.

\end{itemize}
\end{enumerate}
\end{enumerate}

\end{enumerate}

Since a contradiction was reached in each case, we may conclude that the assumption that $\m a$ is simple is false, and thus the proof is complete.
\end{proof}
Due to Proposition~\ref{mintayloraff}, if $\m a$ is a 2-generated abelian minimal Taylor algebra on a four-element domain, then it has to be term-equivalent to an affine Mal'cev algebra $(A;x-y+z)$ over $\mathbb{Z}_4$ or $\mathbb{Z}_2\times \mathbb{Z}_2$. It is not hard to show that both algebras have a nontrivial quotient. That being said, we reach the following result.
\begin{cor}\label{notsimplecor}
Any $2$-generated minimal Taylor algebra on a domain of size four is not simple.
\end{cor}

\section{Classification of 2-generated minimal Taylor algebras on a domain of size 4}

\subsection{Overall strategy for classification and a summary of results}

We proved that every 2-generated minimal Taylor algebra on a four-element domain is not simple. We shall classify these algebras based on their maximal quotients. From \Cref{s3elem} follows that each minimal Taylor algebra $\m a=\mathrm{Sg}\{a,b\}$ on a four-element set must have at least one of the following quotients: two-element semilattice $\m s$, two-element majority algebra $\m m$, or one of the affine algebras $\mathbb{Z}_2^{\text{aff}}$ and $\mathbb{Z}_3^{\text{aff}}$.

We first suppose that $\m a$ has a 2-element semilattice quotient (Subsection~\ref{s:semi}) and discover seven such minimal Taylor algebras $\T_{4,1}-\T_{4,7}$. Next, we suppose that $\m a$ has a 2-element majority quotient, but no 2-element semilattice quotient (Subsection~\ref{s:maj}) and find minimal Taylor algebras $\T_{4,8}-\T_{4,10}$. When $\m a$ has a quotient isomorphic to  $\mathbb{Z}_2^{\text{aff}}$, but neither of the previous two types of quotients in Subsection~\ref{s:Z2}, we discover four more minimal Taylor algebras $\T_{4,11}-\T_{4,13}$. The remaining case with the quotient $\mathbb{Z}_3^{\text{aff}}$ yields algebras $\T_{4,14}-\T_{4,18}$.

Of those nineteen algebras, these are all subdirect products of smaller algebras: $\T_{4,1}\leq_{sd}\m s\times\T_1^{\n}, \T_{4,2}\leq_{sd}\m s\times\T_2^{\n},\T_{4,3},\T_{4,4}\leq_{sd}\m s\times\T_3^{\n}$ (both subdirectly embed into the same product), $\T_{4,7}\leq_{sd}\m s\times\mathbb{Z}_3^{\text{aff}}, \T_{4,9}\cong\m m\times\mathbb{Z}_2^{\text{aff}}$. The algebras $\T_{4,5},\T _{4,6},\T_{4,8},\T_{4,10},\T_{4,11},\T_{4,12},\T_{4,13},\T_{4,14},\T_{4,15},\T_{4,16},\T_{4,17}$ and $\T_{4,18}$ are subdirectly irreducible. Of the subdirectly irreducible examples, three have appeared in \cite{dreamteam}, namely $\T_{4,10}$, $\T_{4,13}$, and $\T_{4,14}$ were described in Example 5.16, Example 5.17, and Example 5.20 of \cite{dreamteam}, respectively. The remaining nine subdirectly irreducible minimal Taylor algebras, to our knowledge, appear for the first time in print here, so we will describe them more thoroughly in case they become useful later.

\subsection{Auxiliary lemmas for the classification}

The following lemma by Barto and Kozik shall be used on numerous occasions later to prove the existence of a cyclic term. 

\begin{lem}[Lemma 4.19 of \cite{BartoKozik}]\label{bkcyclic}
Let $\m A$ be a finite, idempotent algebra and $\alpha$ be a congruence of $\m A$. If ${\m A}/\alpha$ and every $\alpha$-class in $A$ has a cyclic operation of arity $k$, then so does $\m A$.
\end{lem}

From now on, we may assume that $\m a$ is a minimal Taylor algebra with underlying set $A=\{a,b,c,d\}$ generated by elements $a$ and $b$. As before, according to \Cref{compositionactslike} and \Cref{gen}, $\m a$ contains a ternary operation $m$ which generates the clone $\mathrm{Clo}(\m a)$. Henceforth, we perceive $\m a$ as a minimal Taylor algebra with $m$ as the only basic operation. We begin with a couple of lemmas we will use in the first subsection.

\begin{lem}\label{triangle}
Let $\m a=\mathrm{Sg}\{a,b\}$ be a minimal Taylor algebra such that $a\neq b$ and $B\trianglelefteq_2 {\m A}$, where $B=A\backslash \{a\}$. If $\{a,c\}$ and $\{b,c\}$ are subuniverses for some $c\in B$ and if $t(a,b),t(b,a)\in \{b,c\}$ for some binary term $t$, then $A=\{a,b,c\}$.
\end{lem}
\begin{proof}
The proof is similar to the proof of Lemma~\ref{3elem}.
Take a cyclic operation $g$ of arity $p$. Our assumptions imply that $t(a,x),t(x,a)\in\{b,c\}$, whenever $x\in \{b,c\}$. We define a cyclic operation $g'$ of arity $p$ as follows
\begin{align*}
g'(x_1,\dotso, x_p) = g(&t(\dotso t(t(x_1,x_2),x_3),\dotso x_p), t(\dotso t(t(x_2,x_3),x_4),\dotso x_1),\dotso ,\\
& t(\dotso t(t(x_p,x_1),x_2),\dotso x_{p-1})).
\end{align*}
For the term $g'$, it holds that if $x_i\in \{b,c\},$ for some $1\leqslant i\leqslant p$, then
$$g'(x_1,x_2,\dotso , x_p)\in \{b,c\}.$$

Therefore, the set $\{a,b,c\}$ is closed under the cyclic term $g'$. Furthermore, by Proposition~\ref{subuniverse}\, it is a subuniverse of $\m A$, which implies $\mathrm{Sg}\{a,b\}=\{a,b,c\}$. This concludes the proof.
\end{proof}

The next lemma explores the structure of $\m a$, especially when $B$ is the universe of a conservative algebra.

\begin{lem}\label{triangle2}
Let ${\m A}=\mathrm{Sg}\{a,b\}$ be a minimal Taylor algebra on domain $\{a,b,c,d\}$ and $B=\{b,c,d\}\trianglelefteq_2 {\m A}$. If $\{a,c\}$ and $\{b,c\}$ are subuniverses, then $(b,d)$ is a semilattice edge. If, additionally, $\m b$ is a conservative algebra, then:
\begin{itemize}
\item $(b,d)$, $(a,d)$ and $(b,c)$ are semilattices edges,
\item $\{c,d\}$ is a two-element majority algebra or a two-element affine algebra.
\end{itemize}
\end{lem}
\begin{proof}
Since $c\in\mathrm{Sg}\{a,b\}$, there exists a binary term $t$ such that $c=t(a,b)$. By Lemma~\ref{triangle} we have $t(b,a)\notin \{b,c\}$, so it must be $t(b,a)=d$ because $t(b,a)\in B$. Define a new binary term $t_1$ by $t_1(x,y)=m(t(x,y),x,x)$. Now we obtain
\begin{gather*}
t_1(a,b)=m(t(a,b),a,a)=m(c,a,a)=c,\\
t_1(b,a)=m(t(b,a),b,b)=m(d,b,b).
\end{gather*}
From Lemma~\ref{triangle} we get $t_1(b,a)=m(d,b,b)=d$.
Next, let $t_2(x,y):=m(x,t(x,y),t(x,y))$, so
\begin{gather*}
t_2(a,b)=m(a,t(a,b),t(a,b))=m(a,c,c)=c,\\
t_2(b,a)=m(b,t(b,a),t(b,a))=m(b,d,d).
\end{gather*}
Once again, we deduce $t_2(b,a)=m(b,d,d)=d$ from Lemma~\ref{triangle}.
But now a binary term $t'(x,y):=m(x,y,y)$ acts as a semilattice operation on the set $\{b,d\}$ with an absorbing element $d$. Hence, $(b,d)$ semilattice edge.

Suppose now that $B$ is a subuniverse of a conservative subalgebra.
By the first part of the proof, we see that $(b,d)$ is a semilattice edge. Let $R_{ab}:=\mathrm{Sg}_{\m a^2}\{(a,b),(b,a)\}$. Similarly to before, $t(a,b)=c$ and $t(b,a)=d$ for some binary term $t$, hence $(c,d)\in R_{ab}$.
If $\mathrm{Sg}\{a,d\}$ is not a two-element subalgebra, then $c\in \mathrm{Sg}\{a,d\}$. From $(b,a), (c,d)\in R_{ab}$ follows $(b,c)\in R_{ab}$ or $(c,c)\in R_{ab}$, but then Lemma~\ref{triangle} implies $A=\{a,b,c\}$, a contradiction. Therefore, $(a,d)$ is a semilattice edge. Furthermore, if we replace $d$ with $c$ in the first part of the proof, we get that $(b,c)$ is a semilattice edge. 

Finally, $\{c,d\}$ cannot be a two-element semilattice with the operation $s(x,y)$, as otherwise, $\m A$ has an element, say $d$, that absorbs the entire algebra. Hence, $\{a,b,d\}$ is a 3-element semilattice with respect to the operation $s(t(x,y),t(y,x))$, and by Proposition~\ref{subuniverse} we get $\mathrm{Sg}\{a,b\}=\{a,b,d\}$, a contradiction.
\end{proof}

\subsection{Two-element semilattice quotient}\label{s:semi}

\begin{table}[h!]
\centering
 \begin{tabular}{|c||c|c|c|c|c|c|}
\hline
    Algebra & $\T_{4,1}$ & $\T_{4,2}$ & $\T_{4,3}$ & $\T_{4,4}$ & $\T_{4,5}$ & $\T_{4,6}$\\
\hline\hline
     $g{\restriction_{\{0,1\}}}$  & maj& aff & maj & maj & maj & aff  \\
\hline
    $g{\restriction_{\{0,2\}}}$  & $\text{min}^3_{0\leqslant 2}$ & $\text{min}^3_{0\leqslant 2}$ & aff & aff & $\text{min}^3_{0\leqslant 2}$ & $\text{min}^3_{0\leqslant 2}$ \\

\hline
	$g(0,0,3)$ & 0 & 1 & 0 & 2 &0 & 0\\
\hline
	$g(0,3,3)$ & 1 & 0 & 1 & 0 &0 & 0\\	

\hline
    $g(1,1,2)$  & 1 & 0 & 2 & 2 & 1 & 1 \\
\hline
	$g(1,2,2)$  & 0 & 1 & 0 & 0 & 1 & 1\\
	
\hline
    $g(1,1,3)$  & 1 & 1 & 1 & 2 & 1 & 1 \\
\hline
	$g(1,3,3)$ & 1 & 1 & 1 & 0 & 1& 1\\

\hline
	$g(2,2,3)$ & 0 & 1 & 0 & 2 & 0 & 0\\
\hline
	$g(2,3,3)$ & 1 & 0 & 2 & 2 & 1 & 1\\

\hline
	$g(0,1,2)$ & 0 & 1 & 2 & 2 & 0 & 1\\
\hline
	$g(0,2,1)$ & 0 & 1 & 2 & 2 & 1 & 0\\
\hline
	$g(0,1,3)$ & 1 & 0 & 1 & 2 & 0 & 1\\
\hline
	$g(0,3,1)$ & 1 & 0 & 1 & 2 & 1 & 0\\
\hline
	$g(0,2,3)$ & 0 & 1 & 2 & 0 & 0 & 0\\
\hline
	$g(0,3,2)$ & 0 & 1 & 2 & 0 & 0 & 1\\
\hline
	$g(1,2,3)$ & 1 & 0 & 2 & 0 & 1 & 0\\
\hline
	$g(1,3,2)$ & 1 & 0 & 2 & 0 & 1 & 1\\
\hline
\end{tabular}  
\caption{Algebras with ternary cyclic operation.}
\label{table:5}
\end{table}

\begin{table}[h!]
\centering
 \begin{tabular}{c|c c c c}

    $t$ & 0 & 1 & 2  & 3\\
\hline

    0 & 0 & 2 & 1 & 0  \\

	1 & 2 & 1 & 0 & 2 \\
	
	2 & 1 & 0 & 2 & 1\\
	
	3 & 0 & 2 & 1 & 3\\

\end{tabular}  
\caption{Algebra $\T_{4,7}$ with binary commutative operation.}
\label{table:6}
\end{table}

\begin{figure}[ht]
	\begin{subfigure}{0.24\textwidth}
		\centering
		\begin{tikzpicture}
			\tikzstyle{every node}=[draw,circle,fill=white,minimum size=4pt,inner sep=0pt]
			
			\node (a) [label=above:\strut$a$] {};
			\node (b) [fill=gray,right=1.5cm of a,label=above:\strut$b$] {};
			\node (c) [fill=gray, below=1.5cm of a,label=below:\strut$c$] {};
			\node (d) [fill=gray,right=1.5cm of c,label=below:\strut$d$] {};

             \draw[thin, gray, densely dotted, rounded corners] ([shift={(0.14,0.1)}]b.north east)--([shift={(-0.14,0.14)}]b.north east)--([shift={(-0.14,0.1)}]c.south west)--([shift={(-0.1,-0.14)}]c.south west)--([shift={(0.14,-0.14)}]d.south east)--cycle;
             \node[draw=none] at ([shift={(0.4,-0.3)}]b.east) {$\color{gray} \m t_1^{\n}$};

             \draw[thin, gray, densely dotted, rounded corners] ([shift={(0.14,0.14)}]a.north west)--([shift={(-0.14,0.1)}]a.north west)--([shift={(-0.14,0.1)}]c.south west)--([shift={(-0.1,-0.14)}]c.south west)--([shift={(0.14,-0.14)}]d.south east)--([shift={(0.14,0.14)}]d.south east)--cycle;
             \node[draw=none] at ([shift={(-0.4,-0.3)}]a.west) {$\color{gray} \m t_1^{\n}$};
			
			\path[->,>=stealth,line width=1.2pt, shorten <=0.7mm, shorten >=0.7mm]
			(a) edge [line width=0.4pt] [bend left=12] node[draw=none,near end,above,outer sep=3pt] {\scriptsize{s}} (d)
            (b) edge [line width=0.4pt] [bend right=12] node[draw=none,near end,above,outer sep=3pt] {\scriptsize{s}} (c)
            (a) edge [line width=0.4pt] [bend left=12] node[draw=none,midway,above] {\scriptsize{s}} (b)
            (b) edge [line width=0.4pt] [bend left=12] node[draw=none,midway,below, outer sep=1pt] {\scriptsize{s}} (a)
			(b) edge node[draw=none,midway,right] {\scriptsize{s}} (d)
			(a) edge node[draw=none,midway,left] {\scriptsize{s}} (c);
            \path[-,line width=1.2pt,densely dashed, shorten <=0.7mm, shorten >=0.7mm]
			(b) edge [line width=0.4pt] [bend left=12] node[draw=none,near start,right, outer sep=2pt] {\scriptsize{m}} (c)
            (a) edge [line width=0.4pt] [bend right=12] node[draw=none,near start,left, outer sep=2pt] {\scriptsize{m}} (d)
            (a) edge [line width=0.4pt] [bend left=35] node[draw=none,midway,above] {\scriptsize{m}} (b)
			(c) edge node[draw=none,midway,above] {\scriptsize{m}} (d);

			
		\end{tikzpicture}
		\caption{$\m t_{4,1}$}
		
	\end{subfigure}%
	\hfill
	\begin{subfigure}{0.24\textwidth}
		\centering
		\begin{tikzpicture}
			\tikzstyle{every node}=[draw,circle,fill=white,minimum size=4pt,inner sep=0pt]
			
			\node (a) [label=above:\strut$a$] {};
			\node (b) [fill=gray,right=1.5cm of a,label=above:\strut$b$] {};
			\node (c) [fill=gray, below=1.5cm of a,label=below:\strut$c$] {};
			\node (d) [fill=gray,right=1.5cm of c,label=below:\strut$d$] {};

              \draw[thin, gray, densely dotted, rounded corners] ([shift={(0.14,0.1)}]b.north east)--([shift={(-0.14,0.14)}]b.north east)--([shift={(-0.14,0.1)}]c.south west)--([shift={(-0.1,-0.14)}]c.south west)--([shift={(0.14,-0.14)}]d.south east)--cycle;
             \node[draw=none] at ([shift={(0.4,-0.3)}]b.east) {$\color{gray} \m t_2^{\n}$};

             \draw[thin, gray, densely dotted, rounded corners] ([shift={(0.14,0.14)}]a.north west)--([shift={(-0.14,0.1)}]a.north west)--([shift={(-0.14,0.1)}]c.south west)--([shift={(-0.1,-0.14)}]c.south west)--([shift={(0.14,-0.14)}]d.south east)--([shift={(0.14,0.14)}]d.south east)--cycle;
             \node[draw=none] at ([shift={(-0.4,-0.3)}]a.west) {$\color{gray} \m t_2^{\n}$};
			
			\path[->,>=stealth,line width=1.2pt, shorten <=0.7mm, shorten >=0.7mm]
			(a) edge [line width=0.4pt] [bend left=12] node[draw=none,near end,above,outer sep=3pt] {\scriptsize{s}} (d)
            (b) edge [line width=0.4pt] [bend right=12] node[draw=none,near end,above,outer sep=3pt] {\scriptsize{s}} (c)
            (a) edge [line width=0.4pt] [bend left=12] node[draw=none,midway,above] {\scriptsize{s}} (b)
            (b) edge [line width=0.4pt] [bend left=12] node[draw=none,midway,below, outer sep=1pt] {\scriptsize{s}} (a)
			(b) edge node[draw=none,midway,right] {\scriptsize{s}} (d)
			(a) edge node[draw=none,midway,left] {\scriptsize{s}} (c);
            \path[-,line width=1.2pt,densely dotted, shorten <=0.7mm, shorten >=0.7mm]
			(b) edge [line width=0.4pt] [bend left=12] node[draw=none,near start,right, outer sep=2pt] {\scriptsize{a}} (c)
            (a) edge [line width=0.4pt] [bend right=12] node[draw=none,near start,left, outer sep=2pt] {\scriptsize{a}} (d)
            (a) edge [line width=0.4pt] [bend left=35] node[draw=none,midway,above,outer sep=1pt] {\scriptsize{a}} (b)
			(c) edge node[draw=none,midway,above] {\scriptsize{a}} (d);

		\end{tikzpicture}
		\caption{$\m t_{4,2}$}
	\end{subfigure}
	\hfill
	\begin{subfigure}{0.24\textwidth}
		\centering
		\begin{tikzpicture}
			\tikzstyle{every node}=[draw,circle,fill=white,minimum size=4pt,inner sep=0pt]
			
			\node (a) [label=above:\strut$a$] {};
			\node (b) [fill=gray,right=1.5cm of a,label=above:\strut$b$] {};
			\node (c) [fill=gray, below=1.5cm of a,label=below:\strut$c$] {};
			\node (d) [fill=gray,right=1.5cm of c,label=below:\strut$d$] {};

            \draw[thin, gray, densely dotted, rounded corners] ([shift={(0.14,0.14)}]a.north west)--([shift={(-0.14,0.1)}]a.north west)--([shift={(-0.14,0.1)}]c.south west)--([shift={(-0.1,-0.14)}]c.south west)--([shift={(0.14,-0.14)}]d.south east)--([shift={(0.14,0.14)}]d.south east)--cycle;
             \node[draw=none] at ([shift={(-0.4,-0.3)}]a.west) {$\color{gray} \m t_1^{\n}$};

             \draw[thin, gray, densely dotted, rounded corners] ([shift={(0.14,0.1)}]b.north east)--([shift={(-0.14,0.14)}]b.north east)--([shift={(-0.14,0.1)}]c.south west)--([shift={(-0.1,-0.14)}]c.south west)--([shift={(0.14,-0.14)}]d.south east)--cycle;
             \node[draw=none] at ([shift={(0.4,-0.3)}]b.east) {$\color{gray} \m t_3^{\n}$};

			\path[->,>=stealth,line width=1.2pt, shorten <=0.7mm, shorten >=0.7mm]
			(a) edge [line width=0.4pt] [bend left=12] node[draw=none,near end,above,outer sep=3pt] {\scriptsize{s}} (d)
            (a) edge [line width=0.4pt] [bend left=12] node[draw=none,midway,above, outer sep=1pt] {\scriptsize{s}} (b)
			(a) edge node[draw=none,midway,left] {\scriptsize{s}} (c);
            \path[-,line width=1.2pt,densely dashed, shorten <=0.7mm, shorten >=0.7mm]
            (a) edge [line width=0.4pt] [bend right=12] node[draw=none,near start,left, outer sep=2pt] {\scriptsize{m}} (d)
			(c) edge node[draw=none,midway,above] {\scriptsize{m}} (d);
            \path[-,line width=1.2pt,densely dotted, shorten <=0.7mm, shorten >=0.7mm]
            (b) edge [line width=0.4pt] [bend left=12] node[draw=none,near start,right,outer sep=2pt] {\scriptsize{a}} (c)
            (b) edge [line width=0.4pt] [bend left=12] node[draw=none, midway,below,outer sep=1pt] {\scriptsize{a}} (a)
            (b) edge node[draw=none,midway,right] {\scriptsize{a}} (d);
		\end{tikzpicture}
		\caption{$\m t_{4,3}$}
	\end{subfigure}
	\hfill
	\begin{subfigure}{0.24\textwidth}
		\centering
		\begin{tikzpicture}
			\tikzstyle{every node}=[draw,circle,fill=white,minimum size=4pt,inner sep=0pt]
			
			\node (a) [label=above:\strut$a$] {};
			\node (b) [fill=gray,right=1.5cm of a,label=above:\strut$b$] {};
			\node (c) [fill=gray, below=1.5cm of a,label=below:\strut$c$] {};
			\node (d) [fill=gray,right=1.5cm of c,label=below:\strut$d$] {};

              \draw[thin, gray, densely dotted, rounded corners] ([shift={(0.14,0.1)}]b.north east)--([shift={(-0.14,0.14)}]b.north east)--([shift={(-0.14,0.1)}]c.south west)--([shift={(-0.1,-0.14)}]c.south west)--([shift={(0.14,-0.14)}]d.south east)--cycle;
             \node[draw=none] at ([shift={(0.4,-0.3)}]b.east) {$\color{gray} \m t_3^{\n}$};

             \draw[thin, gray, densely dotted, rounded corners] ([shift={(0.14,0.14)}]a.north west)--([shift={(-0.14,0.1)}]a.north west)--([shift={(-0.14,0.1)}]c.south west)--([shift={(-0.1,-0.14)}]c.south west)--([shift={(0.14,-0.14)}]d.south east)--([shift={(0.14,0.14)}]d.south east)--cycle;
             \node[draw=none] at ([shift={(-0.4,-0.3)}]a.west) {$\color{gray} \m t_2^{\n}$};
			
			\path[->,>=stealth,line width=1.2pt, shorten <=0.7mm, shorten >=0.7mm]
			(a) edge [line width=0.4pt] [bend left=12] node[draw=none,near end,above,outer sep=3pt] {\scriptsize{s}} (d)
            (a) edge [line width=0.4pt] [bend left=12] node[draw=none,midway,above, outer sep=1pt] {\scriptsize{s}} (b)
			(a) edge node[draw=none,midway,left] {\scriptsize{s}} (c);
            \path[-,line width=1.2pt,densely dotted, shorten <=0.7mm, shorten >=0.7mm]
			(b) edge [line width=0.4pt] [bend left=12] node[draw=none,near start,right, outer sep=2pt] {\scriptsize{a}} (c)
            (a) edge [line width=0.4pt] [bend right=12] node[draw=none,near start,left, outer sep=2pt] {\scriptsize{a}} (d)
            (b) edge [line width=0.4pt] [bend left=12] node[draw=none, midway,below,outer sep=1pt] {\scriptsize{a}} (a)
			(c) edge node[draw=none,midway,above] {\scriptsize{a}} (d);
            \path[-,line width=1.2pt,densely dashed, shorten <=0.7mm, shorten >=0.7mm]
            (b) edge node[draw=none,midway,right] {\scriptsize{m}} (d);

		\end{tikzpicture}
		\caption{$\m t_{4,4}$}
	\end{subfigure}
\begin{subfigure}{0.33\textwidth}
		\centering
		\begin{tikzpicture}
			\tikzstyle{every node}=[draw,circle,fill=white,minimum size=4pt,inner sep=0pt]
			
			\node (a) [label=above:\strut$a$] {};
			\node (b) [fill=gray,right=1.5cm of a,label=above:\strut$b$] {};
			\node (c) [fill=gray, below=1.5cm of a,label=below:\strut$c$] {};
			\node (d) [fill=gray,right=1.5cm of c,label=below:\strut$d$] {};

             \draw[thin, gray, densely dotted, rounded corners] ([shift={(0.14,0.1)}]b.north east)--([shift={(-0.14,0.14)}]b.north east)--([shift={(-0.14,0.1)}]c.south west)--([shift={(-0.1,-0.14)}]c.south west)--([shift={(0.14,-0.14)}]d.south east)--cycle;
             \node[draw=none] at ([shift={(0.4,-0.3)}]b.east) {$\color{gray} \m t_3^{\co}$};

              \draw[thin, densely dotted, rounded corners] ([shift={(0.14,0.14)}]a.north west)--([shift={(-0.14,0.1)}]a.north west)--([shift={(-0.14,0.1)}]c.south west)--([shift={(-0.1,-0.14)}]c.south west)--([shift={(0.14,-0.14)}]d.south east)--([shift={(0.14,0.14)}]d.south east)--cycle;
             \node[draw=none] at ([shift={(-0.4,-0.3)}]a.west) {$\color{gray} \m t_3^{\co}$};

			\path[->,>=stealth,line width=1.2pt, shorten <=0.7mm, shorten >=0.7mm]
			(a) edge node[draw=none,near end,above,outer sep=3pt] {\scriptsize{s}} (d)
            (b) edge node[draw=none,midway,left,outer sep=3pt] {\scriptsize{s}} (c)
            (a) edge [line width=0.4pt] [bend left=12] node[draw=none,midway,above, outer sep=1pt] {\scriptsize{s}} (b)
            (b) edge [line width=0.4pt] [bend left=12] node[draw=none,midway,below, outer sep=1pt] {\scriptsize{s}} (a)
			(b) edge node[draw=none,midway,right] {\scriptsize{s}} (d)
			(a) edge node[draw=none,midway,left] {\scriptsize{s}} (c);
            \path[-,line width=1.2pt,densely dashed, shorten <=0.7mm, shorten >=0.7mm]
			(c) edge node[draw=none,midway,above] {\scriptsize{m}} (d);

		\end{tikzpicture}
		\caption{$\m t_{4,5}$}
		
	\end{subfigure}%
	\hfill
	\begin{subfigure}{0.33\textwidth}
		\centering
		\begin{tikzpicture}
			\tikzstyle{every node}=[draw,circle,fill=white,minimum size=4pt,inner sep=0pt]
			
			\node (a) [label=above:\strut$a$] {};
			\node (b) [fill=gray,right=1.5cm of a,label=above:\strut$b$] {};
			\node (c) [fill=gray, below=1.5cm of a,label=below:\strut$c$] {};
			\node (d) [fill=gray,right=1.5cm of c,label=below:\strut$d$] {};

             \draw[thin, gray, densely dotted, rounded corners] ([shift={(0.14,0.1)}]b.north east)--([shift={(-0.14,0.14)}]b.north east)--([shift={(-0.14,0.1)}]c.south west)--([shift={(-0.1,-0.14)}]c.south west)--([shift={(0.14,-0.14)}]d.south east)--cycle;
             \node[draw=none] at ([shift={(0.4,-0.3)}]b.east) {$\color{gray} \m t_7^{\co}$};

              \draw[thin, densely dotted, rounded corners] ([shift={(0.14,0.14)}]a.north west)--([shift={(-0.14,0.1)}]a.north west)--([shift={(-0.14,0.1)}]c.south west)--([shift={(-0.1,-0.14)}]c.south west)--([shift={(0.14,-0.14)}]d.south east)--([shift={(0.14,0.14)}]d.south east)--cycle;
             \node[draw=none] at ([shift={(-0.4,-0.3)}]a.west) {$\color{gray} \m t_7^{\co}$};

			\path[->,>=stealth,line width=1.2pt, shorten <=0.7mm, shorten >=0.7mm]
			(a) edge node[draw=none,near end,above,outer sep=3pt] {\scriptsize{s}} (d)
            (b) edge node[draw=none,midway,left,outer sep=3pt] {\scriptsize{s}} (c)
            (a) edge [line width=0.4pt] [bend left=12] node[draw=none,midway,above, outer sep=1pt] {\scriptsize{s}} (b)
            (b) edge [line width=0.4pt] [bend left=12] node[draw=none,midway,below, outer sep=1pt] {\scriptsize{s}} (a)
			(b) edge node[draw=none,midway,right] {\scriptsize{s}} (d)
			(a) edge node[draw=none,midway,left] {\scriptsize{s}} (c);
            \path[-,line width=1.2pt,densely dotted, shorten <=0.7mm, shorten >=0.7mm]
			(c) edge node[draw=none,midway,above] {\scriptsize{a}} (d);

		\end{tikzpicture}
		\caption{$\m t_{4,6}$}
	\end{subfigure}
	\hfill
	\begin{subfigure}{0.33\textwidth}
		\centering
		\begin{tikzpicture}
			\tikzstyle{every node}=[draw,circle,fill=white,minimum size=4pt,inner sep=0pt]
			
			\node (a) [label=above:\strut$a$] {};
			\node (b) [fill=gray,right=1.5cm of a,label=above:\strut$b$] {};
			\node (c) [fill=gray, below=1.5cm of a,label=below:\strut$c$] {};
			\node (d) [fill=gray,right=1.5cm of c,label=below:\strut$d$] {};

             \draw[thin, gray, densely dotted, rounded corners] ([shift={(0.14,0.1)}]b.north east)--([shift={(-0.14,0.14)}]b.north east)--([shift={(-0.14,0.1)}]c.south west)--([shift={(-0.1,-0.14)}]c.south west)--([shift={(0.14,-0.14)}]d.south east)--cycle;
             \node[draw=none] at ([shift={(0.4,-0.3)}]b.east) {$\color{gray} \m t_5^{\n}$};
			
			\path[->,>=stealth,line width=1.2pt, shorten <=0.7mm, shorten >=0.7mm]
			(a) edge [line width=0.4pt] [bend left=12] node[draw=none,near end,above,outer sep=3pt] {\scriptsize{s}} (d)
            (a) edge [line width=0.4pt] [bend right=12] node[draw=none,midway,below, outer sep=1pt] {\scriptsize{s}} (b)
			(a) edge node[draw=none,midway,left] {\scriptsize{s}} (c);
            \path[-,line width=1.2pt,densely dotted, shorten <=0.7mm, shorten >=0.7mm]
			(b) edge node[draw=none,near start,right, outer sep=2pt] {\scriptsize{a3}} (c)
            (b) edge node[draw=none,midway,right, outer sep=1pt] {\scriptsize{a3}} (d)
            (a) edge [line width=0.4pt] [bend left=12] node[draw=none,midway,above, outer sep=1pt] {\scriptsize{a}} (b)
            (a) edge [line width=0.4pt] [bend right=12] node[draw=none,near start,left, outer sep=3pt] {\scriptsize{a}} (d)
			(c) edge node[draw=none,midway,above, outer sep=1pt] {\scriptsize{a3}} (d);

		\end{tikzpicture}
		\caption{$\m t_{4,7}$}
	\end{subfigure}
	
	\caption{Directed graph of $\m t_{4,i}$ for $i=1,2,\dots, 7$.}
\end{figure}
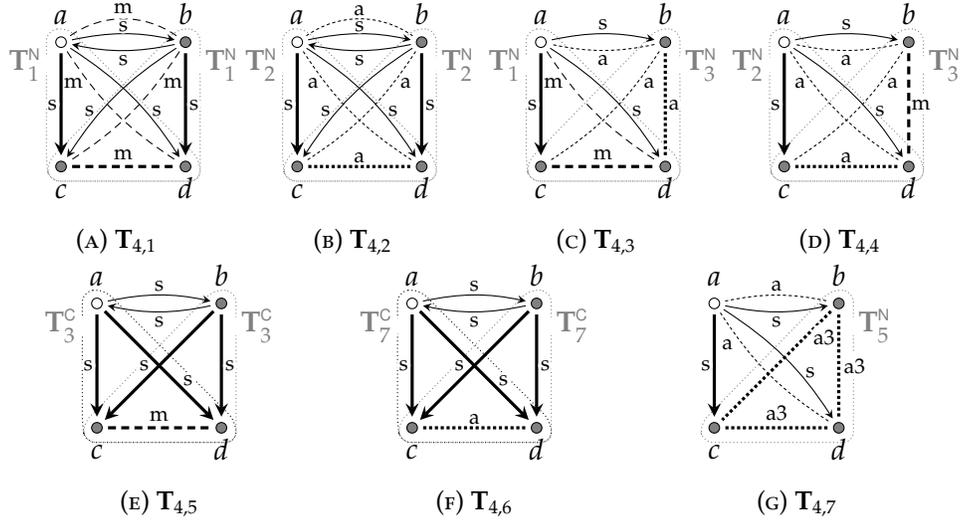

In this subsection, we shall prove that seven minimal Taylor algebras on a four-element domain have a two-element semilattice quotient. \Cref{table:5} contains the first six of these algebras, each of them containing an idempotent ternary cyclic operation. The seventh algebra contains a binary commutative operation as presented in \Cref{table:6}. In addition, we note that the algebras ${\m T}_{4,5}$ and ${\m T}_{4,6}$ do not have a unique ternary cyclic term, unlike the first four algebras.

Suppose that $\m a/\theta$ is a two-element semilattice for some congruence $\theta$. Since $\m a$ is 2-generated and each $\theta$-class is a subuniverse, we know that each $\theta$-class contains one of the generators.
By Proposition~\ref{bradybinaryprop}, $\theta$'s absorbing class has three elements. Without loss of generality, let $B:=b/\theta=\{b,c,d\}$ and $a/\theta=\{a\}$. By Proposition~\ref{smedge} we know that there must exist an element $x\in \{c,d\}$, such that $(a,x)$ is a semilattice edge. Until the end of Subsection 4.1, we assume that $(a,c)$ is a semilattice edge (with absorbing element $c$).


We start by examining cases when $\m b$ is nonconservative. If $\m B$ has a binary absorbing set $B_1$, then by Proposition~\ref{composition}, follows that $B_1\trianglelefteq_2 {\m A}$ and $B_1\cup \{a\}\leq {\m A}$. Consequently, the element $b$ cannot belong to some binary absorbing set of algebra ${\m B}$. As ${\m T}_4^{\n}$ is the union of its 2-element absorbing subsets, $\m b\cong {\m T}_4^{\n}$ can't hold.

\begin{claim}
If ${\m B}\cong {\m T}_1^{\n}$ or ${\m B}\cong {\m T}_2^{\n}$, then ${\m A}$ is $ {\T}_{4,1}$, respectively ${\m A}$ is $ {\T}_{4,2}$, up to isomorphism and term-equivalence.
\end{claim}

\begin{proof}
Suppose that ${\m B}\cong {\m T}_i^{\n}$, $i\in\{1,2\}$. The majority of the proof goes the same whichever $i\in\{1,2\}$ we pick, so we fix $i$ as one of these two indices. Since the algebra ${\m T}_i^{\n}$ has a ternary cyclic term, the same holds for $\m b$. Then, by Lemma~\ref{bkcyclic} applied to the congruence $\{\{a\},\{b,c,d\}\}$, we know that the algebra $\m A$ also has a ternary cyclic term, and we denote by $g$ one such term.

Since $\{0,1\}\trianglelefteq_2 {\m T}_i^{\n}$, using Proposition~\ref{composition} (2) and Proposition~\ref{bradybinaryprop} (3), the isomorphism ${\m B}\cong {\m T}_1^{\n}$ must satisfy $b\mapsto 2$, where elements $0,1,2$ are from \Cref{table:1} and/or \Cref{3elemgraf} (otherwise $|\mathrm{Sg}\{a,b\}|=3$). If $c\mapsto 1$ and $d\mapsto 0$, then $\{a,c\}$ and $\{b,c\}$ would be subuniverses (both are semilattice edges), and by Lemma~\ref{triangle2} $\{b,d\}$ would also be a subuniverse, i.e. $\m b$ would be conservative and ${\m B}\cong {\m T}_i^{\n}$ would fail. Hence, the isomorphism ${\m B}\cong {\m T}_i^{\n}$ is $b\mapsto 2$, $c\mapsto 1$ and $d\mapsto 0$. Then $\{c,d\}\trianglelefteq_2 {\m A}$ and by Proposition~\ref{bradybinaryprop}, $\{a,c,d\}$ is a subalgebra. By Lemma~\ref{Aljahomework}, ${\m T}_i^{\n}$ has a single ternary cyclic term, the one given in \Cref{table:1}, so we already know $g|_B$. Furthermore, by Proposition~\ref{bradybinaryprop} and $\mathrm{Sg}\{a,b\}=A$, we have ${\m A}/\alpha\cong \m T_4^{\n}$, for the congruence $\alpha=\{\{a\},\{b\},\{c,d\}\}$, so we know that $g(a,a,b), g(b,b,a)\in \{c,d\}$.

From Lemma~\ref{triangle2} and the fact that $\{b,d\}$ is a subuniverse follows that $\{a,d\}$ is not a subuniverse. Using the assumption that $\{a,c\}$ is a two-element semilattice and \Cref{table:1}, we see that $\mathrm{Sg}\{a,d\}=\{a,c,d\}\cong {\T}_i^{\n}$ holds with $a\mapsto 2$, $c\mapsto 0$ and $d\mapsto 1$. Hence, we have also determined $g|_{\{a,c,d\}}$, according to Lemma~\ref{Aljahomework}. It remains to determine $g$ on tuples $(a,b,b), (b,a,a), (b,a,c), (b,c,a)$, $(b,a,d)$ and $(b,d,a)$. 

Now we assume that $\m b\cong {\T}_1^{\n}$ and we define the ternary term $t'(x,y,z):=g(x,x,g(x,y,z))$ and let $g'$ be the ternary cyclic term defined by $$g'(x,y,z)=g(t'(x,y,z),t'(y,z,x),t'(z,x,y)).$$
Note that for $\overline{a}=g(b,a,a), \overline{b}=g(a,b,b)$, we have
\begin{align*}
t'\left(
\begin{bmatrix}
a \\
b \\
b
\end{bmatrix}\!,\!
\begin{bmatrix}
b \\
a \\
b
\end{bmatrix}\!,\!
\begin{bmatrix}
b \\
b \\
a
\end{bmatrix}
\right)& =g\left(
\begin{bmatrix}
a \\
b \\
b
\end{bmatrix}\!,\! 
\begin{bmatrix}
a \\
b \\
b
\end{bmatrix}\!,
g\left(
\begin{bmatrix}
a \\
b \\
b
\end{bmatrix}\!,\!
\begin{bmatrix}
b \\
a \\
b
\end{bmatrix}\!,\!
\begin{bmatrix}
b \\
b \\
a
\end{bmatrix}
\right)
\right)=g\left(
\begin{bmatrix}
a \\
b \\
b
\end{bmatrix}\!,\!
\begin{bmatrix}
a \\
b \\
b
\end{bmatrix}\!,\!
\begin{bmatrix}
\overline{b} \\
\overline{b} \\
\overline{b}
\end{bmatrix}
\right)
=
\begin{bmatrix}
c \\
d \\
d
\end{bmatrix}
\end{align*}
so $g'(a,b,b)=g(c,d,d)=d$. Similarly, after swapping $a$ and $b$, thus $\overline{a}$ and $\overline{b}$, we get $g'(b,a,a)=c$.

Finally, consider the ternary term $g''$ given by $$g''(x,y,z)=g'(g'(x,x,y),g'(y,y,z),g'(z,z,x)).$$ The term $g''$ is cyclic and we have
\begin{align*}
g''(a,b,b)&=g'(g'(a,a,b),g'(b,b,b),g'(b,b,a))=g'(c,b,d)=d,\\
g''(b,a,a)&=g'(g'(b,b,a),g'(a,a,a),g'(a,a,b))=g'(d,a,c)=c,\\
g''(b,a,c)&=g'(g'(b,b,a), g'(a,a,c), g'(c,c,b))=g'(d,c,c)=c,\\
g''(b,c,a)&=g'(g'(b,b,c), g'(c,c,a), g'(a,a,b))=g'(d,c,c)=c,\\
g''(b,a,d)&=g'(g'(b,b,a), g'(a,a,d), g'(d,d,b))=g'(d,c,d)=d,\\
g''(b,d,a)&=g'(g'(b,b,d), g'(d,d,a), g'(a,a,b))=g'(d,d,c)=d,
\end{align*}
so $g''$ is completely determined. Thus there can be at most one minimal Taylor algebra satisfying $\m b\cong {\T}_1^{\n}$.

Similarly, if $\m b\cong {\T}_2^{\n}$, we define the ternary term $$t'(x,y,z):=g(x,g(x,y,z),g(x,y,z))$$ and the cyclic term $g'$ by
$$g'(x,y,z)=g(t'(x,y,z),t'(y,z,x),t'(z,x,y)).$$
For $\overline{a}=g(b,a,a), \overline{b}=g(a,b,b)$, we have
\begin{align*}
t'\left(
\begin{bmatrix}
a \\
b \\
b
\end{bmatrix}\!,\!
\begin{bmatrix}
b \\
a \\
b
\end{bmatrix}\!,\!
\begin{bmatrix}
b \\
b \\
a
\end{bmatrix}
\right)& =g\left(
\begin{bmatrix}
a \\
b \\
b
\end{bmatrix}\!,
g\left(
\begin{bmatrix}
a \\
b \\
b
\end{bmatrix}\!,\! 
\begin{bmatrix}
b \\
a \\
b
\end{bmatrix}\!,\!
\begin{bmatrix}
b \\
b \\
a
\end{bmatrix}
\right)\!,
g\left(
\begin{bmatrix}
a \\
b \\
b
\end{bmatrix}\!,\! 
\begin{bmatrix}
b \\
a \\
b
\end{bmatrix}\!,\!
\begin{bmatrix}
b \\
b \\
a
\end{bmatrix}\!
\right)
\right)=g\left(
\begin{bmatrix}
a \\
b \\
b
\end{bmatrix}\!,\! 
\begin{bmatrix}
\overline{b} \\
\overline{b} \\
\overline{b}
\end{bmatrix}\!,\!
\begin{bmatrix}
\overline{b} \\
\overline{b} \\
\overline{b}
\end{bmatrix}
\right)=
\begin{bmatrix}
c \\
d \\
d
\end{bmatrix}\!,
\end{align*}
thus, $ g'(a,b,b)=g(c,d,d)=c$. As before, by swapping $a$ and $b$, thus $\overline{a}$ and $\overline{b}$, we get $g'(b,a,a)=d$. In this case, the ternary cyclic term $g''$ is defined by $$ g''(xyz)=g'(g'(xyy),g'(yzz),g'(zxx)),$$ so we get
\begin{align*}
g''(a,b,b)&=g'(g'(a,b,b), g'(b,b,b), g'(b,a,a))=g'(c,b,d)=c,\\
g''(b,a,a)&=g'(g'(b,a,a), g'(a,a,a), g'(a,b,b))=g'(d,a,c)=d,\\
g''(b,a,d)&=g'(g'(b,a,a), g'(a,d,d), g'(d,b,b))=g'(d,c,d)=c,\\
g''(b,d,a)&=g'(g'(b,d,d), g'(d,a,a), g'(a,b,b))=g'(d,d,c)=c,\\
g''(b,a,c)&=g'(g'(b,a,a), g'(a,c,c), g'(c,b,b))=g'(d,c,c)=d,\\
g''(b,c,a)&=g'(g'(b,c,c), g'(c,a,a), g'(a,b,b))=g'(d,c,c)=d.
\end{align*}

Now consider the Taylor algebra $\m c=(A;g'')$, where $g''$ is the ternary cyclic term we obtained from one of the cases $\m b\cong {\T}_1^{\n}$ and $\m b\cong {\T}_2^{\n}$. $\beta=\{\{a,c\},\{b,d\}\}$ is a congruence of the Taylor algebra $\m c$, and the quotient $\m c/\beta$ is term equivalent to the two-element majority algebra in the case $\m b\cong {\T}_1^{\n}$, or the two-element affine algebra in the case $\m b\cong {\T}_2^{\n}$. Therefore, every Taylor reduct of the algebra $\m c$ must have congruences $\alpha, \beta$ and these congruences completely determine any ternary cyclic term, making it equal to $g''$. Such a ternary cyclic term must exist in any Taylor reduct of $(A;g'')$ by Lemma~\ref{bkcyclic} applied to either $\alpha$ or $\beta$ (each congruence works), so $(A;g'')$ is a minimal Taylor algebra. This completes the proof that $\m A$ is term equivalent to the minimal Taylor algebra ${\m T}_{4,1}$, respectively ${\m T}_{4,2}$, with respect to the isomorphism $ (a\mapsto 3, b\mapsto 2, c\mapsto 1, d\mapsto 0)$.
\end{proof}

\begin{claim}
If ${\m B}\cong {\m T}_3^{\n}$, then ${\m A}$ is ${\T}_{4,3}$ or ${\T}_{4,4}$, up to isomorphism and term-equivalence.
\end{claim}

\begin{proof}
Since algebra ${\m T}_3^{\n}$ has a ternary cyclic term, by Lemma \ref{bkcyclic} we know that $\m A$ has a ternary cyclic term. Let $g$ denote a ternary cyclic term of $\m a$. Since ${\m B}\cong {\m T}_3^{\n}$ and $(a,c)$ is a semilattice edge, by Lemma~\ref{triangle2} we know that $\{b,c\}$ can't be a subuniverse, and as $\{0,1\}$ and $\{0,2\}$ are subuniverses, thus $d\mapsto 0$. Since $\{a,c\}, \{b,d\}$ and $\{c,d\}$ are subuniverses and $\m b$ is nonconservative, Lemma~\ref{triangle2} implies that $\mathrm{Sg}\{a,d\}\neq\{a,d\}$. Next, notice that $\mathrm{Sg}\{a,d\}\not= {\m A}$. Otherwise, we could apply Lemma~\ref{triangle2} to $\m a=\mathrm{Sg}\{a,d\}$, and deduce that $(d,b)$ is a semilattice edge, a contradiction. In the case when $b\mapsto 2$, $c\mapsto 1$ and $d\mapsto 0$ hold, we have $\mathrm{Sg}\{a,d\}=\{a,c,d\}\cong {\T}_1^{\n}$, while if $b\mapsto 1$, $c\mapsto 2$ and $d\mapsto 0$, we have $\mathrm{Sg}\{a,d\}=\{a,c,d\}\cong {\T}_2^{\n}$. In either case,  by Lemma~\ref{Aljahomework}, $g|_{\{a,c,d\}}$ and $g|_B$ are uniquely determined. We proceed by examining each of the cases, and it is notable that in both of them, it will be proved as a consequence that $ \{\{b\},\{d\},\{a,c\}\}$ is a congruence.

Assume first that $b\mapsto 2$, $ c\mapsto 1$ and $d\mapsto 0$. We will find a ternary cyclic term $g''$ for which $g''(b,b,a)=d$ and $g''(a,a,b)=b$. First we want a ternary cyclic term $g'$ such that $g'(b,b,a)=d$. Since $\{b,d\}\trianglelefteq_3 B \trianglelefteq_3 {\m A}$, then $\{b,d\}\trianglelefteq_3 {\m A}$. Consequently, we have $g(b,b,a)\in \{b,d\}$. If $g(b,b,a)=d$, then we can take $g'(x,y,z):=g(x,y,z)$. Otherwise, we define the term $g'$ by $g'(x,y,z):=g(g(x,x,y),g(y,y,z),g(z,z,x))$, so
$$g'(b,b,a)=g(g(b,b,b),g(b,b,a),g(a,a,b))=g(b,b,\overline{a})=d,$$ where  $\overline{a}=g(a,a,b)\in \{c,d\}$.

If $g'(a,a,b)=b$, then we can take $g'$ as the desired term $g''$. Otherwise, let $\Tilde{a}:=g'(a,a,b)\in \{c,d\}.$ Define the ternary term $t(x,y,z):=g'(x,g'(x,y,z),$ $g'(x,y,z))$ and the cyclic term
 $$g''(x,y,z):=g'(t(x,y,z),t(y,z,x),t(z,x,y)).$$
 Now, we have
 $$g''(b,b,a)=g'(g'(b,d,d),g'(b,d,d),g'(a,d,d))=g'(b,b,d)=d,$$
 $$g''(a,a,b)=g'(g'(a,\Tilde{a},\Tilde{a}),g'(a,\Tilde{a},\Tilde{a}),g'(b,\Tilde{a},\Tilde{a}))=g'(\Tilde{a},\Tilde{a},b)=b.$$
 
Next, we define a new cyclic term
 $$g^*(x,y,z):=g''(g''(x,x,g''(y,y,z)),g''(y,y,g''(z,z,x)),g''(z,z,g''(x,x,y))).$$
Of course, $g^*$ is determined on $\{a,c,d\}$ and $\{b,c,d\}$ by Lemma~\ref{Aljahomework}. Since $g''$ is completely determined on all entries which have two equal values, we compute
\begin{align*}
g^*(a,b,b)&=g''(g''(a,a,g''(b,b,b)),g''(b,b,g''(b,b,a)),g''(b,b,g''(a,a,b)))=g''(b,d,b)=d,\\
g^*(b,a,a)&=g''(g''(b,b,g''(a,a,a)),g''(a,a,g''(a,a,b)),g''(a,a,g''(b,b,a)))=
g''(d,b,c)=b,\\
g^*(b,a,d)&=g''(g''(b,b,g''(a,a,d)),g''(a,a,g''(d,d,b)),g''(d,d,g''(b,b,a)))
=g''(d,b,d)=b,\\
g^*(b,d,a)&=g''(g''(b,b,g''(d,d,a)),g''(d,d,g''(a,a,b)),g''(a,a,g''(b,b,d)))
=g''(d,b,c)=b,\\
g^*(b,a,c)&=g''(g''(b,b,g''(a,a,c)),g''(a,a,g''(c,c,b)),g''(c,c,g''(b,b,a)))=
g''(d,b,c)=b,\\
g^*(b,c,a)&=g''(g''(b,b,g''(c,c,a)),g''(c,c,g''(a,a,b)),g''(a,a,g''(b,b,c)))
=g''(d,b,c)=b.
\end{align*}
Hence there is at most one minimal Taylor clone in this case, as each contains the cyclic operation $g^*$.

Now, the algebra $\m c:=(A;g^*)$ has the congruence $\beta=\{\{b\},\{d\},\{a,c\}\}$, such that ${\m c}/\beta\cong {\m T}_3^{\n}$, where $\{b\}\mapsto 2, \{d\}\mapsto 0$ and $\{a,c\}\mapsto 1$ ($0,1,2$ are elements of the algebra ${\m T}_3^{\n}$ from \Cref{table:1}). So, every Taylor reduct of $\m c$ must have as congruences $\beta$ and $\{\{a\},\{b,c,d\}\}$ with semilattice quotient, and these two fully determine any ternary cyclic term, making it equal to $g^*$. Also, any Taylor reduct of $\m c$ must have a ternary cyclic term by Lemma~\ref{bkcyclic} applied to $\beta$, so $\m c$ is minimal Taylor. Moreover, $\m c$ is isomorphic to ${\m T}_{4,3}$ from \Cref{table:6}. Consequently, the algebra $\m A$ is term-equivalent to ${\m T}_{4,3}$ $(a\mapsto 3, b\mapsto 2, c\mapsto 1, d\mapsto 0)$.

Assume now that $\m b\cong{\m T}_3^{\n}$ with $b\mapsto 1, c\mapsto 2, d\mapsto 0$. Recall that $\mathrm{Sg}\{a,d\}=\{a,c,d\}\cong {\T}_2^{\n}$ and that g is determined on $B$ and on $\{a,c,d\}$. Let us find first a ternary cyclic term $g'$ for which $g'(b,b,a)=c$ .

If $g(b,b,a)=c$, then we take $g':=g$. If $g(b,b,a)=d$, we first define $t(x,y,z):=g(x,g(x,y,z),g(x,y,z))$ and the cyclic term $$g'(x,y,z):=g(t(x,y,z),t(y,z,x),t(z,x,y)).$$ Now we have
$$g'(b,b,a)=g(g(b,d,d),g(b,d,d),g(a,d,d))=g(d,d,c)=c.$$
If $g(b,b,a)=b$, then we know that $\ast:=g(a,a,b)\in \{c,d\}$. Next, we define $$t(x,y,z):=g(x,g(x,x,y),g(x,x,y))$$ and $$g'(x,y,z):=g(t(x,y,z),t(y,z,x),t(z,x,y)).$$ Hence,
$$g'(b,b,a)=g(g(b,b,b),g(b,b,b),g(a,\ast,\ast))=g(b,b,c)=c.$$

Therefore, we can assume $g'(b,b,a)=c$. 
 
Now we seek the ternary cyclic term $g''$ such that $g''(b,b,a)=c$ and $g''(a,a,b)=d$. Assume that $\circ:=g'(a,a,b)\in \{b,c\}$, otherwise we could take $g''=g'$. Define $t(x,y,z):=g'(x,g'(x,y,z),g'(x,y,z))$ and the cyclic term $g''(x,y,z):=g'(t(x,y,z),t(y,z,x),t(z,x,y)).$ Next, note that $\bullet:=g'(b,\circ,\circ)\in \{b,d\}$, and
\begin{align*}
g''(b,b,a)=g'(g'(b,c,c),g'(b,c,c),g'(a,c,c))=g'(d,d,c)=c,\\
g''(a,a,b)=g'(g'(a,\circ,\circ),g'(a,\circ,\circ),g'(b,\circ,\circ))=g'(c,c,\bullet)=d.
\end{align*}

As in the previous case, we will obtain the ternary cyclic term $g^*$ which is determined on all values. Define
$$g^*(x,y,z):=g''(g''(x,x,y),g''(y,y,z),g''(z,z,x)).$$
We know from Lemma~\ref{Aljahomework} that every ternary cyclic operation acts as the one from Table \ref{table:1} on $\m b\cong {\m T}_3^{\n}$ and $\mathrm{Sg}\{a,d\}\cong{\m T}_2^{\n}$. We need to determine the values of $g^*$ for tuples $(a,b,b), (b,b,a), (b,a,c), (b,c,a),$ $(b,a,d)$ and $(b,d,a)$. We leave to the reader to verify that
\begin{align*}
     g^*(a,b,b)&=g^*(b,a,d)=g^*(b,d,a)=c\\
     \shortintertext{and}
     g^*(b,a,a)&=g^*(b,a,c)=g^*(b,c,a)=d.
\end{align*}
Now, the algebra $\m c:=(A;g^*)$ has the congruence $\beta=\{\{b\},\{d\},\{a,c\}\}$, such that ${\m c}/\beta\cong {\m T}_3^{\n}$, where $\{b\}\mapsto 1, \{d\}\mapsto 0$ and $\{a,c\}\mapsto 2$. So, every Taylor reduct of $\m c$ must have as congruences $\beta$ and $\{\{a\},\{b,c,d\}\}$ with semilattice quotient, and these two determine that any ternary cyclic term is equal to $g^*$. Also, any Taylor reduct of $\m c$ must have a ternary cyclic term by Lemma~\ref{bkcyclic} applied to $\beta$, so $\m c$ is minimal Taylor. Moreover, $\m c$ is isomorphic to ${\m T}_{4,4}$ from \Cref{table:6}. Consequently, the algebra $\m A$ is term-equivalent to ${\m T}_{4,4}$ $(a\mapsto 3, b\mapsto 1, c\mapsto 2, d\mapsto 0)$.
\end{proof}

\begin{claim}
If ${\m B}\cong {\m T}_5^{\n}\cong \mathbb{Z}^{\mathrm{aff}}_3$, then ${\m A}$ is $ {\m T}_{4,7}$, up to isomorphism and term-equivalence.
\end{claim}

\begin{proof}
The minimal Taylor algebra $\mathbb{Z}^{\mathrm{aff}}_3$ has a Taylor operation $d(x,y,z)=x-y+z$ and a unique binary commutative operation, $p(x,y):=d(x,y,x)=2x+2y$. Note that $d(x,y,z)=p(p(x,z),y)$, so $p$ also generates the clone of $\mathbb{Z}^{\mathrm{aff}}_3$. Since $\{a\}$, $\m b$ and the semilattice quotient have a binary commutative term, by Lemma~\ref{bkcyclic}, the algebra $\m A$ has some binary commutative term $t$.

We know that $\{a,c\}$ is the two-element semilattice and so $t(b,c)=d$, $t(b,d)=c$, $t(c,d)=b$ and $t(a,c)=c$. Therefore, we only need to determine $t(a,b)$ and $t(a,d)$. Let $R_{ab}:=\mathrm{Sg}_{\m a^2}\{(a,b),(b,a)\}$. 

Let us first determine $t(a,b)$. If $t(a,b)=b$, then $\mathrm{Sg}\{a,b\}=\{a,b\}$, a contradiction. If $t(a,b)=c$, then $(c,c)\in R_{ab}$. Since $(c,c), (a,b)\in R_{ab}$, then $(c,d)\in R_{ab}$ and $(d,c)\in R_{ab}$ by symmetry. But then $(b,b)=(t(c,d),t(d,c))\in R_{ab}$, again $(a,b)$ is a semilattice edge, i.e. $\mathrm{Sg}\{a,b\}=\{a,b\}$, a contradiction. Thus, $t(a,b)=d$ must hold. Now, we know that $(d,d)\in R_{ab}$, and $(t(a,d),t(b,d))=(t(a,d),c)\in R_{ab}$. We have seen that $(c,c)\not\in R_{ab}$ and $(d,c)\notin R_{ab}$. So, $t(a,d)=b$ is forced. Since $t$ is Taylor, there can be at most one minimal Taylor clone in case ${\m B}\cong {\m T}_5^{\n}$.

The Taylor algebra $\m c:=(A;t)$ has the congruence $\alpha=\{\{a,c\},\{b\},\{d\}\}$, such that ${\m c}/\alpha$ is term-equivalent to $\mathbb{Z}^{\mathrm{aff}}_3$. Hence, every Taylor reduct $\m d$ of $\m c$ must have congruences $\alpha$ and $\beta=\{\{a\},\{b,c,d\}\}$ with the semilattice quotient, and these congruences ensure that any binary commutative term of $\m d$ is equal to $t$. Moreover, using Lemma~\ref{bkcyclic} on $\alpha$, we know that every Taylor reduct of $\m c$ must have a commutative binary term. We conclude that $\m c$ is a minimal Taylor algebra isomorphic to ${\m T}_{4,7}$ $ (a\mapsto 3, b\mapsto 1, c\mapsto 0, d\mapsto 2)$.
\end{proof}

According to Lemma~\ref{triangle2} and \Cref{table:4}, the only two possibilities for a conservative $\m b$ are $\m t_3^{\co}$ and $\m t_7^{\co}$.

\begin{claim}
If ${\m B}\cong {\m T}_3^{\co}$, then ${\m A}$ is $ {\m T}_{4,5}$, up to isomorphism and term-equivalence.
\end{claim}
\begin{proof}
If ${\m B}\cong {\m T}_3^{\co}$, then $b\mapsto 2$, where is $2$ element of ${\m T}_3^{\co}$ from \Cref{table:4}. Algebra ${\m T}_3^{\co}$ has a ternary cyclic term, so by Lemma~\ref{bkcyclic} we know that algebra $\m A$ has a ternary cyclic term $g$. We know that $(a,c), (b,c), (a,d)$ and $(b,d)$ must be semilattice edges. As a result, $\{a,c,d\}$ is the underlying set of a subalgebra isomorphic to ${\m T}_3^{\co}$. Also, recall that $\{c,d\}\trianglelefteq_2 {\m A}$ holds.

Lemma~\ref{triangle} implies that for each binary term $t$ we have $t(a,b)=c$ if and only if $t(b,a)=d$. Then there exists an automorphism $\varphi$ of order two such that $\varphi (a)=b$ and $\varphi (c)=d$.

Consider an auxiliary cyclic term $g'$ defined by $$g'(x,y,z)=g(g(x,x,y),g(y,y,z),g(z,z,x)).$$ Then 
\begin{gather*}
    g'(b,c,d)=g'(a,c,d)=c,\\
    g'(b,d,c)=g'(a,d,c)=d.
\end{gather*}


Once again, from Proposition~\ref{bradybinaryprop} follows that ${\m A}/\alpha$ must be $\m T_4^{\n}$, for the congruence $\alpha=\{\{a\},\{b\},\{c,d\}\}$. Henceforth, we assume that $g'(a,a,b), g'(b,b,a)\in \{c,d\}$. Let $t$ be a binary term such that $t(a,b)=d$ and $t(b,a)=c$. Note that for the term $t$ we can take $g'(x,x,y)$ or $g'(x,y,y)$, depending on the value of the term $g'$. We define a new cyclic term $g''$ which shall be completely determined, as follows
$$g''(x,y,z)=g'(g'(x,x,t(x,y)),g'(y,y,t(y,z)),g'(z,z,t(z,x))).$$
The cyclic term $g''$ acts like
\begin{align*}
    g''(a,b,b)&=g''(b,a,c)=g''(b,c,a)=g''(b,c,d)=g''(a,c,d)=c\\
    \shortintertext{and}
    g''(b,a,a)&=g''(b,a,d)=g''(b,d,a)=g''(b,d,c)=g''(a,d,c)=d.
\end{align*}
Hence, all minimal Taylor algebras with ${\m B}\cong {\m T}_3^{\co}$ are term equivalent to $\m c:=(A;g'')$. Any Taylor reduct of $\m c$ has a ternary cyclic term, the congruence $\alpha$ and the automorphism $\varphi$, so any minimal Taylor reduct of $\m c$ has the term $g''$. So, $\m c$ is a minimal Taylor algebra and it is isomorphic to ${\m T}_{4,5}$ $(a\mapsto 3, b\mapsto 2, c\mapsto 0, d\mapsto 1)$.
\end{proof}

\begin{claim}
If ${\m B}\cong {\m T}_7^{\co}$, then ${\m A}$ is $ {\m T}_{4,6}$, up to isomorphism and term-equivalence.
\end{claim}

\begin{proof}
The proof is the same as in the previous claim, except the term $g''$ is defined by $$g''(x,y,z)=g'(g'(x,t(x,y),t(x,y)),g'(y,t(y,z),t(y,z)),g'(z,t(z,x),t(z,x))).$$
As a consequence we get that $\m A\cong {\m T}_{4,6}$ $ (a\mapsto 3, b\mapsto 2, c\mapsto 0, d\mapsto 1)$.
\end{proof}


\subsection{Two-element majority quotient}\label{s:maj}
As before, let $\mathrm{Sg}\{a,b\}=\m a$ and assume that $\m a$ has a two-element majority quotient. Furthermore, we may suppose that none of the quotients of $\m a$ are isomorphic to the two-element semilattice. We shall prove that such a minimal Taylor algebra must be one of the three algebras from \Cref{table:7}.

\begin{table}[H]
\centering
\begin{tabular}{|c||c|c|c|}
\hline
    Algebra & $\T_{4,8}$ & $\T_{4,9}$ & $\T_{4,10}$ \\
\hline\hline
     $g{\restriction_{\{0,2\}}}$  & $\text{min}^3_{0\leqslant 2}$& aff & aff  \\
\hline
    $g{\restriction_{\{1,3\}}}$  & $\text{min}^3_{1\leqslant 3}$ & aff & aff\\

\hline
	$g(0,0,1)$ & \multirow{2}{*}{maj} & 2 & 0 \\
\cline{1-1} \cline{3-4}
	$g(0,1,1)$ &  & 3 & 3 \\	
\hline

    $g(0,0,3)$  & \multirow{2}{*}{maj} & \multirow{2}{*}{maj} & 2  \\
\cline{1-1} \cline{4-4}
	$g(0,3,3)$  &  &  & 3 \\
\hline

$g(1,1,2)$  & \multirow{2}{*}{maj} & \multirow{2}{*}{maj} & 1  \\
\cline{1-1} \cline{4-4}
	$g(1,2,2)$  &  &  & 0 \\
\hline
 
	$g(2,2,3)$ & 0 & 0 & 2 \\
\hline
	$g(2,3,3)$ & 1 & 1 & 1 \\

\hline
	$g(0,1,2)$ & 0 & 0 & 2\\
\hline
	$g(0,2,1)$ & 0 & 0 & 2 \\
\hline
	$g(0,1,3)$ & 1 & 1 & 1 \\
\hline
	$g(0,3,1)$ & 1 & 1 & 1 \\
\hline
	$g(0,2,3)$ & 0 & 2 & 0 \\
\hline
	$g(0,3,2)$ & 0 & 2 & 0 \\
\hline
	$g(1,2,3)$ & 1 & 3 & 3 \\
\hline
	$g(1,3,2)$ & 1 & 3 & 3 \\
\hline
\end{tabular} 
\caption{Four-element minimal Taylor algebras with 2-majority quotient.}
\label{table:7}
\end{table}

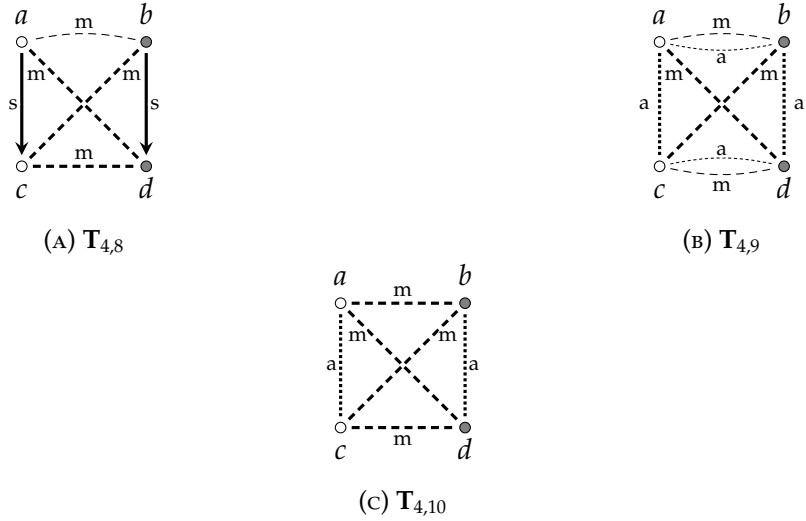
\begin{figure}[ht]
 \begin{subfigure}{0.33\textwidth}
		\centering
		\begin{tikzpicture}
			\tikzstyle{every node}=[draw,circle,fill=white,minimum size=4pt,inner sep=0pt]
			
			\node (a) [label=above:\strut$a$] {};
			\node (b) [fill=gray,right=1.5cm of a,label=above:\strut$b$] {};
			\node (c) [ below=1.5cm of a,label=below:\strut$c$] {};
			\node (d) [fill=gray,right=1.5cm of c,label=below:\strut$d$] {};

			\path[->,>=stealth,line width=1.2pt, shorten <=0.7mm, shorten >=0.7mm]
			(b) edge node[draw=none,midway,right] {\scriptsize{s}} (d)
			(a) edge node[draw=none,midway,left] {\scriptsize{s}} (c);
            \path[-,line width=1.2pt,densely dashed, shorten <=0.7mm, shorten >=0.7mm]
            (b) edge [line width=0.4pt] [bend right=12] node[draw=none,midway,above] {\scriptsize{m}} (a)
            (a) edge node[draw=none,near start, left, outer sep=3pt] {\scriptsize{m}} (d)
            (b) edge node[draw=none,near start,right, outer sep=3pt] {\scriptsize{m}} (c)
			(c) edge node[draw=none,midway,above] {\scriptsize{m}} (d);

		\end{tikzpicture}
		\caption{$\m t_{4,8}$}
	\end{subfigure}
	\hfill
     \begin{subfigure}{0.33\textwidth}
		\centering
		\begin{tikzpicture}
			\tikzstyle{every node}=[draw,circle,fill=white,minimum size=4pt,inner sep=0pt]
			
			\node (a) [label=above:\strut$a$] {};
			\node (b) [fill=gray,right=1.5cm of a,label=above:\strut$b$] {};
			\node (c) [below=1.5cm of a,label=below:\strut$c$] {};
			\node (d) [fill=gray,right=1.5cm of c,label=below:\strut$d$] {};

			\path[-,line width=1.2pt, densely dotted, shorten <=0.7mm, shorten >=0.7mm]
			(a) edge node[draw=none,midway,left,outer sep=3pt] {\scriptsize{a}} (c)
            (b) edge node[draw=none,midway,right,outer sep=3pt] {\scriptsize{a}} (d)
            (a) edge [line width=0.4pt] [bend right=12] node[draw=none,midway,below] {\scriptsize{a}} (b)
            (c) edge [line width=0.4pt] [bend left=12] node[draw=none,midway,above] {\scriptsize{a}} (d);
            \path[-,line width=1.2pt,densely dashed, shorten <=0.7mm, shorten >=0.7mm]
            (a) edge node[draw=none,near start,left,outer sep=3pt] {\scriptsize{m}} (d)
            (b) edge node[draw=none,near start,right,outer sep=3pt] {\scriptsize{m}} (c)
            (a) edge [line width=0.4pt] [bend left=12] node[draw=none,midway,above] {\scriptsize{m}} (b)
			(c) edge [line width=0.4pt] [bend right=12] node[draw=none,midway,below] {\scriptsize{m}} (d);

		\end{tikzpicture}
		\caption{$\m t_{4,9}$}
	\end{subfigure}
	\hfill
	\begin{subfigure}{0.33\textwidth}
		\centering
		\begin{tikzpicture}
			\tikzstyle{every node}=[draw,circle,fill=white,minimum size=4pt,inner sep=0pt]
			
			\node (a) [label=above:\strut$a$] {};
			\node (b) [fill=gray,right=1.5cm of a,label=above:\strut$b$] {};
			\node (c) [below=1.5cm of a,label=below:\strut$c$] {};
			\node (d) [fill=gray,right=1.5cm of c,label=below:\strut$d$] {};

			\path[-,line width=1.2pt, densely dotted, shorten <=0.7mm, shorten >=0.7mm]
            (b) edge node[draw=none,midway,right] {\scriptsize{a}} (d)
			(a) edge node[draw=none,midway,left] {\scriptsize{a}} (c);
            \path[-,line width=1.2pt,densely dashed, shorten <=0.7mm, shorten >=0.7mm]
			(b) edge node[draw=none,near start,right, outer sep=2pt] {\scriptsize{m}} (c)
            (a) edge node[draw=none,midway,above] {\scriptsize{m}} (b)
            (a) edge node[draw=none,near start,left, outer sep=2pt] {\scriptsize{m}} (d)
			(c) edge node[draw=none,midway,below] {\scriptsize{m}} (d);

		\end{tikzpicture}
		\caption{$\m t_{4,10}$}
	\end{subfigure}
	
	\caption{Directed graph of $\m t_{4,i}$ for $i=8,9,10$.}
\end{figure}

Before the analysis, we give an example from \cite{dreamteam}, which we will also be referring to in the section with $\mathbb{Z}_2^{\mathrm{aff}}$ quotient. In the present paper, the following algebra is denoted $\m t_{4,10}$.
\begin{exmp}[Example 5.16 of \cite{dreamteam}]\label{example5.16}
Let $A=\{0,1,2,3\}$ and $\alpha$ the equivalence relation on $A$ with classes $\{0,2\}$ and $\{1,3\}$. Define a symmetric operation $g$ on $A$ as follows. When two of the inputs to $g$ are equal, $g$ is given by $g(a,a,a+1)=a$, $g(a,a,a+2)=g(a,a,a+3)=a+2\; (\mathrm{all\; modulo\; 4})$ and when all three inputs to $g$ are distinct, $g$ is given by $g(a,b,c)=d-1\pmod{4}$ where $a,b,c,d$ are any permutation of $0,1,2,3$. Then $\m a=(A;g)$ is a minimal Taylor algebra, $\alpha$ is a congruence on $\m a$, and each pair of elements in different $\alpha$-classes is a majority edge with witnessing congruence $\alpha$.
\end{exmp}

\subsubsection{Equal size classes}
Suppose, without loss of generality, that $$\alpha=\{\{a,c\},\{b,d\}\}$$ is a congruence and that $\m a/\alpha$ is the two-element majority algebra. By Lemma~\ref{bkcyclic}, there exists a ternary cyclic term $g$ in $\m a$.
\begin{claim}
    There are only two possibilities regarding subuniverses $\{a,c\}$ and $\{b,d\}$:
    \begin{itemize}
        \item both are semilattices with absorbing elements $c$ and $d$, or
        \item both are affine.
    \end{itemize}
\end{claim}
\begin{proof}
If $\{a,c\}$ is a two-element majority algebra or a semilattice with the absorbing element $a$, then $\{a\}\trianglelefteq_3 {\m A}$ holds by \Cref{singletoncenter}. If $c=t(a,b)$ and $t$ has the first dominant coordinate, this would mean $t(a,b)=a$, while if $t$ has the second dominant coordinate, then $t(a,b)\in\{b,d\}$, a contradiction.

Now let $(a,c)$ be a semilattice edge and let $\{b,d\}$ be an affine algebra. The sets $\{c\}$ and $\{b,d\}$ are 3-absorbing in $\m a$, so $\{b,c,d\}$ is a subalgebra due to \Cref{composition}~(3). According to \Cref{table:1}, this 3-element algebra must be conservative, so $\{b,c\}$ is also a subalgebra which must be the 2-element majority algebra. Similar to before, we have $t(a,b)=c$ and $t(b,a)\in \{b,d\}$ for some binary term $t$. If $t(b,a)=b$, then we would know $t|_{\{a,b,c\}}$ (using that the first coordinate is dominant in $t$) and $\{a,b,c\}$ would be closed under $t$. Then $g(t(x,y),t(y,z),t(z,x))$ would be a cyclic term that preserves $\{a,b,c\}$, so $\mathrm{Sg}\{a,b\}=\{a,b,c\}$, a contradiction. 

Hence, we obtain $t(a,b)=c$ and $t(b,a)=d$. Let $t'(x,y):=g(x,t(x,y),t(x,y))$. Then
 \begin{align*}
     t'(a,b)&=g(a,t(a,b),t(a,b))=g(a,c,c)=c,\\
     t'(b,a)&=g(b,t(b,a),t(b,a))=g(b,d,d)=b,
 \end{align*}
 follows, so $\mathrm{Sg}\{a,b\}=\{a,b,c\}$, a contradiction.  
\end{proof}

\begin{claim}
If $(a,c)$ and $(b,d)$ are semilattice edges, then $\m a$ is $\m t_{4,8}$.
\end{claim}

\begin{proof}
Since $(a,c)$ and $(b,d)$ are semilattice edges, the sets $\{c\}$ and $\{d\}$ are 3-absorbing by \Cref{singletoncenter}. Furthermore, $\{a,c,d\}$, $\{b,c,d\}$, and $\{c,d\}$ are subuniverses by Proposition~\ref{composition} (3). 

If $\mathrm{Sg}\{a,d\}=\{a,c,d\}$, then $p(a,d)=c$ for some binary term $p$ with dominant first coordinate. Hence $p(d,a)=d$ and $p(a,c)=p(c,a)=c$, so $\{c,d\}\trianglelefteq_2\{a,c,d\}$. Since $\{a,b\}$ generates $\m a$ and $\{\{a,c\},\{b,d\}\}$ is its majority quotient, there exists a binary term $t$ with dominant first coordinate such that $t(a,b)=c$ and $t(b,a)\in\{b,d\}$. Moreover, $t(a,d),t(d,a),t(a,c),t(c,a)\in\{c,d\}$ since $\{c,d\}\trianglelefteq_2\{a,c,d\}$. Thus $\{b,c,d\}\trianglelefteq_2\m a$ and $\m a$ has a two-element semilattice quotient, which was analyzed in the previous section. The case $\mathrm{Sg}\{b,c\}=\{b,c,d\}$ is analogous.

Therefore, we can assume that $\{a,d\}$ and $\{b,c\}$ are majority subalgebras. This means that the subalgebras on $\{a,c,d\}$ and $\{b,c,d\}$ are both conservative and isomorphic to $\T_4^{\co}$. Hence, $\{c,d\}\trianglelefteq_3\{a,c,d\}$ and $\{c,d\}\trianglelefteq_3\{b,c,d\}$, using the ternary term given in \Cref{table:4}. This implies $\{c,d\}\trianglelefteq_3\m a$ since any relevant triple for checking this absorption is either in $\{a,c,d\}$ or in $\{b,c,d\}$.

We shall now prove that if $t(a,b)=a$ (resp.~$t(a,b)=c$) for any binary term $t$, then $t(b,a)=b$ (resp.~$t(b,a)=d$). Suppose, on the contrary, $t(a,b)=a$ and $t(b,a)=d$. Define the ternary cyclic term $g'$ by $$g'(x,y,z)=g(t(t(x,y),z),t(t(y,z),x),t(t(z,x),y)).$$ The reader can easily verify that $\{a,b,d\}$ is closed under $g'$, so $\m a$ is not minimal Taylor algebra, a contradiction. Therefore, in the rest of the proof, we assume there is an automorphism that switches the elements in each of the sets $\{a,b\}$ and $\{c,d\}$. Furthermore, equations $g(a,a,b)=c$ and $g(b,b,a)=d$ must be satisfied. Let $g''$ be a cyclic ternary term defined by $$g''(x,y,z)=g(g(x,x,y),g(y,y,z),g(z,z,x)).$$ Then $g''(a,a,b)=g(g(a,a,a),g(a,a,b),g(b,b,a))=g(a,c,d)=c$, where the last equality holds by Proposition~\ref{ternaryabscyclic}, because $\{a,c\}$ and $\{c,d\}$ are 3-absorbing. Similarly, $g''(b,b,a)=g(b,d,c)=d$ follows from the fact that $\{b,d\}$ and $\{c,d\}$ are 3-absorbing. Furthermore, we get
\begin{align*}
    g''(a,c,d)&=g''(a,d,c)=g''(a,b,c)=g''(a,c,b)=c,\\
    g''(b,c,d)&=g''(b,d,c)=g''(a,b,d)=g''(a,d,b)=d,
\end{align*}
which completely determines the cyclic term $g''$ and $(A;g'')$ is isomorphic to $\m t_{4,8}$.
Next we prove that $\{a,b\}$ is not closed for any ternary cyclic term from $\mathrm{Clo}(A;g'')$. If such a cyclic term existed, then $(a,a,a)\in S$ where 
\[
S=\mathrm{Sg}_{\m a^3}\left\{
\begin{bmatrix}
a \\
a \\
b
\end{bmatrix}\!,\!
\begin{bmatrix}
a \\
b \\
a
\end{bmatrix}\!,\!
\begin{bmatrix}
b \\
a \\
a
\end{bmatrix}
\right\}.
\]
We note, since it will be used later, that $S$ is \emph{symmetric}, that is, invariant under permutation of coordinates. Next, observe that each triple from $S$ on at least two coordinates has elements of the set $\{a,c\}$. This property follows from the fact that the same holds for the generating set of $S$ and that $g''$ acts on $\{\{a,c\},\{b,d\}\}$ as a majority operation. The only way to obtain $(a,a,a)$ from other triples in $S$ is to apply $g''$ onto $(a,a,d)$, $(a,d,a)$ and $(d,a,a)$. Hence, we will now show that these triples are not in $S$. If $(a,a,d)\in S$, then 
\[
g''\left(
\begin{bmatrix}
x_1 \\
x_2 \\
x_3
\end{bmatrix}\!,\!
\begin{bmatrix}
y_1 \\
y_2 \\
y_3
\end{bmatrix}\!,\!
\begin{bmatrix}
z_1 \\
z_2 \\
z_3
\end{bmatrix}
\right)=
\begin{bmatrix}
a \\
a \\
d
\end{bmatrix},
\]
where $(x_1,x_2,x_3),(y_1,y_2,y_3),(z_1,z_2,z_3)\in S\setminus\{(a,a,a), (d,a,a),(a,d,a),(a,a,d)\}$. We know that $x_1,x_2,y_1,y_2,z_1,z_2$ must all be in $\{a,d\}$, and since we observed that all these elements must be in $\{a,c\}$, then $x_3,y_3,z_3\in\{b,c\}$ follows as well. Hence $g''(x_3,y_3,z_3)\in\{b,c\}$, a contradiction.

As any ternary cyclic term $h$ acts on $\{(a,a,b),(b,b,a)\}$ the same as $g$, the term $h''$ we obtain from $h$ by the same construction we used to obtain $g''$ from $g$ is also completely determined and satisfies $h''=g''$. Thus $(A;g'')$ is a minimal Taylor algebra.
\end{proof}

Finally, we handle the case where both $\{a,c\}$ and $\{b,d\}$ are affine.
 \begin{claim}
  If  $\{a,c\}$ and $\{b,d\}$ are affine subalgebras, then ${\m A}$ is either $\m t_{4,9}$ or $\m t_{4,10}$.
 \end{claim}

\begin{proof}
$\mathrm{Sg}\{a,d\}=\{a,c,d\}$ is impossible, since none of non-conservative three element algebras has $\{a,c\}$ as an affine subalgebra and $\{\{a,c\},\{d\}\}$ as a majority quotient. Hence, either $\mathrm{Sg}\{a,d\}=\{a,d\}$ or $\mathrm{Sg}\{a,d\}=\m a$. 

Assume that $\mathrm{Sg}\{a,d\}=\{a,d\}$. We prove the boldface equations for any ternary cyclic term. If $g(a,a,b)=a$, then $c\notin \mathrm{Sg}\{a,b\}$, a contradiction. Thus, $\mathbf{g(a,a,b)=c}$. If $g(c,c,b)=a$, then the ternary cyclic term $g'$ given by 
$$
\begin{gathered}
g'(x,y,z)=g(g(x,g(x,y,z),g(x,y,z)),\\
g(y,g(y,z,x),g(y,z,x)),g(z,g(z,x,y),g(z,x,y)))
\end{gathered}
$$ 
satisfies $g'(a,a,b)=a$, a contradiction. It follows that $\mathbf{g(c,c,b)=c}$. 
Now, if $g(c,c,d)=c$, then $\{b,c,d\}$ is a subuniverse which cannot be 2-generated, thus $\{b,c\}$ and $\{c,d\}$ must be majority edges. Since $\mathrm{Sg}\{a,b\}=\m a$, it has to be $g(b,b,a)=d$. Let $g'':=g(g(x,y,y),g(y,z,z),g(z,x,x))$, and so
 \begin{align*}
     g''(a,c,d)&=g(g(a,c,c),g(c,d,d),g(d,a,a))=g(a,d,a)=a,\\
     g''(a,a,b)&=g(g(a,a,a),g(a,b,b),g(b,a,a))=g(a,d,c)=c,\\
     g''(b,b,a)&=g(g(b,b,b),g(b,a,a),g(a,b,b))=g(b,c,d)=d.
 \end{align*}
If we define now $g'''$ as follows $$g'''(x,y,z)=g''(g''(x,x,g''(x,x,y)),g''(y,y,g''(y,y,z)),g''(z,z,g''(z,z,x))),$$
then $g'''(a,a,b)=g''(g''(a,a,a),g''(a,a,c),g''(b,b,d))=g''(a,c,d)=a$, which is a contradiction, as $g'''(a,a,b)=c$ for any ternary cyclic term $g'''$. Thus, $\mathbf{g(c,c,d) =a}$. Next, we assume $g(b,b,c)=d$, and define the cyclic term $$\Tilde{g}(x,y,z):=g(g(x,x,g(x,y,z),g(y,y,g(y,z,x)),g(z,z,g(z,x,y)))),$$ so $\Tilde{g}(a,a,b)=a$, a contradiction. So, $\mathbf{g(b,b,c)=b}$, and thus $\mathbf{\{b,c\}}$ {\bf is a majority subalgebra}. If $g(b,b,a)=b$, then $d\notin\mathrm{Sg}\{a,b\}$, which is impossible, so, $\mathbf{g(b,b,a)=d}$. Further, if $g(d,d,c)=d$, then $\{a,c,d\}$ must be a conservative subuniverse by \Cref{table:1}, contradicting $g(c,c,d)=a$. Therefore, $\mathbf{g(d,d,c)=b}$. Now we prove $\mathbf{g(a,c,d)=g(a,d,c)=c}$. For the sake of contradiction, suppose that $g(a,c,d)=a$ or $g(a,d,c)=a$. Here we define cyclic terms $g_1(x,y,z):=g(g(x,x,y),g(y,y,z),g(z,z,x))$ and $g_2(x,y,z):=g(g(x,y,y),g(y,z,z),g(z,x,x))$, therefore we get $g_1(a,a,b)=g(a,c,d)=a$ or $g_2(a,a,b)=g(a,d,c)=a$, a contradiction. Using again $g_1$ and $g_2$, we also derive $\mathbf{g(b,c,d)=g(b,d,c)=d}$. It remains to determine $g(a,b,\ast)$ and $g(a,\ast,b)$ when $\ast\in\{c,d\}$. Let $g^*$ be 
$$
\begin{gathered}
g^*(x,y,z)=g(g(x,g(x,y,z),g(x,y,z)),\\
g(y,g(y,z,x),g(y,z,x)),g(z,g(z,x,y),g(z,x,y))).
\end{gathered}
$$
We proved that $g$ and $g^*$ match on all previously discussed triples, while it is easy to check that 
\begin{align*}
    g^*(a,b,c)&=g^*(a,c,b)=a,\\
    g^*(a,b,d)&=g^*(a,d,b)=b.
\end{align*}
So, we found a cyclic term $g^*$ that is fully determined and thus there is at most one minimal Taylor clone in this case for which $\mathrm{Sg}\{a,d\}=\{a,d\}$.

The Taylor algebra $\m c:=(A;g^*)$ has the congruence $\beta=\{\{a,d\},\{b,c\}\}$ such that $\m a/\beta$ is term equivalent to two-element affine algebra. As before, every Taylor reduct of $\m c$ must have congruences $\alpha$ and $\beta$, and that completely determines any ternary cyclic term, making it equal to $g^*$. Therefore, $\m a$ is a minimal Taylor algebra term-equivalent to $(A;g^*)$ which is isomorphic to $\m t_{4,9}$ ($a\mapsto 0, b\mapsto 1, c\mapsto 2, d\mapsto 3$).

Now we assume that $\mathrm{Sg}\{a,d\}=\m a$. Due to the previous case, we can assume that $\mathrm{Sg}\{b,c\}=\mathrm{Sg}\{c,d\}=\m a$, also hold. First, we prove that we can assume that there are elements $\bullet,\circ\in A$ in distinct $\alpha$-classes such that $g(\bullet,\bullet,\circ)=\bullet$. If that is not the case, then $g(a,a,b)=c$, $g(b,b,a)=d$, $g(b,b,c)=d$, etc. We define the term $g'$ by 
\begin{equation*}\tag{$\star$}
g'(x,y,z)=g(g(x,g(x,y,z),g(x,y,z)),g(y,g(y,z,x),g(y,z,x)),g(z,g(z,x,y),g(z,x,y))),
\end{equation*}
so $g'(a,a,b)=g(g(a,c,c),g(a,c,c),g(b,c,c))=g(a,a,a)=a$. We can use $g'$ instead of $g$, thus, we impose the assumption $\mathbf{g(a,a,b)=a}$. We can do this because our assumptions are somewhat symmetric: we know $\bullet$ and $\circ$ are in different $\alpha$-classes and $\mathrm{Sg}\{\bullet,\circ\}=\m a$, so without loss of generality we can assume $\bullet=a$ and $\circ=b$. Since $\mathrm{Sg}\{a,b\}=\m a$, our assumption $g(a,a,b)=a$ implies that $\mathbf{g(a,b,b)=d}$ and $\mathbf{g(a,a,d)=c}$ must hold. Suppose that $g(c,c,d)=a$. Using the same construction for $g'$ as in $(\star)$, we obtain
\begin{equation*}\tag{$\Box$}
\begin{aligned}
    g'(a,a,d)&=g(g(a,c,c),g(a,c,c),g(d,c,c))=g(a,a,a)=a,\\
    g'(a,a,b)&=g(g(a,a,a),g(a,a,a),g(b,a,a))=g(a,a,a)=a,
\end{aligned}
\end{equation*}
so $\{a,b,d\}$ is a subuniverse, a contradiction. Hence, $\mathbf{g(c,c,d)=c}$. Again, using $\mathrm{Sg}\{c,d\}=\m a$ we deduce that $\mathbf{g(c,d,d)=b}$ and $\mathbf{g(c,c,b)=a}$ must hold. Assuming that $g(a,d,d)=b$, take $g'$ again from $(\star)$, and we obtain
\begin{equation*}
\begin{aligned}
    g'(b,b,a)&=g(g(b,d,d),g(b,d,d),g(a,d,d))=g(b,b,b)=b,\\
    g'(a,a,b)&=g(g(a,a,a),g(a,a,a),g(b,a,a))=g(a,a,a)=a,
\end{aligned}
\end{equation*}
so $\{a,b\}$ would be a subuniverse, a contradiction. Hence, $\mathbf{g(a,d,d)=d}$. Finally, assuming that $g(b,b,c)=d$, we obtain
\begin{equation*}
\begin{aligned}
    g'(c,c,d)&=g(g(c,c,c), g(c,c,c), g(d,c,c))=g(c,c,c)=c,\\
    g'(d,d,c)&=g(g(d,b,b), g(d,b,b), g(c,b,b))=g(d,d,d)=d,
\end{aligned}
\end{equation*}
implying that $\{c,d\}$ is a subuniverse, a contradiction with $\mathrm{Sg}\{c,d\}=\m a$. So, $\mathbf{g(b,b,c)=b}$, and together with the affine assumptions, we have determined all values $g(x,y,z)$ when $|\{x,y,z\}|=2$. It is not hard to compute that $g'(x,x,y)=g(x,x,y)$ holds for all $x,y\in A$ ($g'$ taken again from $(\star)$). Now we shall determine the remaining values for $g'$:
\begin{align*}
    g'(a,b,c)&=g(a,a,c)=c,\\
    g'(a,c,b)&=g(a,c,a)=c,
\end{align*}
and similarly
\begin{align*}
    g'(a,b,d)&=g'(a,d,b)=b,\\
    g'(a,c,d)&=g'(a,d,c)=a,\\
    g'(b,c,d)&=g'(b,d,c)=d.
\end{align*}
Hence, $\m a$ must have the term $g'$ and one can verify that $(A;g')$ is isomorphic to $\m t_{4,10}$, which is a minimal Taylor algebra, as proved in \Cref{example5.16}. 
\end{proof}

\subsubsection{Non-equal sizes classes}

Suppose, without loss of generality, that $\alpha=\{\{b\},\{a,c,d\}\}$ is a congruence and that $\m a/\alpha$ is the majority quotient. Denote by $\m c$ the subalgebra with universe $C:=\{a,c,d\}$, so $\m C\trianglelefteq_3 \m a$. Since $\{b\}\trianglelefteq_3\m a$, it follows $m(b,b,x)=m(b,x,b)=m(x,b,b)=b$ for all $x\in A$. Also, because of $\mathrm{Sg}\{a,b\}=\m a$ and $m(a,b,b)=b$, we may assume $m(a,a,b)=c$ without any loss of generality.

In what follows, through a series of claims, we shall prove that no new minimal Taylor algebras satisfy the aforementioned properties.
\begin{claim}
 If $\{c,d\}\trianglelefteq_2 C$, then $\{b,c,d\}\trianglelefteq_2\m a$.   
\end{claim}

\begin{proof}
   By Proposition~\ref{composition}, we have $\{b,c,d\}\leq \m a$. Define the binary term $t(x,y):=m(x,x,y)$. Then $t(a,c),t(c,a),t(a,d),t(d,a)\in \{c,d\}$ and $t(a,b),t(b,a)\in \{b,c,d\}$, which altogether implies $\{b,c,d\}\trianglelefteq_2\m a$.
\end{proof}
According to the previous claim, if the set $\{c,d\}$ 2-absorbs $C$, then $\m a$ has the semilattice quotient, which was examined before. Henceforth, in the rest of the subsection, we suppose that $\{c,d\}$ does not binary absorb $C$.

\begin{claim}\label{claimanijeu3aps}
  If $ C'\trianglelefteq_3 C$ and $C'\neq C$, then $a\notin C'$. In particular, $\m c$ is not isomorphic to $\T^{\scriptscriptstyle{\mathsf{P}}}_1$.
\end{claim}
\begin{proof}
Suppose that $a\in C'\trianglelefteq_3 C$. Via Proposition~\ref{composition}, from $C'\trianglelefteq_3C\trianglelefteq_3\m a$ follows that $C'\trianglelefteq_3\m a$ and $C'\cup\{b\}\leq \m a$. However, the later implies $\{a,b,c,d\}=\mathrm{Sg}\{a,b\}\subseteq \{b\}\cup C'$, so $C'=\{a,c,d\}=C$.   
\end{proof}
\begin{claim}\label{claimde3apbs}
If $\{c\}\trianglelefteq_3C$ (resp.~$\{d\}\trianglelefteq_3C$), then $\{b,c\}$ (resp.~$\{b,d\}$) is a majority algebra.
\end{claim}
\begin{proof}
    Similar to the proof of the previous claim.
\end{proof}
\begin{claim}\label{claimnijepodalg}
If $(a,c)$ is a semilattice edge, then $\{b,c\}$ is not a subuniverse. 
\end{claim}
\begin{proof}
As before, assume that $\{b,c\}$ is a subalgebra. Since $\mathrm{Sg}\{a,b\}=\m a$, we have $t(a,b)=c$ for some binary term $t$, which further implies $t(b,a)=b$. Define the ternary term $m'$ by $$m'(x,y,z)=m(t(t(x,y),z),t(t(y,z),x),t(t(z,x),y)).$$
It can be easily verified that $t(a,b)=t(a,c)=t(c,a)=t(c,b)=c$, $t(b,a)=t(b,c)=b$ and $\{a,b,c\}$ is closed under $m'$. Moreover, it can be checked that $m'$ is cyclic on $\{a,b,c\}$, so $\{a,b,c\}$ is a subuniverse, contradiction.
\end{proof}

\begin{claim}
Neither of the sets $\{c\},\{d\}$ binary absorbs $C$.
\end{claim}
\begin{proof}
Suppose, on the contrary, that $\{c\}\trianglelefteq_2C$. Then, according to \Cref{smedge}, $(a,c)$ is a semilattice edge. On the other hand, due to \Cref{claimde3apbs}, since $\{c\}\trianglelefteq_3C$ (which follows from binary absorption), we have that $\{b,c\}$ is a subuniverse, but this contradicts \Cref{claimnijepodalg}.
\end{proof}
From now on, we can narrow the list of candidates for $\m c$ as it cannot have a semilattice or majority factor. Hence, $\m c$ is not isomorphic to $\T_1^{\n}, \T_2^{\n},\T_4^{\n},\T^{\scriptscriptstyle{\mathsf{S}}}_2,\T^{\co}_2,\T^{\co}_3,\T^{\co}_4,\T^{\co}_5,\T^{\co}_7,\T^{\co}_{11},\T^{\co}_{14}$, or $\T^{\co}_{15}$.

\begin{claim}\label{bar1nijemajclaim}
If $\m a$ has a cyclic ternary term $g$ and $g(a,a,b)=c$, then at least one of $\{a,c\},\{b,c\}$ is not a majority subalgebra of $\m a$.
\end{claim}
\begin{proof}
For the sake of contradiction, assume that $\{a,c\},\{b,c\}$ are majority subalgebras of $\m a$. Define a ternary term $g'$ by $$g'(x,y,z)=g(g(x,y,z),g(x,x,y),g(y,y,x)),$$ and note that
$$
g'\left(
\begin{bmatrix}
a \\
a \\
b
\end{bmatrix}\!,\!
\begin{bmatrix}
a \\
b \\
a
\end{bmatrix}\!,\!
\begin{bmatrix}
b \\
a \\
a
\end{bmatrix}
\right)=g\left(
\begin{bmatrix}
c \\
c \\
c
\end{bmatrix}\!,\!
\begin{bmatrix}
a \\
c \\
b
\end{bmatrix}\!,\!
\begin{bmatrix}
a \\
b \\
c
\end{bmatrix}
\right)=
\begin{bmatrix}
a \\
c \\
c
\end{bmatrix}.
$$
Now we define the term $w$ by $w(x,y,z)=g(x,g'(x,y,z),g'(z,x,y))$, thus we have 
$$
w\left(
\begin{bmatrix}
a \\
a \\
b
\end{bmatrix}\!,\!
\begin{bmatrix}
a \\
b \\
a
\end{bmatrix}\!,\!
\begin{bmatrix}
b \\
a \\
a
\end{bmatrix}
\right)=g\left(
\begin{bmatrix}
a \\
a \\
b
\end{bmatrix}\!,\!
\begin{bmatrix}
a \\
c \\
c
\end{bmatrix}\!,\!
\begin{bmatrix}
c \\
a \\
c
\end{bmatrix}
\right)=
\begin{bmatrix}
a \\
a \\
c
\end{bmatrix}.
$$
Finally, let $u$ be the cyclic term defined by
$$u(x,y,z)=g(w(x,y,z),w(y,z,x),w(z,x,y)).$$ Since $u(a,a,b)=a$ and $u(a,b,b)=b$, $\{a,b\}$ is a subuniverse by Proposition~\ref{subuniverse}, contradicting $\mathrm{Sg}\{a,b\}=\m a$.
\end{proof}

Denote by $\mathcal{C}$ the set consisting of following three-element minimal Taylor algebras: $$\m t_5^{\n}, \m t_1^{\scriptscriptstyle{\mathsf{S}}}, \m t_2^{\scriptscriptstyle{\mathsf{P}}}, \m t_6^{\co}, \m t_8^{\co}, \m t_{12}^{\co}, \m t_{13}^{\co}.$$
These are algebras which are strongly connected by a union of strong affine edges and semilattice edges. In the next three claims $S$ denotes the following ternary relation:
\[
S=\mathrm{Sg}_{\m a^3}\left\{
\begin{bmatrix}
a \\
a \\
b
\end{bmatrix}\!,\!
\begin{bmatrix}
a \\
b \\
a
\end{bmatrix}\!,\!
\begin{bmatrix}
b \\
a \\
a
\end{bmatrix}
\right\},
\]
while assuming that $C$ is isomorphic to some algebra from $\mathcal{C}$. Recall that $S$ is symmetric, that is, invariant under permutation of coordinates.
\begin{claim}
 $(x,y,b)\in S$ for all $x,y\in  C$.
\end{claim}
\begin{proof}
The relation $S$ is symmetric, so it suffices to prove that $(x,a,b)$, $(x,c,b)$, $(x,d,b)\in S$ when $x\in C$. First, we have $(x,a,b)\in S$ for all $x\in C$, since it follows from $(a,a,b),(b,a,a)\in S$ and $\mathrm{Sg}\{a,b\}=\m a$. Indeed, from $\mathrm{Sg}\{a,b\}=\m a$ we know that there is some binary term $t$ such that $t(a,b)=x$, thus the first coordinate of $t$ is dominant and $t(b,a)=b$ because $\{b\}$ is one $\alpha$ class and $\m a/\alpha$ is the two-element majority algebra. Analogously, if $s(a,b)=y$ we get $s(b,a)=b$ and $s(x,x)=x$, so from $(x,a,b),(x,b,a)\in S$ follows $(x,y,b)\in S$.
\end{proof}

\begin{claim}
There exists $z\in C$ such that $(x,y,z)\in S$ for all $x,y\in C$.
\end{claim}
\begin{proof}
Applying Proposition~\ref{ternaryabscyclic} to a cyclic term of $\m a$ and $C\trianglelefteq_3 \m A$, we get $(x',y',z')\in S$ for some $x',y',z'\in C$. We shall observe $\mathrm{Sg}\{b,z'\}$. From the previous claim, we have $(x',y',b)\in S$, so $(x',y',u)\in S$ for all $u\in\mathrm{Sg}\{b,z'\}$. Now we distinguish a few cases regarding $\mathrm{Sg}\{b,z'\}$.

  Suppose first $\mathrm{Sg}\{b,z'\}=\{b,z'\}$, so $\{b,z'\}$ is a majority subalgebra. We shall prove that $z'$ has the asserted property. First we prove $(x',y,z')\in S$ for all $y\in C$. Let $(y',y'')$ be a semilattice or $\{y',y''\}$ a strong affine edge. By applying $m$ to the triples $(x',y'',b),(x',y',z'),(x',y',z')\in S$, we get $(x',y'',z')\in S$. After at most one more such step, we obtain $(x',y,z')\in S$ for any $y\in C$. Fixing now $y$ and working in a similar fashion, we get $(x,y,z')\in S$ for any $x\in C$.

  If $\m b:=\mathrm{Sg}\{b,z'\}=\{b,c,d\}$, then $\m b$ must be the algebra $\m t_1^{\n}$. Therefore, there is $z''\in\{c,d\}$ such that $(z',z'')$ is a semilattice edge and $\{b,z''\}$ is a majority algebra. Applying $m(x,x,y)$ to $x=(x',y',z')$ and $y=(b,y',z'')$ we get $(x'',y',z'')\in S$ for some $x''\in C$ and thus we are done by the previous case.

  Let us now assume $\mathrm{Sg}\{b,z'\}=\m a$. Then $(x',y',a),(x',y',c),(x',y',d)\in S$ and let $u:=m(b,a,a)\in \{c,d\}$. If $(y',y'')$ is a semilattice edge, or $\{y',y''\}$ a strong affine edge, we obtain $(x',y'',u)\in S$ by applying $m$ to the triples $(x',y'',b)$, $(x',y',a)$, $(x',y',a)$. There is semilattice or strong affine edge from $u$ to $v$, for some $v\in C\setminus \{u\}$. Applying $m$ to $(x',y'',u),(x',y',u),(x',y',v)\in S$, we get $(x',y'',v)\in S$. Using the same argument again (if necessary) on another semilattice or strong affine edge out of $v$, we derive that $(x',y'',z)\in S$ for all $z\in C$. The same argument from the start, replacing $y'$ with $y''$, gives us, for any semilattice edge $(y'',y''')$ or strong affine edge $\{y'',y'''\}$, and any $z\in C$, that $(x',y''',z)\in S$, so actually for all $y,z\in C$, $(x',y,z)\in S$. Since $S$ is symmetric, for all $y,z\in C$, we have $(y,z,x')\in S$, so we are done.\end{proof}

\begin{claim}\label{majorityexists}
$(a,a,a)\in S$.
\end{claim}
\begin{proof}
Let $z\in C$ be such that $(x,y,z)\in S$ holds for all $x,y\in C$. Let $(z,z')$ be a semilattice, or $\{z,z'\}$ a strong affine edge. By the symmetry of $S$ we know $(z',a,z),(z,a,z),(z,a,z')\in S$, so applying $m$ to these three we obtain $(z',a,z')\in S$. If $a$ can be selected to be $z'$ we are done, otherwise $(z',a)$ is a semilattice, or $\{z',a\}$ a strong affine edge and by applying $m$ to $(a,z',z'),(z',a,z'),(z',z',a)$ we obtain $(a,a,a)\in S$.
\end{proof}

By Claim~\ref{majorityexists}, and the argument from the end of the proof of Claim~\ref{bar1nijemajclaim}, there exists a term $t(x,y,z)$ which acts as the majority on $\{a,b\}$, a contradiction. Therefore, we have eliminated all algebras of $\mathcal{C}$ as candidates for the subalgebra $\m c$. At the end of the analysis, we have to verify $\m c$ cannot be isomorphic to any of the remaining algebras $\m t^{\n}_3, \m t^{\co}_1, \m t^{\co}_9$ and $\m t^{\co}_{10}$. The following claim concludes the proof.

\begin{claim}
    $\m c$ cannot be isomorphic to $\m t_3^{\n}$, nor to $\m t_i^{\co}$ for any $i\in\{1,9,10\}$.
\end{claim}
\begin{proof}
Suppose, on the contrary, that $\m c\cong \m t_3^{\n}$. Some cyclic ternary term $g$ exists, completely determined on the set $C$. Also, since $\{0,2\}\trianglelefteq_3\T_3^{\n}$, by Claim~\ref{claimanijeu3aps} the isomorphism must map $a\mapsto 1$, thus $\{c,d\}\mapsto\{0,2\}$. Let, without loss of generality, $c\mapsto 2$ and $d\mapsto 0$. Since $\{b\}\trianglelefteq_3\m a$ and $\{c,d\}\trianglelefteq_3\m a$, by Proposition~\ref{composition}, $\{b,c,d\}$ is a subuniverse. Moreover, $\{b,c,d\}$ must be the universe of a conservative algebra, because the only algebra from \Cref{table:1} that has a majority factor also has a semilattice factor. Since $\{a,c\}$ and $\{b,c\}$ are majority algebras, due to Claim~\ref{bar1nijemajclaim}, $g(a,a,b)=d$ follows. Define the cyclic ternary term $g'$ by 
\[
\begin{gathered}
g'(x,y,z)=g(g(g(x,y,z),g(x,y,z),g(x,x,y)),g(g(y,z,x),g(y,z,x),g(y,y,z)),\\
g(g(z,x,y),g(z,x,y),g(z,z,x))).
\end{gathered}
\]
Now, $g'(a,a,b)=g(g(d,d,a),g(d,d,d),g(d,d,b))=g(c,d,d)=c$, a contradiction with Claim \ref{bar1nijemajclaim}.

Suppose that $\m c$ is isomorphic to $\m t_1^{\co}$. As $\{0,1\}$ is 3-absorbing in $\m t_1^{\co}$, we derive $a\mapsto 2$, by Claim~\ref{claimanijeu3aps}. Let, without loss of generality, $c\mapsto 0$ and $d\mapsto 1$. From $\{c\}\trianglelefteq_3 C$ and Claim~\ref{claimde3apbs} we get that $\{b,c\}$ is a majority subalgebra, as is $\{a,c\}$ by $\m t_1^{\co}$ from \Cref{table:4}. Since $\{a,b\}$ is a generating set, there is a binary term $t$ such that $t(a,b)=c$. Then we also have $t(a,c)=a$, $t(b,a)=t(b,c)=b$ and $t(c,b)=t(c,a)=c$. Define the cyclic term $g'$ by 
\begin{equation*}\tag{$\star$}
g'(x,y,z)=g(t(t(x,y),z),t(t(y,z),x),t(t(z,x),y)).
\end{equation*}
 Now $g'(a,a,b)=g(t(a,b),t(c,a),t(b,a))=g(c,c,b)=c$, which, by Claim~\ref{bar1nijemajclaim}, implies that one of subalgebras $\{a,c\}$, $\{b,c\}$ is not majority, a contradiction.

Let now $\m c\cong\m t_9^{\co}$. From $\{0\}\trianglelefteq_3\m t_9^{\co}$ and  Claim~\ref{claimanijeu3aps} we get $a\mapsto\{1,2\}$. Let $c\mapsto 0$. By Proposition~\ref{composition} (4), $\{c\}\trianglelefteq_3\m a$, so by Proposition~\ref{composition} (3), $\{b,c\}$ is a subuniverse. Let $t(a,b)=c$ for some binary term $t$, which further implies $t(c,a)=t(c,b)=c$ and $t(b,a)=t(b,c)=b$. We know that $\m a$ has a ternary cyclic term, call it $g$ and let $g'$ be the cyclic term obtained from $g$ like in $(\star)$. We get $g'(a,a,b)=g(t(a,b),t(c,a),t(b,a))=g(c,c,b)=c$. Then, by Claim~\ref{claimnijepodalg}, we get $a\not\mapsto 1$, while $a\not\mapsto 2$ follows from Claim~\ref{bar1nijemajclaim}, a contradiction.

Finally, suppose that $\m c\cong\m t_{10}^{\co}$. We have $\{0,2\}\trianglelefteq_3 \m t_{10}^{\co}$, so, by Claim~\ref{claimanijeu3aps}, $a\mapsto 1$. Without loss of generality, let $c\mapsto 0$. As before, $\{b,c,d\}$ is the universe of a conservative subalgebra, as no algebra in \Cref{table:1} has a majority factor. Thus $\{b,c\}$ is a subalgebra and $(a,c)$ is a semilattice edge, contradicting Claim~\ref{claimnijepodalg}.
\end{proof}

\subsection{Affine $\mathbb{Z}^{\text{aff}}_2$ quotient}\label{s:Z2}

We will prove that any 2-generated, 4-element minimal Taylor algebra $\m a$ which has a factor isomorphic to $\mathbb{Z}^{\text{aff}}_2$ must be term equivalent to one of the algebras listed in \Cref{table 2-affine}.

\begin{table}[!ht]
\centering
 \begin{tabular}{|c||c|c|c|}
\hline
    Algebra & $\T_{4,11}$ & $\T_{4,12}$ & $\T_{4,13}$\\
\hline\hline
     $g{\restriction_{\{0,2\}}}$  & maj &  aff & aff \\
\hline
    $g{\restriction_{\{1,3\}}}$  & maj &  aff& aff\\

\hline
	$g(0,0,1)$ & 3 &  3 & 3\\
\hline
	$g(0,1,1)$ & 2 &  2 & 0\\	
\hline

    $g(0,0,3)$  & 3 &  1 & 1\\
\hline
	$g(0,3,3)$  & 2 &  2 & 0\\
\hline

        $g(1,1,2)$  & 2 &  0 & 2\\
\hline
	$g(1,2,2)$  & 3 &  3 & 3\\
\hline
 
	$g(2,2,3)$ & \multirow{2}{*}{aff} &  1 & 1\\
\cline{1-1} \cline{3-4}
	$g(2,3,3)$ &  &  0 & 2 \\

\hline
	$g(0,1,2)$ & 3 &  1 & 1 \\
\hline
	$g(0,2,1)$ & 3 &  1 & 1 \\
\hline
	$g(0,1,3)$ & 2 &  0 & 2 \\
\hline
	$g(0,3,1)$ & 2 &  0 & 2\\
\hline
	$g(0,2,3)$ & 3 &  3 & 3 \\
\hline
	$g(0,3,2)$ & 3 &  3 & 3 \\
\hline
	$g(1,2,3)$ & 2 &  2 & 0 \\
\hline
	$g(1,3,2)$ & 2 &  2 &  0\\
\hline
\end{tabular} 
\caption{Four-element minimal Taylor algebras with $\mathbb{Z}^{\text{aff}}_2$ quotient.}
\label{table 2-affine}
\end{table}

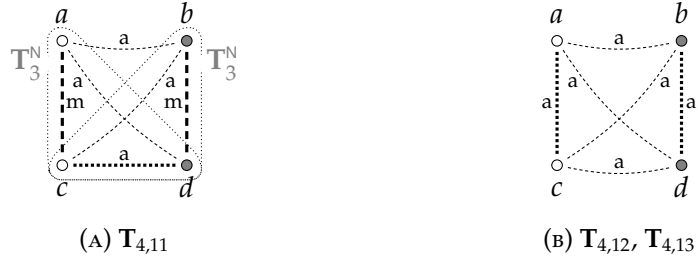
\begin{figure}[ht]
    \begin{subfigure}{0.48\textwidth}
		\centering
		\begin{tikzpicture}
			\tikzstyle{every node}=[draw,circle,fill=white,minimum size=4pt,inner sep=0pt]
			
			\node (a) [label=above:\strut$a$] {};
			\node (b) [fill=gray,right=1.5cm of a,label=above:\strut$b$] {};
			\node (c) [ below=1.5cm of a,label=below:\strut$c$] {};
			\node (d) [fill=gray,right=1.5cm of c,label=below:\strut$d$] {};

             \draw[thin, densely dotted, rounded corners] ([shift={(0.14,0.1)}]b.north east)--([shift={(-0.14,0.14)}]b.north east)--([shift={(-0.14,0.1)}]c.south west)--([shift={(-0.1,-0.14)}]c.south west)--([shift={(0.14,-0.14)}]d.south east)--cycle;
             \node[draw=none] at ([shift={(0.4,-0.3)}]b.east) {$\color{gray} \m t_3^{\n}$};

             \draw[thin, densely dotted, rounded corners] ([shift={(0.14,0.14)}]a.north west)--([shift={(-0.14,0.1)}]a.north west)--([shift={(-0.14,0.1)}]c.south west)--([shift={(-0.1,-0.14)}]c.south west)--([shift={(0.14,-0.14)}]d.south east)--([shift={(0.14,0.14)}]d.south east)--cycle;
             \node[draw=none] at ([shift={(-0.4,-0.3)}]a.west) {$\color{gray} \m t_3^{\n}$};
			
			\path[-,line width=1.2pt, densely dotted, shorten <=0.7mm, shorten >=0.7mm]
			(a) edge [line width=0.4pt] [bend right=12] node[draw=none,near start,left,outer sep=2pt] {\scriptsize{a}} (d)
            (b) edge [line width=0.4pt] [bend left=12] node[draw=none,near start,right,outer sep=2pt] {\scriptsize{a}} (c)
            (a) edge [line width=0.4pt] [bend right=12] node[draw=none,midway,above, outer sep=1pt] {\scriptsize{a}} (b)
            (c) edge node[draw=none,midway,above, outer sep=1pt] {\scriptsize{a}} (d);
            \path[-,line width=1.2pt,densely dashed, shorten <=0.7mm, shorten >=0.7mm]
			(b) edge node[draw=none,midway,left, outer sep=1pt] {\scriptsize{m}} (d)
            (a) edge node[draw=none,midway,right, outer sep=1pt] {\scriptsize{m}} (c);

			
		\end{tikzpicture}
		\caption{$\m t_{4,11}$}
		
	\end{subfigure}%
	\hfill
	\begin{subfigure}{0.48\textwidth}
		\centering
		\begin{tikzpicture}
			\tikzstyle{every node}=[draw,circle,fill=white,minimum size=4pt,inner sep=0pt]
			
			\node (a) [label=above:\strut$a$] {};
			\node (b) [fill=gray,right=1.5cm of a,label=above:\strut$b$] {};
			\node (c) [below=1.5cm of a,label=below:\strut$c$] {};
			\node (d) [fill=gray,right=1.5cm of c,label=below:\strut$d$] {};

            \path[-,line width=1.2pt,densely dotted, shorten <=0.7mm, shorten >=0.7mm]
            (a) edge [line width=0.4pt] [bend right=12] node[draw=none,midway,above, outer sep=1pt] {\scriptsize{a}} (b)
			(b) edge [line width=0.4pt] [bend left=12] node[draw=none,near start,right, outer sep=2pt] {\scriptsize{a}} (c)
            (a) edge [line width=0.4pt] [bend right=12] node[draw=none,near start,left, outer sep=2pt] {\scriptsize{a}} (d)
			(c) edge [line width=0.4pt] [bend right=12] node[draw=none,midway,above,outer sep=1pt] {\scriptsize{a}} (d)
            (b) edge node[draw=none,midway,right,outer sep=1pt] {\scriptsize{a}} (d)
			(a) edge node[draw=none,midway,left,outer sep=1pt] {\scriptsize{a}} (c);
		\end{tikzpicture}
		\caption{$\m t_{4,12}$, $\m t_{4,13}$}
	\end{subfigure}
	\caption{Directed graph of $\m t_{4,i}$ for $i=11,12,13$.}
\end{figure}

\begin{lem}\label{aaalemma}
Let $\m a$ have an automorphism of order two that switches $a$ and $b$. If $S$ is a ternary relation defined as  
$$
S=\mathrm{Sg}_{\m a^3}\left\{
\begin{bmatrix}
a \\
b \\
b
\end{bmatrix}\!,\!
\begin{bmatrix}
b \\
a \\
b
\end{bmatrix}\!,\!
\begin{bmatrix}
b \\
b \\
a
\end{bmatrix}
\right\},
$$ then $(a,a,a)\notin S$.
\end{lem}

\begin{proof}
 On the contrary, suppose that $(a,a,a)\in S$. Then $h(a,b,b)=h(b,a,b)=h(b,b,a)=a$ for some ternary term $h$. Notice that now $h(b,a,a)=h(a,b,a)=h(a,a,b)=b$ follows, since $a$ and $b$ are switched by some automorphism. Then $\{a,b\}$ is a $\mathbb{Z}^{\text{aff}}_2$ subalgebra, a contradiction.
\end{proof}

\subsubsection{Non-equal sizes classes $\mathbb{Z}^{\text{aff}}_2$ quotient}

Let $\alpha=\{\{b\},\{a,c,d\}\}$ be a congruence such that $\m a/\alpha$ is the affine quotient. Denote by $D:=\{a,c,d\}$ and by $\m d$ the subalgebra with universe $D$. Note that $m(x,y,b)=b$ holds for all $x,y\in D$. Also, since $\mathrm{Sg}\{a,b\}=\m a$ and $m(a,a,b)=b$, one of the values $m(a,b,b),m(b,a,b),m(b,b,a)$ must not be equal to $a$. Additionally, according to the remark made after \Cref{compositionactslike}, we have $m(x,y,y)=m(y,y,x)$ for all $x,y$. In some cases, it will be helpful to observe the previously defined ternary relation $S$. We state the following beneficial fact, which cannot be derived straightforwardly from Lemma~\ref{aaalemma} since we do not assume the existence of an automorphism.
\begin{claim}\label{aaaclaim}
$(a,a,a)\notin S$.
\end{claim}
\begin{proof}
Assume $(a,a,a)\in S$. Then $w(a,b,b)=w(b,a,b)=w(b,b,a)=a$ for some ternary term $w$. Since $\m a/\alpha$ has an automorphism switching $a/\alpha$ and $b/\alpha$, we have $w(b,a,a)/\alpha=w(a,b,a)/\alpha=w(a,a,b)/\alpha=b/\alpha$, thus $w(b,a,a)=w(a,b,a)=w(a,a,b)=b$ and the proof concludes the same way as the proof of Lemma~\ref{aaalemma}.
\end{proof}

As in the case of a two-element majority quotient with non-equal size classes, we shall prove that there are no new minimal Taylor algebras.

\begin{claim}\label{D'bissubuniverse}
If $D'\trianglelefteq_2 \m d$ and $D'\neq D$, then $D'\cup\{b\}$ is a subuniverse, and therefore, $a\notin D'$.
\end{claim}
\begin{proof}
Each three-element minimal Taylor algebra with a proper 2-absorbing set has a ternary cyclic term, so by Lemma~\ref{bkcyclic}, it also exists in $\m a$. Let $g$ be a ternary cyclic term of $\m a$. Define the new cyclic term $g'$ by $$g'(x,y,z)=g(g(x,x,g(x,y,z)),g(y,y,g(y,z,x)),g(z,z,g(z,x,y))).$$ Clearly, sets $D'$ and $\{b\}$ are closed under $g'$, and $g'(b,x,y)=b$ when $x,y\in D'$, so we shall only check if $g'(b,b,x)\in D'$ when $x\in D'$. We have $$g'(b,b,x)=g(g(b,b,g(b,b,x)),g(b,b,g(b,x,b)),g(x,x,g(x,b,b))),$$ which must belong to $D'$ due to the item (1) of \Cref{bradybinaryprop} (applied to $g(b,b,g(b,b,x))\in D$ and $g(x,x,g(x,b,b))\in D'$).
\end{proof}
The next claim says that if $\{c,d\}$ is 2-absorbing in $D$, then $\m a$ also has a semilattice quotient.
\begin{claim}
If $\{c,d\}\trianglelefteq_2 D$, then $\{b,c,d\}\trianglelefteq_2\m a$.
\end{claim}
\begin{proof}
Due to the previous claim, $\{b,c,d\}$ is a subuniverse. Since $\m a/\alpha$ is the affine quotient, there exists a binary term $t$ that satisfies $t(a,b)\in\{c,d\}$ and $t(b,a)=b$. From $\{c,d\}\trianglelefteq_2 D$ we have $t(a,c),t(a,d),t(c,a),t(d,a)\in\{c,d\}$, which concludes the proof.
\end{proof}

\begin{claim}
    Neither of the sets $\{c\},\{d\}$ binary absorbs $D$.
\end{claim}
\begin{proof}
Suppose, on the contrary, that $\{c\}\trianglelefteq_2 D$. Then, by Claim~\ref{D'bissubuniverse}, $\{b,c\}$ is an (affine) subalgebra.  Let $t$ be a binary term such that $t(a,b)=c$, thus $t(b,a)=b$. The absorption implies $t(a,c)=t(c,a)=c$, while because of $\{b,c\}$ being a subalgebra we may assume $t(b,c)=b$ and $t(c,b)=c$. Let $g'$ be a cyclic term of arity $p$, where $p$ is large enough. Define a new cyclic term $g''$ of the same arity as 
\begin{align*}
g''(x_1,\dotso, x_p) = g'(&t(\dotso t(t(x_1,x_2),x_3),\dotso x_p), t(\dotso t(t(x_2,x_3),x_4),\dotso x_1),\dotso ,\\
& t(\dotso t(t(x_p,x_1),x_2),\dotso x_{n-1})).
\end{align*}
Now it is not hard to see that $\{a,b,c\}$ is closed under $g''$, contradicting the assumption that $\mathrm{Sg}\{a,b\}=\m a$
\end{proof}
In other words, throughout the rest of the subsection, we may assume that no subset of $D$ is 2-absorbing.
\begin{claim}
$\m d$ cannot have $\mathbb{Z}_2^{{\mathrm{aff}}}$ as quotient.
\end{claim}
\begin{proof}
Let us suppose, for the sake of contradiction, that $\m d$ has an affine quotient $\m d/\beta$, where $\beta=\{D_1,D_2\}$ and $|D_1|=1$.  Here, we can assume that $\m a$ has a ternary cyclic term $g$ since $\m d$ does. We have two cases depending on whether the element $a$ is in $D_1$ or $D_2$.

First, let $D_1=\{a\}$. Let, without loss of generality, $g(a,b,b)=c$. Define the ternary term $w$ by $w(x,y,z)=g(x,g(x,y,y),g(x,y,z))$ and note that $S$ is symmetric and $g$ is the minority operation on $\m d/\beta$. We compute
$$
\begin{gathered}
w\left(
\begin{bmatrix}
a \\
b \\
b
\end{bmatrix}\!,\!
\begin{bmatrix}
b \\
a \\
b
\end{bmatrix}\!,\!
\begin{bmatrix}
b \\
b \\
a
\end{bmatrix}
\right)=g\left(
\begin{bmatrix}
a \\
b \\
b
\end{bmatrix}\!,\!
\begin{bmatrix}
c \\
b \\
b
\end{bmatrix}\!,\!
\begin{bmatrix}
c \\
c \\
c
\end{bmatrix}
\right)=
\begin{bmatrix}
a \\
c' \\
c'
\end{bmatrix},\text{ and hence }\\
g\left(
\begin{bmatrix}
a \\
c' \\
c'
\end{bmatrix}\!,\!
\begin{bmatrix}
c' \\
a \\
c'
\end{bmatrix}\!,\!
\begin{bmatrix}
c' \\
c' \\
a
\end{bmatrix}
\right)=
\begin{bmatrix}
a \\
a \\
a
\end{bmatrix}\in S,
\text{ contradicting Claim~\ref{aaaclaim}.}
\end{gathered}
$$

Assume now that $a\in D_2$ and let, without loss of generality, $D_2=\{a,c\}$, and thus $D_1=\{d\}$. Let $w$ be defined as in the previous paragraph and let $g':=g(w(x,y,z),w(y,z,x),w(z,x,y))$. Therefore, $g'(b,a,a)=g'(b,c,c)=g'(b,a,c)=g'(b,c,a)=b$ and $g'(a,a,c),g'(a,c,c)\in D_2$. Further we have
$$
w\left(
\begin{bmatrix}
a \\
b \\
b
\end{bmatrix}\!,\!
\begin{bmatrix}
b \\
a \\
b
\end{bmatrix}\!,\!
\begin{bmatrix}
b \\
b \\
a
\end{bmatrix}
\right)=g\left(
\begin{bmatrix}
a \\
b \\
b
\end{bmatrix}\!,\!
\begin{bmatrix}
g(a,b,b) \\
b \\
b
\end{bmatrix}\!,\!
\begin{bmatrix}
g(a,b,b) \\
g(a,b,b) \\
g(a,b,b)
\end{bmatrix}
\right)=
\begin{bmatrix}
\ast \\
\circ \\
\circ
\end{bmatrix},
$$
where $\circ\in D$ and $\ast\in\{a,c\}$, so $g'(a,b,b)=g(\ast,\circ,\circ)\in D_2$. Similarly, we get
$$
w\left(
\begin{bmatrix}
c \\
b \\
b
\end{bmatrix}\!,\!
\begin{bmatrix}
b \\
c \\
b
\end{bmatrix}\!,\!
\begin{bmatrix}
b \\
b \\
c
\end{bmatrix}
\right)=c\left(
\begin{bmatrix}
c \\
b \\
b
\end{bmatrix}\!,\!
\begin{bmatrix}
g(c,b,b) \\
b \\
b
\end{bmatrix}\!,\!
\begin{bmatrix}
g(c,b,b) \\
g(c,b,b) \\
g(c,b,b)
\end{bmatrix}
\right)=
\begin{bmatrix}
\star \\
\bullet \\
\bullet
\end{bmatrix},
$$
where $\bullet\in D$, so $g'(c,b,b)=g(\star,\bullet,\bullet)\in D_2$. Therefore, $d\notin\mathrm{Sg}\{a,b\}$, a contradiction.
\end{proof}

\begin{claim}\label{nocons}
$\m d$ cannot be a conservative algebra with a ternary cyclic term.
\end{claim}

\begin{proof}
By the way of contradiction, suppose that $\m d$ is conservative and has a ternary cyclic term $g$. First, since $c,d\in\mathrm{Sg}\{a,b\}$, we may assume $g(a,b,b)=c$ and $g(c,b,b)=d$. Also, $g((a,b,b),(b,a,b),(b,a,b))=(c,b,b)\in S$. Using the symmetry of $S$ we also conclude that $g((c,b,b),(b,c,b),(b,c,b))=(d,b,b)\in S$, so $(x,b,b)\in S$ for any $x\in D$. From that we also conclude $g((c,b,b),(b,c,b),(b,b,c))=(d,d,d)\in S$.

If $g(d,b,b)=a$, then from $(d,b,b), (b,d,b),(b,b,d)\in S$ follows $(a,a,a)\in S$ after applying $g$, which is in a contradiction with Claim~\ref{aaaclaim}. If $g(d,b,b)=c$, then $\{b,c,d\}$ is a subuniverse since $\{c,d\}$ is, and $\mathrm{Sg}\{b,c\}=\mathrm{Sg}\{b,d\}=\{b,c,d\}$. But no such algebra exists, according to \Cref{table:1}. We conclude that $g(d,b,b)=d$ must hold.

Next, we prove that $\{a,d\}$ is a majority subalgebra or $(a,d)$ a semilattice edge. Suppose, on the contrary, that $\{a,d\}$ is an affine subalgebra or $(d,a)$ a semilattice edge. Then $g((d,d,d),(a,b,b),(d,b,b))=(a,d,d)\in S$. Finally, $g((a,d,d),(d,a,d),(d,d,a))=(a,a,a)\in S$, a contradiction. 

Therefore, we have that $(a,d)$ is a semilattice edge or $\{a,d\}$ is a majority subalgebra. Nevertheless, since $\mathrm{Sg}\{a,b\}=\m a$, $t(a,b)=d$ for some binary term $t$. As $\m a/\alpha$ has an automorphism switching $a/\alpha$ and $b/\alpha$, $t(b,a)=b$. Let $g'$ be the ternary cyclic term defined by \[g'(x,y,z)=g(g(x,t(x,y),t(x,y)),g(y,t(y,z),t(y,z)),g(z,t(z,x),t(z,x))).\]
Then one can easily check that $g'(a,b,b)=d$ holds, so $\{a,b,d\}$ is closed under $g'$, which is a contradiction. 
\end{proof}

We proved that $\m d$ must be term equivalent to an isomorphic copy of one of these four algebras: $\mathbb{Z}^{\text{aff}}_3$, $\m t_1^{\scriptscriptstyle{\mathsf{S}}}$, $\m t_1^{\scriptscriptstyle{\mathsf{P}}}$, $\m t_2^{\scriptscriptstyle{\mathsf{P}}}$. To conclude the subsection, we exclude all four possibilities.

\begin{claim}\label{Daff}
$\m d$ cannot be term equivalent to a copy of $\mathbb{Z}^{\mathrm{aff}}_3$ or $\m t_2^{\scriptscriptstyle{\mathsf{P}}}$.
\end{claim}
\begin{proof}
Let us suppose the opposite, hence $\m d$ is a Mal'cev subalgebra. We prove that $\m a$ has no semilattice edges, nor weak majority edges. Working by definition, it clear that there can be no semilattice or weak majority edges between two elements of $D$. Moreover, no semilattice edge can connect $b$ and a member of $D$, in either direction, since this would be a two-element algebra $\m e$ with binary absorption, and we know that $\m e$ would also inherit the structure $\mathbb{Z}^{\mathrm{aff}}_2$ from $\m a/\alpha$. If there is a weak majority edge $\{x,b\}$, where $x\in D$, let $\m e=\mathrm{Sg}\{x,b\}$ and $\theta$ be the witnessing congruence. If $m$ is the term from Theorem~\ref{compositionactslike}, then $m(x,x,b)\in b/\alpha=\{b\}$ and $m(x,x,b)\in x/\theta$, a contradiction. Thus by Theorem~\ref{Malcevcriterionedges}, $\m a$ has a Mal'cev term $p$. Let $S$ be the previously defined ternary relation. If $\circ\in D$, then $t(a,b)=\circ$ and $t(b,a)\in b/\alpha=\{b\}$ for some binary term $t$, so we get $t((a,b,b),(b,a,b))=(\circ,b,b)\in S$ (thus, $(b,\circ,b),(b,b,\circ)\in S$).  Further, by applying $p$ to the generating triples of $S$, we get $(a,\ast,a)\in S$, where $\ast=p(b,a,b)$. Of course, if $\ast=a$, we are done by contradiction with Claim~\ref{aaaclaim}. The case $\ast=b$ is not possible because $p(b,a,b)\neq b$ due to the $\mathbb{Z}^{\mathrm{aff}}_2$ quotient. In the case $\ast\in\{c,d\}$ we can use $(b,c,b),(b,d,b)\in S$ to get $p((a,\ast,a),(b,\ast,b),(b,a,b))=(a,a,a)\in S$, contradicting again Claim~\ref{aaaclaim}.
\end{proof}

\begin{claim}
$\m d$ cannot be term equivalent to a copy of $\m t_1^{\scriptscriptstyle{\mathsf{S}}}$.
\end{claim}

\begin{proof}
Suppose, for the sake of contradiction, $\m d\cong\m t^{\scriptscriptstyle{\mathsf{S}}}_1$ and let $a$ absorb $d$. Similar to before, we get $(\circ,b,b)\in S$ for all $\circ\in D$. Let $(\bullet,\ast,\bullet)\in S$ be the result of applying $m$ to a generating triple of $S$. We discuss several cases, depending on what $\bullet$ and $\ast$ are equal to. 

If $\bullet=\ast=a$, then $(a,a,a)\in S$ and we have a contradiction by Claim~\ref{aaaclaim}.

If $\bullet=\ast=d$, then $(a,\diamond,\diamond)\in S$ is implied by $(a,b,b),(d,d,d)\in S$,  where $\diamond=m(b,b,d)$. Of course, if $\diamond=a$, the proof is done, so we assume $\diamond\neq a$. In the case of $\diamond=d$, we get $(a,a,a)\in S$ from $(a,d,d),(d,a,d),(d,d,a)\in S$ and the fact that $m(a,d,d)=a$ holds. On the other hand, if $\diamond=c$, that is $(a,c,c)\in S$, then we apply $m$ to $(a,c,c)$ and $(d,d,d)$, so we get $(a,d,d)\in S$. This reduces to the previously discussed option, so we again reach a contradiction.

Suppose now $\bullet=\ast=c$. By applying $m$ to $(d,b,b),(d,b,b)$ and $(c,c,c)$ we obtain $(d,\star,\star)\in S$, where $\star=m(b,b,c)\in D$. If $\star=a$, we have $(d,a,a)\in S$, so $m((d,a,a),(a,d,a),(a,a,d))=(a,a,a)\in S$. Let us now assume $\star=d$, so $(d,d,d)\in S$. Now we use $m$ on the triples $(a,b,b),(a,b,b),(d,d,d)\in S$, thus we get $(a,\triangledown,\triangledown)\in S$. We may assume that $\triangledown=c$, as otherwise $(a,a,a)\in S$ or $(a,d,d)\in S$ and we get a contradiction as before. We now apply any binary operation which is not projection on the triples $(a,c,c),(d,d,d)\in S$. Thus we get $(a,d,d)\in S$, which implies $(a,a,a)\in S$ as before, a contradiction. Finally, if $\star=c$, we get $m((d,c,c),(c,d,c),(c,c,d))=(d,d,d)\in S$, and we proceed as above.

In the case $\{\bullet,\ast\}=\{a,d\}$, we get $m((a,a,d),(a,d,a),(d,a,a))=(a,a,a)\in S$ or $m((a,d,d),(d,a,d),(d,d,a))=(a,a,a)\in S$, a contradiction.

If $\{\bullet,\ast\}=\{a,c\}$, then $(c,c,c)\in S$, and therefore $(a,a,a)\in S$ by the third case, a contradiction.

Finally, we deal with the case when $\{\bullet,\ast\}=\{c,d\}$. Here we get $(d,d,d)\in S$, leading to a contradiction due to the second case.
\end{proof}

\begin{claim}
$\m d$ cannot be term equivalent to a copy of $\m t_1^{\scriptscriptstyle{\mathsf{P}}}$.  
\end{claim}

\begin{proof}
We shall first prove that $\mathrm{Sg}\{b,c\}=\{b,c\}$ (resp.~$\mathrm{Sg}\{b,d\}=\{b,d\}$) implies $\mathrm{Sg}\{a,b\}=\{a,b,c\}$ (resp.~$\mathrm{Sg}\{a,b\}=\{a,b,d\}$). Suppose $\mathrm{Sg}\{b,c\}=\{b,c\}$ and let $t(a,b)=c$ (thus, $t(b,a)=b$) for some binary term $t$. Define the term $m'$ by $$m'(x,y,z)=m(m(x,t(x,y),t(x,y)),m(y,t(y,z),t(y,z)),m(z,t(z,x),t(z,x))).$$ Then $m'(a,b,b)=m(m(a,c,c),m(b,b,b),m(b,b,b))=m(c,b,b)=c$, and similarly, $m'(b,a,b)=m'(b,b,a)=c$. Of course, $m'$ acts as a minority operation on $\{b,c\}$, $m'(b,a,a),m'(a,b,a),m'(a,a,b)\in b/\alpha=\{b\}$ and for the same reason $m'(x,y,z)=b$ when $\{x,y,z\}=\{a,b,c\}$. Hence, $\mathrm{Sg}\{a,b\}=\{a,b,c\}$, a contradiction. Since $\mathrm{Sg}\{b,c\}=\{b,c\}$ and $\mathrm{Sg}\{b,d\}=\{b,d\}$ are impossible, $\{b,c,d\}$ also cannot be a subuniverse, since even if it were nonconservative, Table~\ref{table:1} shows that either $\{b,c\}$ or $\{b,d\}$ would be a subuniverse. Hence, $\mathrm{Sg}\{a,b\}=\mathrm{Sg}\{b,c\}=\mathrm{Sg}\{b,d\}=A$. Now we observe the relation $S$. As before, let $(\bullet,\ast,\bullet)\in S$ be the result of applying $m$ to a generating triple of $S$. From $(\bullet,\bullet,\ast),(\bullet,\ast,\bullet),(\ast,\bullet,\bullet)\in S$ we get $(\bullet,\bullet,\bullet)\in S$, where $\bullet=m(a,b,b)=m(b,b,a)\in D$. Assume that $\bullet\neq a$, as otherwise we have a contradiction. Since $a\in\mathrm{Sg}\{b,\bullet\}$, then $t(b,\bullet)=a$ for some binary term $t$. We further have $(\bullet,a,a)\in S$ after applying $t$ onto $(\bullet,\bullet,\bullet),(\bullet,b,b)\in S$. Finally, we apply $m$ onto $(\bullet,a,a),(a,\bullet,a),(a,a,\bullet)\in S$, so $(a,a,a)\in S$, a contradiction. This concludes the proof.
\end{proof}

\subsubsection{Equal size classes $\mathbb{Z}^{\text{aff}}_2$ quotient}

\begin{exmp}[Example 5.17 of \cite{dreamteam}]\label{example5.17}
   Let $\m A = (\{a, b, c, d\}; p)$, where $p$ is a Mal’cev operation with the following properties. The operation $p$ commutes with the permutations $\sigma = (a\; c)$ and $\tau = (b\; d)$. The polynomials
$+_a = p(\cdot, a, \cdot)$, $+_b = p(\cdot, b, \cdot)$ define abelian groups:
\begin{table}[h!]
\begin{subtable}{.45\linewidth}
\centering
 \begin{tabular}{c|c c c c}

    $+_a$ & $a$ & $b$& $c$  & $d$\\
\hline

    $a$ & $a$ & $b$ & $c$ & $d$  \\

	$b$ & $b$ & $c$ & $d$ & $a$ \\
	
	$c$ & $c$ & $d$ & $a$ & $b$\\
	
	$d$ & $d$ & $a$ & $b$ & $c$\\

\end{tabular}  

\end{subtable}
\begin{subtable}{.45\linewidth}
\centering
 \begin{tabular}{c|c c c c}

    $+_b$ & $a$ & $b$& $c$  & $d$\\
\hline

    $a$ & $b$ & $a$ & $d$ & $c$  \\

	$b$ & $a$ & $b$ & $c$ & $d$ \\
	
	$c$ & $d$ & $c$ & $b$ & $a$\\
	
	$d$ & $c$ & $d$ & $a$ & $b$\\

\end{tabular}  

\end{subtable}
\caption{Tables of operations from Example \ref{example5.17}}
\end{table}

Then $\m a$ is a minimal Taylor algebra.
\end{exmp}

Let $\alpha=\{\{a,c\},\{b,d\}\}$ be a congruence such that $\m a/\alpha$ is term equivalent to $\mathbb{Z}_2^\mathrm{aff}$. Denote by $g$ a ternary cyclic term of $\m a$, which must exist due to Proposition~\ref{composition}.

\begin{enumerate}[wide, labelindent=0pt]
    \item[\textsc{\textbf{Case 1.}}] $\{a,c\}$ is a two-element semilattice.
\begin{claim}\label{Cl62}
$(a,c)$ is a semilattice edge, and $\{b,c,d\}$ is a subalgebra.
\end{claim}

\begin{proof}
 Suppose that $(u,v)$ is a semilattice edge, where $\{u,v\}=\{a,c\}$. Let the term $t$ be defined by $t(x,y,z)=g(x,g(x,y,y),g(x,y,z))$ and let $g'(x,y,z):=g(t(x,y,z),t(y,z,x),t(z,x,y))$. Therefore,
 $$
t\left(
\begin{bmatrix}
v \\
b \\
b
\end{bmatrix}\!,\!
\begin{bmatrix}
b \\
v \\
b
\end{bmatrix}\!,\!
\begin{bmatrix}
b \\
b \\
v
\end{bmatrix}
\right)=g\left(
\begin{bmatrix}
v \\
b \\
b
\end{bmatrix}\!,\!
\begin{bmatrix}
g(v,b,b) \\
g(b,v,v) \\
b
\end{bmatrix}\!,\!
\begin{bmatrix}
g(v,b,b) \\
g(v,b,b)\\
g(v,b,b)
\end{bmatrix}
\right)=
\begin{bmatrix}
v \\
\circ \\
\ast
\end{bmatrix},
$$ where $\circ,\ast\in \{a,c\}$. This implies $g'(v,b,b)=v$ and, similarly, $g'(v,d,d)=v$. By analogous computation, we derive $g'(v,b,d)=g'(v,d,b)=v$, so $\{b,v,d\}$ is closed under the cyclic term $g'$. Thus $v$ can't be equal to $a$, i.e. $v=c$, which concludes the proof.
\end{proof}

\begin{claim}\label{Cl63}
It may be assumed that $\{a,d\}$ is a subalgebra.
\end{claim}
\begin{proof}
Suppose that $\mathrm{Sg}\{a,d\}\neq \{a,d\}$. Then $c\in\mathrm{Sg}\{a,d\}$, so there exists a binary term $t_1$ such that $t_1(a,d)=c$ and $t_1(d,a)\in \{b,d\}$. Since $\mathrm{Sg}\{a,b\}=\m a$, we know that $t_2(a,b)=c$ and $t_2(b,a)\in\{b,d\}$ for some binary term $t_2$. Let $t_3$ be a binary term defined by $$t_3(x,y)=g(t_1(x,y),t_1(x,y),t_2(x,y)).$$ It is an easy exercise to verify that $t_3$ witnesses the absorption $\{b,c,d\}\trianglelefteq_2\m a$. Since this case with a semilattice quotient is already examined, we indeed may assume that $\{a,d\}$ is a subalgebra.
\end{proof}

\begin{claim}
$\{b,d\}$ is not an affine subalgebra.
\end{claim}
\begin{proof}
For the sake of contradiction, let $\{b,d\}$ be an affine subalgebra. We know that $\{b,c,d\}$ is a subalgebra, and since it has an affine quotient, we conclude it must be a conservative one. Hence, $\{b,c\}$ and $\{c,d\}$ are subalgebras. 

Observe now that $\{a,c,d\}$ is also a subalgebra. Indeed, if we define a cyclic ternary term $g'$ by $$g'(x,y,z)=g(g(x,y,y),g(y,z,z),g(z,x,x)),$$ we have
$$
\begin{gathered}
g'(a,a,d)=g'(c,c,d)=d, \text{and  } g'(a,c,d)=g(c,c,d)=d, \text{and  }\\
g'(a,d,c)=g(a,d,c).
\end{gathered}
$$
Since $g(a,d,c)$ is not determined, we define a new cyclic term $g''$ by $$g''(x,y,z)=g'(g'(x,x,y),g'(y,y,z),g'(z,z,x)),$$
and now it is easy to verify that the set $\{a,c,d\}$ is closed under $g''$.

The next step in the proof is to notice that $g(a,a,b)=b$ may be assumed as well. Suppose that $g(a,a,b)=d$ holds and define the ternary term $h$ by $h(x,y,z)=g(x,g(x,y,y),g(x,y,z))$. Then 
$$
h\left(
\begin{bmatrix}
b \\
a \\
a
\end{bmatrix}\!,\!
\begin{bmatrix}
a \\
b \\
a
\end{bmatrix}\!,\!
\begin{bmatrix}
a \\
a \\
b
\end{bmatrix}
\right)=g\left(
\begin{bmatrix}
b \\
a \\
a
\end{bmatrix}\!,\!
\begin{bmatrix}
d \\
a/c \\
a
\end{bmatrix}\!,\!
\begin{bmatrix}
d \\
d \\
d
\end{bmatrix}
\right)=
\begin{bmatrix}
b \\
d \\
d
\end{bmatrix},
$$ so, if we define $g^*$ by $g^*(x,y,z)=g(h(x,y,z),h(y,z,x),h(z,x,y))$, we get $g^*(a,a,b)=b$. The cyclic term $g^*$ can be used instead of $g$; thus, we can assume $g(a,a,b)=b$. Of course, $g(a,a,b)=b$ implies $g(a,b,b)=c$.

We shall now derive a contradiction by showing $d\notin\mathrm{Sg}\{a,b\}$. Define the cyclic ternary term $g^+$ by $$g^+(x,y,z)=g(g(g(x,y,y),z,z),g(g(y,z,z),x,x),g(g(z,x,x),y,y)).$$ Since $\{b,c\}$ is a subalgebra, we need to compute only $g^+(b,a,a)$, $g^+(b,a,c)$ and $g^+(b,c,a)$. Hence,
\begin{align*}
   g^+(b,a,a)&=g(g(g(b,a,a),a,a),g(g(a,a,a),b,b),g(g(a,b,b),a,a))\\
            &=g(g(b,a,a),g(a,b,b),g(c,a,a))\\
            &=g(b,c,c)=b,
\end{align*}
and similarly $g^+(b,a,c)=g^+(b,c,a)=b$, a contradiction.
\end{proof}
\begin{claim}
$\{b,d\}$ cannot be a majority subalgebra; therefore, $(b,d)$ is a semilattice edge.
\end{claim}

\begin{proof}
For the contradiction, assume that $\{b,d\}$ is a majority subalgebra.
Since $\{b,c,d\}$ is a subalgebra, then $\{b,c\}$ or $\{c,d\}$ is also a subalgebra.
As $d\in\mathrm{Sg}\{a,b\}$, we have $t(b,a)=d$ for some binary term $t$. Then $t(a,b)\in\{a,c\}$. Assume that $t(a,b)=a$ and define the cyclic ternary term $g'$ by $$g'(x,y,z)=g(t(t(x,y),z),t(t(y,z),x),t(t(z,x),y)).$$
Suppose first $t(d,b)=d$. We compute $g'(a,b,b)=g'(a,d,d)=g'(a,d,b)=a$, but $g'(a,b,d)=g(a,d,t(t(d,a),b))$. Since $t(b,a)=d$ and $\{a,d\}$ is subalgebra, we have $t(d,a)=d$. It follows $t(t(d,a),b)=d$, thus $g'(a,b,d)=a$. This means that $\{a,b,d\}$ is closed under $g'$, a contradiction. If $t(d,b)=b$, then a contradiction is obtained using $t'$ instead of $t$ where $t'$ is defined by $t'(x,y)=t(t(x,y),t(y,x))$. Hence, 
\begin{equation}\tag{$\ast$}\label{binaryeq}
t(b,a)=d\text{ implies }t(a,b)=c\text{ for any binary term }t.
\end{equation}

Suppose first that $\{b,c\}$ is a subalgebra; thus, it must be affine. Now let $t'(x,y):=g(x,x,t(x,y))$, so $t'(a,b)=c$ and $t'(b,a)=b$. Define the cyclic term $g^*$ by $$g^*(x,y,z)=g(t'(t'(x,y),z),t'(t'(y,z),x),t'(t'(z,x),y)).$$ We compute $g^*(b,a,a)= g^*(b,c,c)=g^*(b,a,c)=g^*(b,c,a)=b$, so $d\notin\mathrm{Sg}\{a,b\}$, a contradiction.

Consequently, from now on, we assume $\mathrm{Sg}\{c,d\}=\{c,d\}$ and $\mathrm{Sg}\{b,c\}=\{b,c,d\}$ which must be $\T^{\n}_3$, according to Table~\ref{table:1}, and thus any ternary cyclic term is uniquely determined on $\{b,c,d\}$. Let $t$ be a binary term such that $t(b,a)=d$, thus $t(a,b)=c$. Additionally, we have $t(c,b)=c$ because $\{b,c,d\}$ is a subalgebra. Let $g'$ be defined as in the first paragraph, so, $$g'(b,a,a)=g(t(d,a),t(a,b),t(c,a))=g(d,c,c)=d$$ follows and by \eqref{binaryeq} $g'(a,b,b)=c$. Define the binary term $t_1$ by $$t_1(x,y)=g'(x,g'(y,x,x),g'(y,x,x)).$$
Hence, $t_1(a,b)=a$ and $t_1(b,a)=g'(b,c,c)=d$ by the table of $\T^{\n}_3$. This contradicts \eqref{binaryeq}, concluding the proof.
\end{proof}

By dualizing Claims~\ref{Cl62} and \ref{Cl63}, $\{a,c,d\}$ and $\{b,c\}$ are subalgebras, and by intersecting we get that $\{c,d\}$ is one, too. Hence, both $\{b,c,d\}$ and $\{a,c,d\}$ are conservative. Moreover, $\{a,d\}$ and $\{b,c\}$ are affine subalgebras. 

\begin{claim}\label{automorphism claim}
The permutation which transposes $a$ with $b$, and also transposes $c$ with $d$, is an automorphism.
\end{claim}

\begin{proof}
It suffices to prove that if $t$ is any binary term, then $t(a,b)=c$ if and only if $t(b,a)=d$. Indeed, if the equivalence is not true, and if, without loss of generality, $t(a,b)=c$ and $t(b,a)=b$, it can be easily verified that $\{a,b,c\}$ is closed under the cyclic term $g'$ defined by $$g'(x,y,z)=g(t(t(x,y),z),t(t(y,z),x),t(t(z,x),y)),$$ which is a contradiction.
\end{proof}

In particular, we have $g(a,b,b)=c$ and $g(a,a,b)=d$. Furthermore, the only undetermined values remaining are $g(a,b,c)$ and $g(a,c,b)$.

\begin{claim}
There exists a cyclic ternary term $g^*$ such that $\{a,b\}$ is closed under it, which contradicts the assumption that $\mathrm{Sg}\{a,b\}=\m a$.
\end{claim}

\begin{proof}
First, we shall find a cyclic term that is completely determined. Let $g'(x,y,z):=g(g(x,y,y),g(y,z,z),g(z,x,x))$, so, we get $g'(a,b,c)=b$ and $g'(a,c,b)=d$, while it matches values $g$ on the remaining places.
Let $t(x,y):=g'(x,y,y)$ and define a cyclic term $g^*$ by
\begin{gather*}
    g^*(x,y,z)=g'(t(x,t(g(x,x,y),g(y,z,z))),\\
    t(y,t(g(y,y,z),g(z,x,x))),t(z,t(g(z,z,x),g(x,y,y)))).
\end{gather*}
Now we have 
\begin{gather*}
    g^*(a,b,b)=g'(t(a,t(d,b)),t(b,t(b,d)),t(b,t(c,c)))=\\
    g'(t(a,d),t(b,d),t(b,c))= g'(a,d,b)=a,
\end{gather*}
and by the automorphism, $g^*(b,a,a)=b$, so we are done.
\end{proof}

    \item[\textsc{\textbf{Case 2.}}] $\{a,c\}$ is a majority subalgebra.
\begin{claim}
 $\{b,d\}$ cannot be an affine subalgebra; thus, it must be a majority.
\end{claim}
\begin{proof}
Assume that $\{b,d\}$ is affine. We look at the possibilities for $g(a,a,b)$ and $g(a,b,b)$. We cannot have $g(a,a,b)=b$ and $g(a,b,b)=a$ simultaneously. Hence, we have three cases left to discuss.

Suppose that $g(a,b,b)=a$ and $g(a,a,b)=d$ hold. Define the term $h$ by $$h(x,y,z)=g(x,g(x,y,y),g(x,y,z)),$$ and let $g'$ be defined by $g'(x,y,z)=g(h(x,y,z),h(y,z,x),h(z,x,y))$. 
$$
\begin{gathered}
h\left(
\begin{bmatrix}
b \\
a \\
a
\end{bmatrix}\!,\!
\begin{bmatrix}
a \\
b \\
a
\end{bmatrix}\!,\!
\begin{bmatrix}
a \\
a \\
b
\end{bmatrix}
\right)=g\left(
\begin{bmatrix}
b \\
a \\
a
\end{bmatrix}\!,\!
\begin{bmatrix}
d \\
a \\
a
\end{bmatrix}\!,\!
\begin{bmatrix}
d \\
d \\
d
\end{bmatrix}
\right)=
\begin{bmatrix}
b \\
\ast \\
\ast
\end{bmatrix},
\hspace{0.2cm}\text{ while }\\
h\left(
\begin{bmatrix}
a \\
b \\
b
\end{bmatrix}\!,\!
\begin{bmatrix}
b \\
a \\
b
\end{bmatrix}\!,\!
\begin{bmatrix}
b \\
b \\
a
\end{bmatrix}
\right)=g\left(
\begin{bmatrix}
a \\
b \\
b
\end{bmatrix}\!,\!
\begin{bmatrix}
a \\
d \\
b
\end{bmatrix}\!,\!
\begin{bmatrix}
a \\
a \\
a
\end{bmatrix}
\right)=
\begin{bmatrix}
a \\
\circ \\
a
\end{bmatrix},
\end{gathered}
$$

where $\ast\in\{b,d\}$ and $\circ\in\{a,c\}$. Thus, $g'(a,b,b)=a$ and $g'(a,a,b)=b$ for the cyclic term $g'$, a contradiction.

Let now $g(a,b,b)=c$ and $g(a,a,b)=b$. We define $w$ by $w(x,y,z)=g(x,x,g(x,y,z))$ and $g''(x,y,z)=g(w(x,y,z),w(y,z,x),w(z,x,y))$. Then
$$
\begin{gathered}
w\left(
\begin{bmatrix}
b \\
a \\
a
\end{bmatrix}\!,\!
\begin{bmatrix}
a \\
b \\
a
\end{bmatrix}\!,\!
\begin{bmatrix}
a \\
a \\
b
\end{bmatrix}
\right)=g\left(
\begin{bmatrix}
b \\
a \\
a
\end{bmatrix}\!,\!
\begin{bmatrix}
b \\
a \\
a
\end{bmatrix}\!,\!
\begin{bmatrix}
b \\
b \\
b
\end{bmatrix}
\right)=
\begin{bmatrix}
b \\
b \\
b
\end{bmatrix},
\hspace{0.2cm}\text{ while }\\
w\left(
\begin{bmatrix}
a \\
b \\
b
\end{bmatrix}\!,\!
\begin{bmatrix}
b \\
a \\
b
\end{bmatrix}\!,\!
\begin{bmatrix}
b \\
b \\
a
\end{bmatrix}
\right)=g\left(
\begin{bmatrix}
a \\
b \\
b
\end{bmatrix}\!,\!
\begin{bmatrix}
a \\
b \\
b
\end{bmatrix}\!,\!
\begin{bmatrix}
c \\
c \\
c
\end{bmatrix}
\right)=
\begin{bmatrix}
a \\
\ast \\
\ast
\end{bmatrix},
\end{gathered}
$$
where $\ast=g(b,b,c)\in \{a,c\}$. We know $g''(a,a,b)=b$. Also, If $\ast=a$, then $g''(a,b,b)=a$, which together with $g''(a,a,b)=b$ gives a contradiction. Therefore, let us assume $\ast=g(b,b,c)=c$, so $g''(a,b,b)=c$. 

Assume first that $g(b,c,c)=d$. The subcase $\mathrm{Sg}\{b,c\}=\m a$ leads to the same contradiction as in the case that assumed $g(a,b,b)=a$ and $g(a,a,b)=d$, with $a$ and $c$ swapping places. If $\mathrm{Sg}\{b,c\}=\{b,c,d\}$, this is impossible by Table~\ref{table:1}, as there is no 3-element nonconservative algebra which has both a congruence class and a factor isomorphic to $\mathbb{Z}^{\text{aff}}_2$. So it must be that $g(b,c,c)=b$ and $\mathrm{Sg}\{b,c\}=\{b,c\}$. Let $t(x,y):=g(x,y,y)$ and define the cyclic term $g'''$ by $g'''(x,y,z)=g(t(x,t(y,z)),t(y,t(z,x)),t(z,t(x,y))).$ We compute 
$$
\begin{gathered}
t\left(
\begin{bmatrix}
b \\
a \\
a
\end{bmatrix}\!,
t\left(
\begin{bmatrix}
a \\
b \\
a
\end{bmatrix}\!,\!
\begin{bmatrix}
a \\
a \\
b
\end{bmatrix}
\right)\right)=t\left(
\begin{bmatrix}
b \\
a \\
a
\end{bmatrix}\!,\!
\begin{bmatrix}
a \\
b \\
c
\end{bmatrix}
\right)=
\begin{bmatrix}
b \\
c \\
c
\end{bmatrix},
\;
t\left(
\begin{bmatrix}
a \\
b \\
c
\end{bmatrix}\!,
t\left(
\begin{bmatrix}
b \\
c \\
a
\end{bmatrix}\!,\!
\begin{bmatrix}
c \\
a \\
b
\end{bmatrix}
\right)\right)=t\left(
\begin{bmatrix}
a \\
b \\
c
\end{bmatrix}\!,\!
\begin{bmatrix}
b \\
a \\
c
\end{bmatrix}
\right)=
\begin{bmatrix}
c \\
b \\
c
\end{bmatrix},\\
t\left(
\begin{bmatrix}
a \\
c \\
b
\end{bmatrix}\!,
t\left(
\begin{bmatrix}
c \\
b \\
a
\end{bmatrix}\!,\!
\begin{bmatrix}
b \\
a \\
c
\end{bmatrix}
\right)\right)=t\left(
\begin{bmatrix}
a \\
c \\
b
\end{bmatrix}\!,\!
\begin{bmatrix}
c \\
b \\
c
\end{bmatrix}
\right)=
\begin{bmatrix}
c \\
c \\
b
\end{bmatrix}.
\end{gathered}
$$  
Hence, $g'''(a,a,b)=g'''(a,b,c)=g'''(a,c,b)=b$ and $\mathrm{Sg}\{b,c\}=\{b,c\}$ implies that $g'''(b,c,c)=b$, as well. Thus, $d\notin\mathrm{Sg}\{a,b\}$, a contradiction.

Consider now the case when $g(a,b,b)=c$ and $g(a,a,b)=d$. Let $h'(x,y,z):=g(x,x,g(x,y,y))$ and define a cyclic term $g^+$ by $$g^+(x,y,z)=g(h'(x,y,z),h'(y,z,x),h'(z,x,y)).$$ Then 
\begin{equation*}
h'\left(
\begin{bmatrix}
b \\
a \\
a
\end{bmatrix}\!,\!
\begin{bmatrix}
a \\
b \\
a
\end{bmatrix}\!,\!
\begin{bmatrix}
a \\
a \\
b
\end{bmatrix}
\right)=
g\left(
\begin{bmatrix}
b \\
a \\
a
\end{bmatrix}\!,\!
\begin{bmatrix}
b \\
a \\
a
\end{bmatrix}\!,\!
g\left(
\begin{bmatrix}
b \\
a \\
a
\end{bmatrix}\!,\!
\begin{bmatrix}
a \\
b \\
a
\end{bmatrix}\!,\!
\begin{bmatrix}
a \\
b \\
a
\end{bmatrix}
\right)
\right)=
g\left(
\begin{bmatrix}
b \\
a \\
a
\end{bmatrix}\!,\!
\begin{bmatrix}
b \\
a \\
a
\end{bmatrix}\!,\!
\begin{bmatrix}
d \\
c \\
a
\end{bmatrix}
\right)=
\begin{bmatrix}
d \\
a\\
a
\end{bmatrix}.
\end{equation*}
Now, if $g(d,a,a)=b$, then $g^+(a,a,b)=b$. If we want to avoid the immediate contradiction, we need to assume $g^+(a,b,b)=c$, but, according to the two previous paragraphs, this also leads to a contradiction. That being said, we have $g(d,a,a)=d$. The following discussion is based on what $g(a,d,d)$ is equal to. 

Suppose first that $g(a,d,d)=a$, thus $\{a,d\}$ is a subalgebra. Let $h'$ be defined as before, and we define $h''$ by $h''(x,y,z)=g(x,x,g(x,z,z))$. Then 
\begin{equation*}
h''\left(
\begin{bmatrix}
a \\
b \\
b
\end{bmatrix}\!,\!
\begin{bmatrix}
b \\
a \\
b
\end{bmatrix}\!,\!
\begin{bmatrix}
b \\
b \\
a
\end{bmatrix}
\right)=
g\left(
\begin{bmatrix}
a \\
b \\
b
\end{bmatrix}\!,\!
\begin{bmatrix}
a \\
b \\
b
\end{bmatrix}\!,\!
g\left(
\begin{bmatrix}
a \\
b \\
b
\end{bmatrix}\!,\!
\begin{bmatrix}
b \\
b \\
a
\end{bmatrix}\!,\!
\begin{bmatrix}
b \\
b \\
a
\end{bmatrix}
\right)
\right)=
g\left(
\begin{bmatrix}
a \\
b \\
b
\end{bmatrix}\!,\!
\begin{bmatrix}
a \\
b \\
b
\end{bmatrix}\!,\!
\begin{bmatrix}
c \\
b \\
d
\end{bmatrix}
\right)=
\begin{bmatrix}
a \\
b\\
d
\end{bmatrix},
\end{equation*}
\begin{equation*}
h'\left(
\begin{bmatrix}
a \\
b \\
b
\end{bmatrix}\!,\!
\begin{bmatrix}
b \\
a \\
b
\end{bmatrix}\!,\!
\begin{bmatrix}
b \\
b \\
a
\end{bmatrix}
\right)=
g\left(
\begin{bmatrix}
a \\
b \\
b
\end{bmatrix}\!,\!
\begin{bmatrix}
a \\
b \\
b
\end{bmatrix}\!,\!
g\left(
\begin{bmatrix}
a \\
b \\
b
\end{bmatrix}\!,\!
\begin{bmatrix}
b \\
a \\
b
\end{bmatrix}\!,\!
\begin{bmatrix}
b \\
a \\
b
\end{bmatrix}
\right)
\right)=
g\left(
\begin{bmatrix}
a \\
b \\
b
\end{bmatrix}\!,\!
\begin{bmatrix}
a \\
b \\
b
\end{bmatrix}\!,\!
\begin{bmatrix}
c \\
d \\
b
\end{bmatrix}
\right)=
\begin{bmatrix}
a \\
d\\
b
\end{bmatrix},
\end{equation*}
Define the term $w'$ by $$w'(x,y,z)=g(x,h'(x,y,z),h''(x,y,z)),$$ so, we have 
\begin{equation*}
w'\left(
\begin{bmatrix}
a \\
b \\
b
\end{bmatrix}\!,\!
\begin{bmatrix}
b \\
a \\
b
\end{bmatrix}\!,\!
\begin{bmatrix}
b \\
b \\
a
\end{bmatrix}
\right)=g\left(
\begin{bmatrix}
a \\
b \\
b
\end{bmatrix}\!,\!
\begin{bmatrix}
a \\
d \\
b
\end{bmatrix}\!,\!
\begin{bmatrix}
a \\
b \\
d
\end{bmatrix}
\right)=
\begin{bmatrix}
a \\
d \\
d
\end{bmatrix}.
\end{equation*}
Define the cyclic term $g^{\circ}$ by $g^{\circ}(x,y,z)=g(w'(x,y,z),w'(y,z,x),w'(z,x,y))$, thus we get $g^{\circ}(a,b,b)=a$. Now, we recall the analysis from the beginning of the proof, which implies that $\{a,b\}$ is closed for some cyclic term, a contradiction.

Therefore, we suppose that $g(a,d,d)=c$. Observe that we may assume $\mathrm{Sg}\{a,d\}=\{a,c,d\}$. Indeed, if $\mathrm{Sg}\{a,d\}=\m a$, after swapping $b$ and $d$, we find ourselves in a case where $\mathrm{Sg}\{a,b\} = \m a$ and $g(b,a,a) = b$, $g(a,b,b) = c$, which we have already ruled out. Hence, $\{a,c,d\}$ is a subalgebra and, according to \Cref{table:1}, it must be $\m t_3^{\scriptscriptstyle{\mathsf{N}}}$. Define the term $w''$ by $$w''(x,y,z)=g(x,g(x,y,y),g(x,y,z)),$$ so, we have 
 \begin{equation*}
w''\left(
\begin{bmatrix}
b \\
a \\
a
\end{bmatrix}\!,\!
\begin{bmatrix}
a \\
b \\
a
\end{bmatrix}\!,\!
\begin{bmatrix}
a \\
a \\
b
\end{bmatrix}
\right)=
g\left(
\begin{bmatrix}
b \\
a \\
a
\end{bmatrix}\!,\!
g\left(
\begin{bmatrix}
b \\
a \\
a
\end{bmatrix}\!,\!
\begin{bmatrix}
a \\
b \\
a
\end{bmatrix}\!,\!
\begin{bmatrix}
a \\
b \\
a
\end{bmatrix}
\right)\!,\!
g\left(
\begin{bmatrix}
b \\
a \\
a
\end{bmatrix}\!,\!
\begin{bmatrix}
a \\
b \\
a
\end{bmatrix}\!,\!
\begin{bmatrix}
a \\
a \\
b
\end{bmatrix}
\right)
\right)=
g\left(
\begin{bmatrix}
b \\
a \\
a
\end{bmatrix}\!,\!
\begin{bmatrix}
d \\
c \\
a
\end{bmatrix}\!,\!
\begin{bmatrix}
d \\
d \\
d
\end{bmatrix}
\right)=
\begin{bmatrix}
b \\
d\\
d
\end{bmatrix}.
\end{equation*}
Define the cyclic term $g^*$ by $g^*(x,y,z)=g(w''(x,y,z),w''(y,z,x),w''(z,x,y))$. It holds $g^*(b,a,a)=b$, but we have eliminated this case (the one with $g(a,b,b)=c$ and $g(a,a,b)=b$) before. This concludes the proof of the claim.
\end{proof}
We distinguish three cases regarding which algebra the set $\{a,d\}$ generates.
\begin{enumerate}[wide, labelindent=0pt]
    \item[\textsc{\textbf{Subcase 2.1}}] $\mathrm{Sg}\{a,d\}=\{a,d\}$.
    
    We shall prove that this condition is not satisfied by any unclassified algebra. Let us start by checking how a cyclic term acts on the generating set $\{a,b\}$. The following claims apply to any cyclic term.
    \begin{claim}\label{c(abb)=d claim}
    It holds $g(a,b,b)=c$.
    \end{claim}
    \begin{proof}
    For the sake of contradiction, suppose that $g(a,b,b)=a$, and thus $g(b,a,a)=d$. Let $t(x,y):=g(x,y,y)$ and define $g'$ by $$g'(x,y,z)=g(t(t(x,y),z),t(t(y,z),x),t(t(z,x),y)).$$ By the assumptions we just made and the assumption that $\{a,d\}$ is a subuniverse, we see that 
    \[
    \begin{gathered}
    g'(a,d,b)=g(t(a,b),t(b,a),t(d,d))=g(a,d,d)=a=g'(a,d,d)\\
    g'(a,b,b)=g(t(a,b),t(b,a),t(d,b))=g(a,d,b)=\\
    g(t(a,d),t(d,a),t(d,b))=g'(a,b,d).
    \end{gathered}
    \]
    If $g(a,d,b)=a$, then $\{a,d,b\}$ would be closed under $g'$, so we must have $g'(a,b,b)=c=g'(a,b,d)$. Now we define $g''$ by $$g''(x,y,z)=g'(t(t(x,y),z),t(t(y,z),x),t(t(z,x),y)).$$ The term $g''$ is cyclic and one can verify that
    \[
    \begin{gathered}
     g''(a,d,b)=g'(t(a,b),t(b,a),t(d,d))=g'(a,d,d)=a=g''(a,d,d)\\    g''(a,b,b)=g'(t(a,b),t(b,a),t(d,b))=g'(a,d,b)=a\\
     g''(a,b,d)=g'(t(a,d),t(d,a),t(d,b))=g'(a,d,b)=a.    
    \end{gathered}
    \]
    Hence, $\{a,b,d\}$ is closed under $g''$, a contradiction.
    \end{proof}

    \begin{claim}\label{bdaffineclaim}
    $\{b,c\}$ is a (affine) subalgebra.
    \end{claim}
    \begin{proof}
    Suppose not, then we have $g(c,b,b)=a$ or $g(b,c,c)=d$. Suppose first that $g(c,b,b)=a$ and let $h(x,y,z):=g(x,x,g(x,y,z))$ and define the ternary cyclic term $g'$ by $g'(x,y,z)=g(h(x,y,z),h(y,z,x),h(z,x,y))$. We get that $g'(a,b,b)=g(g(a,a,c),g(b,b,c),g(b,b,c))=g(a,a,a)=a$, which contradicts Claim~\ref{c(abb)=d claim}. 

    Assume now $g(b,c,c)=d$ and $g(b,b,c)=c$. Let $$h(x,y,z):=g(x,g(x,x,g(x,y,z)),g(x,x,g(x,y,z)))$$ and let $g'$ be the ternary cyclic term defined by $$g'(x,y,z)=g(h(x,y,z),h(y,z,x),h(z,x,y)).$$ Again, we compute $$h(a,b,b)=g(a,g(a,a,c),g(a,a,c))=a$$ and $$h(b,a,b)=h(b,b,a)=g(b,g(b,b,c),g(b,b,c))=g(b,c,c)=d.$$ Hence, $g'(a,b,b)=g(a,d,d)=a$, a contradiction once again.
    \end{proof}

    \begin{claim}\label{g(baa)=d}
    It holds $g(b,a,a)=d$.
    \end{claim}
    \begin{proof}
    Similar to the proof of Claim~\ref{c(abb)=d claim}.
    \end{proof}

    \begin{claim}\label{c(bde)=d}
    It holds $g(b,c,d)=g(b,d,c)=c$, and, similarly, $g(a,c,d)=g(a,d,c)=d$.
    \end{claim}
    \begin{proof}
    Suppose, for the contradiction, that $g(b,d,c)=a$. Define the cyclic ternary term $g'$ by $$g'(x,y,z)=g(g(x,y,y),g(y,z,z),g(z,x,x)).$$ Then $g'(a,b,b)=g(c,b,d)=a$, which contradicts Claim~\ref{c(abb)=d claim}. The case $g(b,c,d)=a$ is handled similarly, just $g'$ is defined by $$g'(x,y,z)=g(g(x,x,y),g(y,y,z),g(z,z,x)).$$
    \end{proof}

    \begin{claim}
    A permutation of order two which maps $a$ to $b$ and $c$ to $d$ is an automorphism.  
    \end{claim}
    \begin{proof}
    Similarly to the proof of Claim~\ref{automorphism claim}, we assume $t(a,b)=c$, $t(b,a)=b$ and define the cyclic term $g'(x,y,z)=g(t(t(x,y),z),t(t(y,z),x),t(t(z,x),y))$. If $t(c,a)=c$, then $t(t(c,a),b)=t(c,b)=c$, while if $t(c,a)=a$, then again $t(t(c,a),b)=t(a,b)=c$. Hence, $g'(b,c,a)=g(b,c,c)=b$. If we define $g''$ by $$g''(x,y,z)=g'(t(t(x,y),z),t(t(y,z),x),t(t(z,x),y)),$$ we get $g''(b,a,a)=g'(b,c,t(c,a))$. If $t(c,a)=c$ we get $g''(b,a,a)=g'(b,c,c)=b$, while if $t(c,a)=a$ we also get $g''(b,a,a)=g'(b,c,a)=b$, either way a contradiction with Claim~\ref{g(baa)=d}.
    \end{proof}

    At this point, we have two possibilities: $\mathrm{Sg}\{c,d\}=\{c,d\}$ or $\mathrm{Sg}\{c,d\}=\m a$. If the latter of these two cases holds, then we have $g(c,d,d)=a$ and $g(d,c,c)=b$, so we can use $c$ and $d$ as generators too. Thus by the dual of Claim~\ref{c(bde)=d}, the operation $g$ is completely determined. This algebra also has a two-element majority quotient and we already classified it as $\mathbf{T}_{4,9}$. Therefore, we may assume that $\{c,d\}$ is a subalgebra.
    
    We shall find one completely determined cyclic term. Define $g'$ by $$g'(x,y,z)=g(g(x,y,y),g(y,z,z),g(z,x,x)),$$ so we get $g'(c,b,a)=g(c,d,c)=d$ and $g'(d,a,b)=c$, while $g'(a,b,c)$ and $g'(b,a,d)$ remain undetermined. Hence, we define $g''$ by $$g''(x,y,z)=g'(g'(x,y,y),g'(y,z,z),g'(z,x,x)).$$ Now, $g''(a,b,c)=g'(c,b,a)=d=g''(a,c,b)$ and $g''(a,b,d)=g''(a,d,b)=c$, while on other triples $g''$ fulfills Claim~\ref{c(abb)=d claim}, Claim~\ref{g(baa)=d} and Claim~\ref{c(bde)=d}. The next claim settles the discussion in the case when $\{a,d\}$ is a subalgebra. 
    \begin{claim}
    $(a,a,a)\in S$ holds, which contradicts Lemma~\ref{aaalemma}.
    \end{claim}
    \begin{proof}
    First, we get $(d,b,c)\in S$, by applying $g''$ to $(a,b,b),(a,b,b),(b,b,a)\in S$. Further, from $(d,b,c),(d,c,b),(a,b,b)\in S$ we have $(a,c,c)\in S$, which implies $(a,b,d)\in S$ after we apply $g''$ to $(a,c,c),(a,c,c),(c,b,d)$. Notice that from $(a,b,d),(d,b,a),(a,c,c)\in S$ we get $(d,c,d)\in S$. If we apply $g''$ to $(c,d,d),(a,c,c),(a,c,c)\in S$, we get $(a,d,d)\in S$. Finally, applying $g''$ to $(a,d,d),(d,a,d),(d,d,a)\in S$, we get $(a,a,a)\in S$ which concludes the proof.
    \end{proof}
    
    \item[\textsc{\textbf{Subcase 2.2}}] $\mathrm{Sg}\{a,d\}=\{a,c,d\}$.
    
    Since it is not a conservative algebra, the algebra $\m t_3^{\n}$ is the only possibility for $\{a,c,d\}$, where $\{c,d\}$ is an affine algebra. If  $\mathrm{Sg}\{b,c\}=\m a$, then we are in Subcase 2.1 with $a$ and $c$ switching roles, since $\{c,d\}$ is a subalgebra. The case  $\mathrm{Sg}\{b,c\}=\{b,c\}$ also returns us to Subcase 2.1 with $a$ and $b$ interchanged. Thus $\mathrm{Sg}\{b,c\}=\{b,c,d\}$ and it is also isomorphic to $\m t_3^{\n}$. Of course, we have that $\{c,d\}$ is an affine subalgebra and the cyclic term $g$ is completely determined on $\{a,c,d\}$ and $\{b,c,d\}$ since $\m t_3^{\n}$ has the unique ternary cyclic term.
    \begin{claim}
    It may be assumed that $g(a,b,b)=c$ and $g(b,a,a)=d$.
    \end{claim}
    \begin{proof}
    Suppose, without loss of generality, that $g(a,b,b)=c$ and $g(b,a,a)=b$. We find a cyclic term $g_2$ such that $g_2(a,b,b)=c$ and $g_2(b,a,a)=d$. Let $g_1$ and $g_2$ be ternary terms such that
    \begin{align*}
        g_1(x,y,z)&=g(g(x,y,y),g(y,z,z),g(z,x,x)),\\
        g_2(x,y,z)&=g_1(g_1(x,y,y),g_1(y,z,z),g_1(z,x,x)).
    \end{align*}
    We compute $g_1(a,b,b)=g(c,b,b)=c$, $g_1(b,a,c)=g(b,c,c)=d$ and $g_1(b,a,a)=g(b,a,c)\in \{b,d\}$. Hence $g_2(a,b,b)=g_1(c,b,g_1(b,a,a))=c$ by the table of $\m t_3^{\n}$. Also, if $g_1(b,a,a)=d$, then $g_2(b,a,a)=g_1(d,a,c)=d$  by the table of $\m t_3^{\n}$. If $g_1(b,a,a)=b$ then $g_2(b,a,a)=g_1(b,a,c)=d$, which was needed to prove the claim.
    \end{proof}
    Hence, let $g(a,b,b)=c$ and $g(b,a,a)=d$. It remains to determine values $g(a,b,c)$, $g(a,c,b)$, $g(a,b,d)$ and $g(a,d,b)$. Define a new cyclic ternary term $g'$ by 
    \[
    g'(x,y,z)=g(g(x,y,y),g(y,z,z),g(z,x,x)).
    \]
    This term is completely determined, as one can easily check that 
    \begin{gather}\label{mtat11}
    \begin{aligned}
        g'(a,b,b)=g(c,b,d)=c,\,& g'(a,b,d)=g(c,d,d)=c,\,g'(a,d,b)=g(c,b,d)=c,\\
        g'(b,a,a)=g(d,a,c)=d,\,& g'(a,b,c)=g(c,d,a)=d\text{ and }g'(a,c,b)=g(c,c,d)=d
    \end{aligned}
    \end{gather}
    hold. We shall prove that $(A;g')$ is a minimal Taylor algebra. It suffices to show that there is no ternary cyclic term $g^*\in\mathrm{Clo}((A;g'))$ such that $g^*(a,a,b)=b$ and $g^*(b,b,a)=a$, since then we can construct $g'$ from any ternary cyclic term just like we did above.

    \begin{claim}
    Any ternary cyclic term $g^*\in\mathrm{Clo}((A;g'))$ satisfies $g^*(a,a,b)=d$ and $g^*(b,b,a)=c$.
    \end{claim}
    \begin{proof}
We shall prove that any ternary cyclic term $g^*\in\mathrm{Clo}((A;g'))$ satisfies $g^*(a,a,b)=d$, and the proof of $g^*(b,b,a)=c$ is analogous. To do that, it is sufficient to prove that $(b,b)\notin S$, where
\[
S=\mathrm{Sg}_{\m a^2}\left\{
\begin{bmatrix}
a \\
a
\end{bmatrix}\!,\!
\begin{bmatrix}
a \\
b 
\end{bmatrix}\!,\!
\begin{bmatrix}
b \\
a 
\end{bmatrix}
\right\}.
\]
In fact, the core of the argument is showing that the collection of all pairs other than $(b,b)$ forms a subalgebra. Indeed, the only two ways to obtain $b$ are $g'(b,b,b)$ and $g'(b,b,d)$, so to get $(b,b)$ as an output requires $(b,b)$ on at least one input. Therefore, since $A^2\setminus\{(b,b)\}$ is a subuniverse, we have $(b,b)\notin S$, which concludes the proof.
\end{proof}
    The minimal Taylor algebra $(A;g')$ where $\{a,c,d\}$ and $\{b,c,d\}$ are subalgebras isomorphic to $\m t_3^{\n}$ and $g'$ additionally satisfies equations \eqref{mtat11} is denoted $\m t_{4,11}$.

    \item[\textsc{\textbf{Subcase 2.3}}] $\mathrm{Sg}\{a,d\}=\mathrm{Sg}\{b,c\}=\m a$.

    In this case, we can also assume $\mathrm{Sg}\{c,d\}=\m a$, as otherwise we could recall some of the previous cases. In other words, the algebra $\m a$ is generated by any pair of elements from different $\alpha$-classes.
    \begin{claim}
    The following may be assumed:
    \begin{itemize}
        \item $g(a,b,b)=a$, and therefore $g(b,a,a)=d$;
    
        \item $g(a,d,d)=c$.
    \end{itemize}
    \end{claim}
    \begin{proof}
     First, we shall prove that a cyclic term $g'$ satisfying $g'(a,b,b)=a$ can always be constructed. Suppose $g(a,b,b)=c$, and therefore $g(c,b,b)=a$ because, if $g(c,b,b)=c$, we can switch $a$ and $c$. Define a ternary cyclic term $g'$ by $$g'(x,y,z)=g(g(x,x,g(x,y,z)),g(y,y,g(x,y,z)),g(z,z,g(x,y,z))).$$ Then $g'(a,b,b)=g(g(a,a,c),g(b,b,c),g(b,b,c))=g(a,a,a)=a$. Now we show that for any such term, we have $g'(a,d,d)=c$. If $g'(a,d,d)\neq c$, then $g'(a,d,d)=a$, and thus $g'(d,a,a)=b$. Define a new ternary cyclic term $g''$ by 
$$
\begin{gathered}
g''(x,y,z)=g'(g'(x,g'(x,y,z),g'(x,y,z)),\\
g'(y,g'(x,y,z),g'(x,y,z)),g'(z,g'(x,y,z),g'(x,y,z))).
\end{gathered}
$$ Now 
     $$
     \begin{gathered}
     g''(a,b,b)=g'(g'(a,a,a),g'(b,a,a),g'(b,a,a))=g'(a,d,d)=a\text{ and }\\
     g''(b,a,a)=g'(g'(b,d,d),g'(a,d,d),g'(a,d,d))=g'(d,a,a)=b.
     \end{gathered}
     $$ 
     Hence, $\{a,b\}$ is closed under $g''$, a contradiction.
    \end{proof}
   Recall the term $g'$ from the previous claim. From $$g'(a,b,b)=g(g(a,a,a),g(b,b,a),g(b,b,a))=g(a,a,a)=a$$ follows $g'(b,a,a)=d$. On the other hand, we have $$g'(b,a,a)= g(b,g(a,a,d), g(a,a,d)),$$ and since $g(a,a,d)$ is in $\{b,d\}$, which is a majority algebra, we get $g'(b,a,a)=g(b, g(a,a,d), g(a,a,d)) = g(a,a,d)$. Hence, we have $g(a,a,d)=d$. Similarly, by acting with $g'$ on sets $\{a,d\}$ and $\{c,d\}$, we get $g(d,d,c)=c$, $g(d,c,c)=b$, $g(c,c,b)=b$ and $g(c,b,b)=a$. In other words, terms $g$ and $g'$ match on every two-element set. Let us look what happens on the set $\{a,b,c\}$: $g'(a,b,c)=g(g(a,a,\ast),g(b,b,\ast),g(c,c,\ast))$, where $\ast\in\{b,d\}$. Hence, $g'(a,b,c)=g(d,b,b)=b$ holds, and, similarly, $g'(a,c,b)=b$. Furthermore, by the same reasoning (after appropriately permuting the elements), we get:
   \begin{align*}
       g'(a,b,d)&=a=g'(a,d,b),\\
       g'(a,c,d)&=d=g'(a,d,c),\\
       g'(b,c,d)&=c=g'(b,d,c).
   \end{align*}
   The cyclic term $g'$ is completely determined. The next claim proves that a clone generated by $g'$ is not a minimal Taylor.
   \begin{claim}
There exists a cyclic ternary term $g^*\in\mathrm{Clo}((A;g'))$ such that $\{a,b\}$ is closed under it, which contradicts the assumption that $\mathrm{Sg}\{a,b\}=\m a$.
\end{claim}
   
\begin{proof}
Define the ternary term $g''$ by $$g''(x,y,z)=g'(g'(x,y,z),g'(x,y,z),g'(x,y,y)).$$ Then
\begin{equation}\label{eq for g''}
g''\left(
\begin{bmatrix}
a \\
b \\
b
\end{bmatrix}\!,\!
\begin{bmatrix}
b \\
a \\
b
\end{bmatrix}\!,\!
\begin{bmatrix}
b \\
b \\
a
\end{bmatrix}
\right) =g'\left(
g'\left(
\begin{bmatrix}
a \\
b \\
b
\end{bmatrix}\!,\! 
\begin{bmatrix}
b \\
a \\
b
\end{bmatrix}\!,\!
\begin{bmatrix}
b \\
b \\
a
\end{bmatrix}
\right)\!,
g'\left(
\begin{bmatrix}
a \\
b \\
b
\end{bmatrix}\!,\! 
\begin{bmatrix}
b \\
a \\
b
\end{bmatrix}\!,\!
\begin{bmatrix}
b \\
b \\
a
\end{bmatrix}
\right)\!,
g'\left(
\begin{bmatrix}
a \\
b \\
b
\end{bmatrix}\!,\! 
\begin{bmatrix}
b \\
a \\
b
\end{bmatrix}\!,\!
\begin{bmatrix}
b \\
a \\
b
\end{bmatrix}\!
\right)
\right)=
\end{equation}
$$
g'\left(
\begin{bmatrix}
a \\
a \\
a
\end{bmatrix}\!,\! 
\begin{bmatrix}
a \\
a \\
a
\end{bmatrix}\!,\!
\begin{bmatrix}
a \\
d \\
b
\end{bmatrix}
\right)=
\begin{bmatrix}
a \\
d \\
d
\end{bmatrix}\!,
$$
and
\begin{equation}\label{eq for g'' 2}
g''\left(
\begin{bmatrix}
b \\
a \\
a
\end{bmatrix}\!,\!
\begin{bmatrix}
a \\
b \\
a
\end{bmatrix}\!,\!
\begin{bmatrix}
a \\
a \\
b
\end{bmatrix}
\right) =g'\left(
g'\left(
\begin{bmatrix}
b \\
a \\
a
\end{bmatrix}\!,\! 
\begin{bmatrix}
a \\
b \\
a
\end{bmatrix}\!,\!
\begin{bmatrix}
a \\
a \\
b
\end{bmatrix}
\right)\!,
g'\left(
\begin{bmatrix}
b \\
a \\
a
\end{bmatrix}\!,\! 
\begin{bmatrix}
a \\
b \\
a
\end{bmatrix}\!,\!
\begin{bmatrix}
a \\
a \\
b
\end{bmatrix}
\right)\!,
g'\left(
\begin{bmatrix}
b \\
a \\
a
\end{bmatrix}\!,\! 
\begin{bmatrix}
a \\
b \\
a
\end{bmatrix}\!,\!
\begin{bmatrix}
a \\
b \\
a
\end{bmatrix}\!
\right)
\right)=
\end{equation}
$$
g'\left(
\begin{bmatrix}
d \\
d \\
d
\end{bmatrix}\!,\! 
\begin{bmatrix}
d \\
d \\
d
\end{bmatrix}\!,\!
\begin{bmatrix}
d \\
a \\
a
\end{bmatrix}
\right)=
\begin{bmatrix}
d \\
c \\
c
\end{bmatrix}\!.
$$
Now we define the ternary term $w$ by $w(x,y,z)=g'(y,z,g''(x,y,z))$, so, from \eqref{eq for g''} and \eqref{eq for g'' 2} we get
 \begin{equation}\label{eq for w}
w\left(
\begin{bmatrix}
a \\
b \\
b
\end{bmatrix}\!,\!
\begin{bmatrix}
b \\
a \\
b
\end{bmatrix}\!,\!
\begin{bmatrix}
b \\
b \\
a
\end{bmatrix}
\right)=
g'\left(
\begin{bmatrix}
b \\
a \\
b
\end{bmatrix}\!,\!
\begin{bmatrix}
b \\
b \\
a
\end{bmatrix}\!,\!
g''\left(
\begin{bmatrix}
a \\
b \\
b
\end{bmatrix}\!,\!
\begin{bmatrix}
b \\
a \\
b
\end{bmatrix}\!,\!
\begin{bmatrix}
b \\
b \\
a
\end{bmatrix}
\right)
\right)=
g'\left(
\begin{bmatrix}
b \\
a \\
b
\end{bmatrix}\!,\!
\begin{bmatrix}
b \\
b \\
a
\end{bmatrix}\!,\!
\begin{bmatrix}
a \\
d \\
d
\end{bmatrix}
\right)=
\begin{bmatrix}
a \\
a\\
a
\end{bmatrix},
\end{equation}
\begin{equation}\label{eq for w 2}
w\left(
\begin{bmatrix}
b \\
a \\
a
\end{bmatrix}\!,\!
\begin{bmatrix}
a \\
b \\
a
\end{bmatrix}\!,\!
\begin{bmatrix}
a \\
a \\
b
\end{bmatrix}
\right)=
g'\left(
\begin{bmatrix}
a \\
b \\
a
\end{bmatrix}\!,\!
\begin{bmatrix}
a \\
a \\
b
\end{bmatrix}\!,\!
g''\left(
\begin{bmatrix}
b \\
a \\
a
\end{bmatrix}\!,\!
\begin{bmatrix}
a \\
b \\
a
\end{bmatrix}\!,\!
\begin{bmatrix}
a \\
a \\
b
\end{bmatrix}
\right)
\right)=
g'\left(
\begin{bmatrix}
a \\
b \\
a
\end{bmatrix}\!,\!
\begin{bmatrix}
a \\
a \\
b
\end{bmatrix}\!,\!
\begin{bmatrix}
d \\
c \\
c
\end{bmatrix}
\right)=
\begin{bmatrix}
d \\
b\\
b
\end{bmatrix}.
\end{equation}
Finally, we define the cyclic term $g^*$ by $$g^*(x,y,z)=g'(w(x,y,z),w(y,z,x),w(z,x,y)).$$ From \eqref{eq for w} and \eqref{eq for w 2} follows that $\{a,b\}$ is closed for cyclic term $g^*$, which concludes the proof.
\end{proof}

\end{enumerate}

  \item[\textsc{\textbf{Case 3.}}] $\{a,c\}$ and $\{b,d\}$ are affine subalgebras.

Unlike in the previous case, we have two possibilities regarding $\mathrm{Sg}\{a,d\}$ since the case with a three-element algebra cannot happen.

\begin{enumerate}[wide, labelindent=0pt]
    \item[\textsc{\textbf{Subcase 3.1}}] $\mathrm{Sg}\{a,d\}=\{a,d\}$.
    
    The analysis in this case is similar to Subcase 2.1. The following three claims hold for any cyclic term.
    \begin{claim}\label{affine c(abb)=d}
    It holds $g(a,b,b)=c$.
    \end{claim}
    \begin{proof}
    For the sake of contradiction, suppose that $g(a,b,b)=a$, and thus $g(b,a,a)=d$. Let $t(x,y):=g(x,y,y)$ and define $g'$ by $$g'(x,y,z)=g(t(t(x,y),z),t(t(y,z),x),t(t(z,x),y)).$$ We see that $t(a,b)=t(a,d)=a$, $t(b,a)=t(d,a)=t(d,b)=d$, $t(b,d)=b$, and so
    \[
    \begin{gathered}
    g'(a,d,b)=g(t(a,b),t(d,a),t(d,d))=g(a,d,d)=a=g'(a,d,d)\\    g'(a,b,b)=g(t(a,b),t(b,a),t(d,b))=g(a,d,d)=a\\
    g'(a,b,d)=g(t(a,d),t(b,a),t(d,b))=g(a,d,d)=a.
    \end{gathered}
    \] 
    Hence, $\{a,b,d\}$ is closed under $g''$, a contradiction.
    \end{proof}
    \begin{claim}
    $\{b,c\}$ is a (affine) subalgebra.
    \end{claim}
    \begin{proof}
    Suppose that $\{b,c\}$ is not a subalgebra. Then $g(c,c,b)=d$ or $g(c,b,b)=a$. In either case, we have $d\in\mathrm{Sg}\{b,c\}$, thus $t(b,c)=d$ for some binary term $t$. Notice that on any two-element affine algebra, the term $t$ acts as the first projection. Define the cyclic term $g'$ by $$g'(x,y,z)=g(t(x,g(x,y,z)),t(y,g(y,z,x)),t(z,g(z,x,y))).$$ Then $g'(a,b,b)=g(t(a,c),t(b,c),t(b,c))=g(a,d,d)=a$, a contradiction with Claim~\ref{affine c(abb)=d}.
    \end{proof}
    \begin{claim}
    It holds $g(b,a,a)=d$.
    \end{claim}
    \begin{proof}
    Similar to the proof of the \Cref{affine c(abb)=d}.
    \end{proof}
    \begin{claim}
    It holds $g(b,c,d)=g(b,d,c)=c$, and, similarly, $g(a,c,d)=g(a,d,c)=d$.
    \end{claim}
    \begin{proof}
    Suppose, for the contradiction, that $g(b,d,c)=a$. Define the cyclic ternary term $g'$ by $$g'(x,y,z)=g(g(x,y,y),g(y,z,z),g(z,x,x)).$$ Then $g'(a,b,b)=g(c,b,d)=a$, which contradicts Claim~\ref{affine c(abb)=d}. The case $g(b,c,d)=a$ is handled similarly, just $g'$ is defined by $$g'(x,y,z)=g(g(x,x,y),g(y,y,z),g(z,z,x)).$$ The other two cases are analogous modulo switching $a$ with $b$ and $c$ with $d$.
    \end{proof}
    \begin{claim}
    There exists a cyclic term $g^*$ such that $g^*(a,b,b)=a$, which contradicts Claim~\ref{affine c(abb)=d}.
    \end{claim}
    \begin{proof}
    First, define the term $h$ by $h(x,y,z)=g(x,g(x,y,y),g(x,y,z))$ and then let $g^*$ be the cyclic term defined as $g^*(x,y,z):=g(h(x,y,z),h(y,z,x),h(z,x,y))$. We compute that \[g^*(a,b,b)=g(g(a,c,c),g(b,b,c),g(b,d,c))=g(a,c,c)=a,\] which contradicts Claim~\ref{affine c(abb)=d}.
    \end{proof}
    \item[\textsc{\textbf{Subcase 3.2}}] $\mathrm{Sg}\{a,d\}=\mathrm{Sg}\{b,c\}=\m a$.

    It may be assumed $\mathrm{Sg}\{c,d\}=\m a$ as well, since otherwise we would be in the previous case. Regarding how cyclic terms act on two-element subsets whose elements belong to distinct $\alpha$-classes, we have two options:
    \begin{itemize}
        \item[$i)$] there are $\circ,\ast\in A$ from different $\alpha$-classes such that $g(\circ,\ast,\ast)=\circ$ for some ternary cyclic term $g$;
        \item[$ii)$] for any $\circ,\ast\in A$ from different $\alpha$-classes and any ternary cyclic term $g$ it holds $g(\circ,\ast,\ast)\neq\circ$.
    \end{itemize}

    We first deal with the latter of these two cases. Before that, we note that throughout the subcase $g'$ and $g''$ shall be used for cyclic terms defined by $g(g(x,y,y),g(y,z,z),g(z,x,x))$ and $g(g(x,x,y),g(y,y,z),g(z,z,x))$, respectively, where $g$ is any cyclic ternary term. For any cyclic ternary term $g$, the following hold in case $ii)$:
    \begin{align*}
        g(a,b,b)=g(a,d,d)&=c  &  g(b,a,a)=g(b,c,c)&=d\\
         g(c,b,b)=g(c,d,d)&=a  &  g(d,a,a)=g(d,c,c)&=b
    \end{align*}
    These equations can be used to determine the remaining values of $g$. For example, $d=g'(a,a,b)=g(g(a,a,a),g(a,b,b),g(b,a,a))=g(a,c,d)$, so $g(a,c,d)=d$. Also, if $g''$ applied, we obtain $$d=g''(a,a,b)=g(g(a,a,a),g(a,a,b),g(b,b,a))=g(a,d,c),$$ thus $g(a,d,c)=d$. Similarly, we get $g(c,b,d)=g(d,b,c)=c$, $g(b,c,a)=g(a,c,b)=b$ and $g(a,d,b)=g(b,d,a)=a$. This completely determines the term $g$.

    \begin{claim}
    The clone generated by $g$ satisfying case $ii)$ is the same one from $\mathbb{Z}_4^{\mathrm{aff}}$, or, equivalently, the algebra $\m a$ is term-equivalent to $\mathbb{Z}_4^{\mathrm{aff}}$.
    \end{claim}
    \begin{proof}
    Since $\mathbb{Z}_4^{\mathrm{aff}}$ is known to be a minimal Taylor algebra, we shall prove the claim by finding a cyclic term $g$ in $\mathbb{Z}_4^{\mathrm{aff}}$. Let $p$ be a Mal'cev term in $\mathbb{Z}_4^{\mathrm{aff}}$, that is $p(x,y,z)=x-y+z$. Define a cyclic term $g^*$ by $g^*(x,y,z)=p(p(x,y,z),p(y,z,x),p(z,x,y))=-(x+y+z)$. We left as an exercise to prove that the term $g$ is nothing but $g^*$ if $a\mapsto 0$, $b\mapsto 1$, $c\mapsto 2$, $d\mapsto 3$.
    \end{proof}
    For the sake of counting, in this paper, the algebra $\mathbb{Z}_4^{\mathrm{aff}}$ is also denoted by $\m t_{4,12}$.

    From now on, we assume that $i)$ holds; therefore, it may be assumed that $g$ fulfills the conditions stated in $i)$. As we have $\mathrm{Sg}\{a,b\}=\mathrm{Sg}\{a,d\}=\mathrm{Sg}\{b,c\}=\mathrm{Sg}\{c,d\}=\m a$, we can suppose $g(a,b,b)=a$, and thus $g(b,a,a)=d$. Regarding some other undetermined values of $g$, we have five possibilities, as presented in \Cref{table of cases}. The three options not covered are isomorphic to one of the five cases after permuting along the cycle $a\rightarrow b\rightarrow c\rightarrow d\rightarrow a$ in some way: the columns $[a,b,a,d]^T$ and $[a,d,c,d]^T$ are actually case $2^\circ$, while the column $[a,d,a,d]^T$ is case $3^\circ$. It shall be proved that a minimal Taylor algebra is obtained only in the case $4^{\circ}$.

\begin{table}[h!]
    \centering
    \begin{tabular}{r||c|c|c|c|c}
\diagbox[width=4em]{$\mathbf x$}{$g(\mathbf x)$}& $1^{\circ}$ & $2^{\circ}$ & $3^{\circ}$ & $4^{\circ}$ & $5^{\circ}$\\ \hline\hline
$(a,b,b)$& $a$ & $a$ & $a$ & $a$ & $a$ \\ \hline
$(b,c,c)$& $b$ & $b$ & $b$ & $d$ & $d$ \\ \hline
$(c,d,d)$& $c$ & $c$ & $a$ & $c$ & $a$ \\ \hline
$(d,a,a)$& $d$ & $b$ & $b$ & $b$ & $b$ 
\end{tabular}
    \caption{Subcase $3.2\; i)$}
    \label{table of cases}
\end{table}

\begin{claim}
    Equations $g(b,c,c)=b$ and $g(c,d,d)=c$ cannot hold simultaneously; therefore, cases $1^{\circ}$ and $2^{\circ}$ are impossible.
\end{claim}
\begin{proof}
Suppose, on the contrary, that $g(b,c,c)=b$ (and thus $g(c,b,b)=a$) and $g(c,d,d)=c$ for some ternary cyclic term $g$. Define the ternary term $t$ by $$t(x,y,z)=g(g(x,g(x,y,z),g(x,y,z)),g(x,y,y),g(x,z,z)),$$ and the cyclic term $g^*$ by $g^*(x,y,z)=g(t(x,y,z),t(y,z,x),t(z,x,y))$. We have
\begin{equation*}
t\left(
\begin{bmatrix}
b \\
c \\
c
\end{bmatrix}\!,\!
\begin{bmatrix}
c \\
b \\
c
\end{bmatrix}\!,\!
\begin{bmatrix}
c \\
c \\
b
\end{bmatrix}
\right)=
g\left(
g\left(
\begin{bmatrix}
b \\
c \\
c
\end{bmatrix}\!,\!
\begin{bmatrix}
b \\
b \\
b
\end{bmatrix}\!,\!
\begin{bmatrix}
b \\
b \\
b
\end{bmatrix}
\right)\!,
\begin{bmatrix}
b \\
a \\
c
\end{bmatrix}\!,\!
\begin{bmatrix}
b \\
c \\
a
\end{bmatrix}
\right)=
g\left(
\begin{bmatrix}
b \\
a \\
a
\end{bmatrix}\!,\!
\begin{bmatrix}
b \\
a \\
c
\end{bmatrix}\!,\!
\begin{bmatrix}
b \\
c \\
a
\end{bmatrix}
\right)=
\begin{bmatrix}
b \\
c \\
c
\end{bmatrix},
\end{equation*}
so, $g^*(b,c,c)=b$ follows. Similarly, 
\begin{equation*}
t\left(
\begin{bmatrix}
c \\
b \\
b
\end{bmatrix}\!,\!
\begin{bmatrix}
b \\
c \\
b
\end{bmatrix}\!,\!
\begin{bmatrix}
b \\
b \\
c
\end{bmatrix}
\right)=
g\left(
g\left(
\begin{bmatrix}
c \\
b \\
b
\end{bmatrix}\!,\!
\begin{bmatrix}
a \\
a \\
a
\end{bmatrix}\!,\!
\begin{bmatrix}
a \\
a \\
a
\end{bmatrix}
\right)\!,
\begin{bmatrix}
a \\
b \\
b
\end{bmatrix}\!,\!
\begin{bmatrix}
a \\
b \\
b
\end{bmatrix}
\right)=
g\left(
\begin{bmatrix}
c \\
d \\
d
\end{bmatrix}\!,\!
\begin{bmatrix}
a \\
b \\
b
\end{bmatrix}\!,\!
\begin{bmatrix}
a \\
b \\
b
\end{bmatrix}
\right)=
\begin{bmatrix}
c \\
d \\
d
\end{bmatrix},
\end{equation*}
thus, $g^*(c,b,b)=c$ holds. The set $\{b,c\}$ is closed under the cyclic term $g^*$, so $\mathrm{Sg}\{b,c\}=\{b,c\}$, a contradiction.
\end{proof}

\begin{claim}\label{g(abd)=c=g(adb)claim}
    If $g(d,a,a)=b$, then $g(a,b,d)=g(a,d,b)=c$.
\end{claim}

\begin{proof}
    From $g'(b,a,a)=g(d,a,a)=b$ we get $c=g'(a,b,b)=g(a,b,d)$, and similarly $g''(b,a,a)=g(a,a,d)=b$ implies $c=g''(a,b,b)=g(d,b,a)$.
\end{proof}

\begin{claim}
    Equations $g(b,c,c)=b$ and $g(d,a,a)=b$ cannot hold simultaneously; therefore, the case $3^{\circ}$ is impossible.
\end{claim}

\begin{proof}
Assume, for the sake of contradiction, that $g(b,c,c)=b$ (therefore $g(c,b,b)=a$) and $g(d,a,a)=b$ hold. By Claim~\ref{g(abd)=c=g(adb)claim}, $g(a,b,d)=g(a,d,b)=c$ holds. Define the ternary term $t$ by $$t(x,y,z)=g(x,g(x,y,y),g(x,z,z)),$$ so, we have 
 \begin{equation*}
t\left(
\begin{bmatrix}
a \\
b \\
b
\end{bmatrix}\!,\!
\begin{bmatrix}
b \\
a \\
b
\end{bmatrix}\!,\!
\begin{bmatrix}
b \\
b \\
a
\end{bmatrix}
\right)=
g\left(
\begin{bmatrix}
a \\
b \\
b
\end{bmatrix}\!,\!
g\left(
\begin{bmatrix}
a \\
b \\
b
\end{bmatrix}\!,\!
\begin{bmatrix}
b \\
a \\
b
\end{bmatrix}\!,\!
\begin{bmatrix}
b \\
a \\
b
\end{bmatrix}
\right)\!,\!
g\left(
\begin{bmatrix}
a \\
b \\
b
\end{bmatrix}\!,\!
\begin{bmatrix}
b \\
b \\
a
\end{bmatrix}\!,\!
\begin{bmatrix}
b \\
b \\
a
\end{bmatrix}
\right)
\right)=
g\left(
\begin{bmatrix}
a \\
b \\
b
\end{bmatrix}\!,\!
\begin{bmatrix}
a \\
d \\
b
\end{bmatrix}\!,\!
\begin{bmatrix}
a \\
b \\
d
\end{bmatrix}
\right)=
\begin{bmatrix}
a \\
d\\
d
\end{bmatrix},
\end{equation*}
 \begin{equation*}
t\left(
\begin{bmatrix}
b \\
a \\
a
\end{bmatrix}\!,\!
\begin{bmatrix}
a \\
b \\
a
\end{bmatrix}\!,\!
\begin{bmatrix}
a \\
a \\
b
\end{bmatrix}
\right)=
g\left(
\begin{bmatrix}
b \\
a \\
a
\end{bmatrix}\!,\!
g\left(
\begin{bmatrix}
b \\
a \\
a
\end{bmatrix}\!,\!
\begin{bmatrix}
a \\
b \\
a
\end{bmatrix}\!,\!
\begin{bmatrix}
a \\
b \\
a
\end{bmatrix}
\right)\!,\!
g\left(
\begin{bmatrix}
b \\
a \\
a
\end{bmatrix}\!,\!
\begin{bmatrix}
a \\
a \\
b
\end{bmatrix}\!,\!
\begin{bmatrix}
a \\
a \\
b
\end{bmatrix}
\right)
\right)=
g\left(
\begin{bmatrix}
b \\
a \\
a
\end{bmatrix}\!,\!
\begin{bmatrix}
d \\
a \\
a
\end{bmatrix}\!,\!
\begin{bmatrix}
d \\
a \\
a
\end{bmatrix}
\right)=
\begin{bmatrix}
b \\
a\\
a
\end{bmatrix}.
\end{equation*}
Further, let the ternary term $t'$ be defined by $$t'(x,y,z)=g(x,g(x,y,z),t(x,y,z)),$$ 
hence, we get 
\begin{equation}\label{t'values}
t'\left(
\begin{bmatrix}
a \\
b \\
b
\end{bmatrix}\!,\!
\begin{bmatrix}
b \\
a \\
b
\end{bmatrix}\!,\!
\begin{bmatrix}
b \\
b \\
a
\end{bmatrix}
\right)=
g\left(
\begin{bmatrix}
a \\
b \\
b
\end{bmatrix}\!,\!
\begin{bmatrix}
a \\
a \\
a
\end{bmatrix}\!,\!
\begin{bmatrix}
a \\
d \\
d
\end{bmatrix}
\right)=
\begin{bmatrix}
a \\
c\\
c
\end{bmatrix},
\end{equation}
\begin{equation}\label{t'values2}
t'\left(
\begin{bmatrix}
b \\
a \\
a
\end{bmatrix}\!,\!
\begin{bmatrix}
a \\
b \\
a
\end{bmatrix}\!,\!
\begin{bmatrix}
a \\
a \\
b
\end{bmatrix}
\right)=
g\left(
\begin{bmatrix}
b \\
a \\
a
\end{bmatrix}\!,\!
\begin{bmatrix}
d \\
d \\
d
\end{bmatrix}\!,\!
\begin{bmatrix}
b \\
a \\
a
\end{bmatrix}
\right)=
\begin{bmatrix}
d \\
b\\
b
\end{bmatrix}.
\end{equation}
Define now the ternary term $t''$ by $$t''(x,y,z)=g(x,g''(x,y,z),t'(x,y,z)).$$
Then from \eqref{t'values} and \eqref{t'values2} we have
\begin{equation}\label{t''values}
t''\left(
\begin{bmatrix}
a \\
b \\
b
\end{bmatrix}\!,\!
\begin{bmatrix}
b \\
a \\
b
\end{bmatrix}\!,\!
\begin{bmatrix}
b \\
b \\
a
\end{bmatrix}
\right)=
g\left(
\begin{bmatrix}
a \\
b \\
b
\end{bmatrix}\!,\!
\begin{bmatrix}
c \\
c \\
c
\end{bmatrix}\!,\!
\begin{bmatrix}
a \\
c \\
c
\end{bmatrix}
\right)=
\begin{bmatrix}
c \\
b\\
b
\end{bmatrix},
\end{equation}
\begin{equation}\label{t''values2}
t''\left(
\begin{bmatrix}
b \\
a \\
a
\end{bmatrix}\!,\!
\begin{bmatrix}
a \\
b \\
a
\end{bmatrix}\!,\!
\begin{bmatrix}
a \\
a \\
b
\end{bmatrix}
\right)=
g\left(
\begin{bmatrix}
b \\
a \\
a
\end{bmatrix}\!,\!
\begin{bmatrix}
b \\
b \\
b
\end{bmatrix}\!,\!
\begin{bmatrix}
d \\
b \\
b
\end{bmatrix}
\right)=
\begin{bmatrix}
d \\
a\\
a
\end{bmatrix}.
\end{equation}
Finally, if we define the cyclic term $g^*$ by 
$$
g^*(x,y,z)=g(t''(x,y,z),t''(y,z,x),t''(z,x,y)),
$$
from \eqref{t''values} and \eqref{t''values2} we get $g^*(a,b,b)=a$ and $g^*(b,a,a)=b$, so, $\mathrm{Sg}\{a,b\}=\{a,b\}$, a contradiction.
\end{proof}

\begin{claim}
The asserted equations for the case $5^{\circ}$ cannot hold simultaneously; therefore, the case is impossible.
\end{claim}

\begin{proof}
By Claim~\ref{g(abd)=c=g(adb)claim}, $g(a,b,d) = g(a,d,b) = c$. Define the term $t$ by $$t(x,y,z)=g(y,z,g(x,y,y)),$$ so, we get 
 \begin{equation*}
t\left(
\begin{bmatrix}
a \\
b \\
b
\end{bmatrix}\!,\!
\begin{bmatrix}
b \\
a \\
b
\end{bmatrix}\!,\!
\begin{bmatrix}
b \\
b \\
a
\end{bmatrix}
\right)=
g\left(
\begin{bmatrix}
b \\
a \\
b
\end{bmatrix}\!,\!
\begin{bmatrix}
b \\
b \\
a
\end{bmatrix}\!,\!
g\left(
\begin{bmatrix}
a \\
b \\
b
\end{bmatrix}\!,\!
\begin{bmatrix}
b \\
a \\
b
\end{bmatrix}\!,\!
\begin{bmatrix}
b \\
a \\
b
\end{bmatrix}
\right)
\right)=
g\left(
\begin{bmatrix}
b \\
a \\
b
\end{bmatrix}\!,\!
\begin{bmatrix}
b \\
b \\
a
\end{bmatrix}\!,\!
\begin{bmatrix}
a \\
d \\
b
\end{bmatrix}
\right)=
\begin{bmatrix}
a \\
c\\
a
\end{bmatrix},
\end{equation*}
\begin{equation*}
t\left(
\begin{bmatrix}
b \\
a \\
a
\end{bmatrix}\!,\!
\begin{bmatrix}
a \\
b \\
a
\end{bmatrix}\!,\!
\begin{bmatrix}
a \\
a \\
b
\end{bmatrix}
\right)=
g\left(
\begin{bmatrix}
a \\
b \\
a
\end{bmatrix}\!,\!
\begin{bmatrix}
a \\
a \\
b
\end{bmatrix}\!,\!
g\left(
\begin{bmatrix}
b \\
a \\
a
\end{bmatrix}\!,\!
\begin{bmatrix}
a \\
b \\
a
\end{bmatrix}\!,\!
\begin{bmatrix}
a \\
b \\
a
\end{bmatrix}
\right)
\right)=
g\left(
\begin{bmatrix}
a \\
b \\
a
\end{bmatrix}\!,\!
\begin{bmatrix}
a \\
a \\
b
\end{bmatrix}\!,\!
\begin{bmatrix}
d \\
a \\
a
\end{bmatrix}
\right)=
\begin{bmatrix}
b \\
d\\
d
\end{bmatrix}.
\end{equation*}
Define the ternary term $t'$ by $$t'(x,y,z)=g(x,x,t(x,y,z)),$$ hence, we get 
$$
\begin{gathered}
t'\left(
\begin{bmatrix}
a \\
b \\
b
\end{bmatrix}\!,\!
\begin{bmatrix}
b \\
a \\
b
\end{bmatrix}\!,\!
\begin{bmatrix}
b \\
b \\
a
\end{bmatrix}
\right)=
g\left(
\begin{bmatrix}
a \\
b \\
b
\end{bmatrix}\!,\!
\begin{bmatrix}
a \\
b \\
b
\end{bmatrix}\!,\!
\begin{bmatrix}
a \\
c \\
a
\end{bmatrix}
\right)=
\begin{bmatrix}
a \\
g(b,b,c) \\
a
\end{bmatrix}\quad\text{and}\\
t'\left(
\begin{bmatrix}
b \\
a \\
a
\end{bmatrix}\!,\!
\begin{bmatrix}
a \\
b \\
a
\end{bmatrix}\!,\!
\begin{bmatrix}
a \\
a \\
b
\end{bmatrix}
\right)=
g\left(
\begin{bmatrix}
b \\
a \\
a
\end{bmatrix}\!,\!
\begin{bmatrix}
b \\
a \\
a
\end{bmatrix}\!,\!
\begin{bmatrix}
b \\
d \\
d
\end{bmatrix}
\right)=
\begin{bmatrix}
b \\
b \\
b
\end{bmatrix}.
\end{gathered}
$$
We may assume that $g(b,b,c)=c$ holds, since otherwise we get that $\{a,b\}$ is closed under the cyclic term $w$ defined by $$w(x,y,z)=g(t'(x,y,z),t'(y,z,x),t'(z,x,y)),$$ which is a contradiction. Hence, let $g(b,b,c)=c$ and define the ternary term $t''$ by $$t''(x,y,z)=g(g(x,z,z),g(x,z,z),g''(x,y,z)).$$ Then 
$$
\begin{gathered}
t''\left(
\begin{bmatrix}
a \\
b \\
b
\end{bmatrix}\!,\!
\begin{bmatrix}
b \\
a \\
b
\end{bmatrix}\!,\!
\begin{bmatrix}
b \\
b \\
a
\end{bmatrix}
\right)=
g\left(
\begin{bmatrix}
a \\
b \\
d
\end{bmatrix}\!,\!
\begin{bmatrix}
a \\
b \\
d
\end{bmatrix}\!,\!
\begin{bmatrix}
c \\
c \\
c
\end{bmatrix}
\right)=
\begin{bmatrix}
c \\
c \\
a
\end{bmatrix}\quad\text{and}\\
t''\left(
\begin{bmatrix}
b \\
a \\
a
\end{bmatrix}\!,\!
\begin{bmatrix}
a \\
b \\
a
\end{bmatrix}\!,\!
\begin{bmatrix}
a \\
a \\
b
\end{bmatrix}
\right)=
g\left(
\begin{bmatrix}
d \\
a \\
a
\end{bmatrix}\!,\!
\begin{bmatrix}
d \\
a \\
a
\end{bmatrix}\!,\!
\begin{bmatrix}
b \\
b \\
b
\end{bmatrix}
\right)=
\begin{bmatrix}
b \\
d \\
d
\end{bmatrix}.
\end{gathered}
$$
From these two equations, we get that $\{a,b\}$ is closed for the cyclic term $g^*$ defined by $$g^*(x,y,z)=g(t''(x,y,z),t''(y,z,x),t''(z,x,y)),$$ a contradiction.
\end{proof}

\begin{claim}
If the equations of the case $4^{\circ}$ hold, then the following also hold:
\begin{enumerate}
    \item[1)] $g(a,d,d)=g(b,c,d)=g(b,d,c)=a$;
    \item[2)] $g(c,b,b)=g(a,b,d)=g(a,d,b)=c$;
    \item[3)] $g(d,c,c)=g(a,b,c)=g(a,c,b)=b$;
    \item[4)] $g(a,c,d)=g(a,d,c)=d$.
\end{enumerate}
Consequently, the term $g$ is completely determined, and a minimal Taylor algebra $\m a$ is the one from \Cref{example5.17}.
\end{claim}
\begin{proof}
First, $g(d,c,c)=b$ follows from $g(c,d,d)=c$ and $\mathrm{Sg}\{c,d\}=\m a$. By Claim~\ref{g(abd)=c=g(adb)claim} and its dual obtained by transposing $a$ with $c$ and $b$ with $d$, we deduce $g(a,b,d)=g(a,d,b)=c$ and $g(b,c,d)=g(b,d,c)=a$.

Next, we shall prove $g(a,d,d)=a$. If $g(a,d,d)=c$, then $g'(a,d,d)=g(c,d,b)=a$. Since $g'(b,a,a)=b$ and $g'(d,c,c)=d$ hold as well, we may recall case $1^{\circ}$ if $g'(c,b,b)=c$, or case $2^{\circ}$ if $g'(c,b,b)=a$, which are both impossible. Therefore, it may be assumed that $g(a,d,d)=a$ holds, which concludes the proof of the claim's first part. The corresponding claim $g(c,b,b)=c$ is proved similarly.

Let us now prove $g(a,b,c)=g(a,c,b)=b$. First, for the sake of contradiction, suppose that $g(a,b,c)=d$. Define $t_1$ and $g_1$ by 
\begin{align*}
    t_1(x,y,z)&=g(z,g(x,y,z),g(x,y,y)),\\
    g_1(x,y,z)&=g(t_1(x,y,z),t_1(y,z,x),t_1(z,x,y)).
\end{align*}
Then 
$$
\begin{gathered}
t_1\left(
\begin{bmatrix}
a \\
b \\
b
\end{bmatrix}\!,\!
\begin{bmatrix}
b \\
a \\
b
\end{bmatrix}\!,\!
\begin{bmatrix}
b \\
b \\
a
\end{bmatrix}
\right)=
g\left(
\begin{bmatrix}
b \\
b \\
a
\end{bmatrix}\!,\!
\begin{bmatrix}
a \\
a \\
a
\end{bmatrix}\!,\!
\begin{bmatrix}
a \\
d \\
b
\end{bmatrix}
\right)=
\begin{bmatrix}
d \\
c \\
d
\end{bmatrix}\quad\text{and}\\
t_1\left(
\begin{bmatrix}
b \\
a \\
a
\end{bmatrix}\!,\!
\begin{bmatrix}
a \\
b \\
a
\end{bmatrix}\!,\!
\begin{bmatrix}
a \\
a \\
b
\end{bmatrix}
\right)=
g\left(
\begin{bmatrix}
a \\
a \\
b
\end{bmatrix}\!,\!
\begin{bmatrix}
d \\
d \\
d
\end{bmatrix}\!,\!
\begin{bmatrix}
d \\
a \\
a
\end{bmatrix}
\right)=
\begin{bmatrix}
a \\
b \\
c
\end{bmatrix}.
\end{gathered}
$$
We get $g_1(a,b,b)=c$ and, using our assumption $g(a,b,c)=d$, $g_1(b,a,a)=d$. Next we define $g_2$ by 
\begin{align*}
    g_2(x,y,z)&=g(g(x,y,z),g'(x,y,z),g_1(x,y,z)).
\end{align*}
Then $g_2$ is cyclic (being an operation applied to three cyclic operations) and
$$
\begin{gathered}
g_2\left(
\begin{bmatrix}
a \\
b \\
b
\end{bmatrix}\!,\!
\begin{bmatrix}
b \\
a \\
b
\end{bmatrix}\!,\!
\begin{bmatrix}
b \\
b \\
a
\end{bmatrix}
\right)=
g\left(
\begin{bmatrix}
a \\
a \\
a
\end{bmatrix}\!,\!
\begin{bmatrix}
c \\
c \\
c
\end{bmatrix}\!,\!
\begin{bmatrix}
c \\
c \\
c
\end{bmatrix}
\right)=
\begin{bmatrix}
a \\
a \\
a
\end{bmatrix}\quad\text{and}\\
g_2\left(
\begin{bmatrix}
b \\
a \\
a
\end{bmatrix}\!,\!
\begin{bmatrix}
a \\
b \\
a
\end{bmatrix}\!,\!
\begin{bmatrix}
a \\
a \\
b
\end{bmatrix}
\right)=
g\left(
\begin{bmatrix}
d \\
d \\
d
\end{bmatrix}\!,\!
\begin{bmatrix}
b \\
b \\
b
\end{bmatrix}\!,\!
\begin{bmatrix}
d \\
d \\
d
\end{bmatrix}
\right)=
\begin{bmatrix}
b \\
b \\
b
\end{bmatrix}.
\end{gathered}
$$
We get $g_2(a,b,b)=a$ and $g_2(b,a,a)=b$ for the cyclic term $g_2$, a contradiction. This proves $g(a,b,c)=b$. Suppose now that $g(a,c,b)=d$ holds. Here we define $t_3$ and $g_3$ by 
$$
\begin{gathered}
t_3(x,y,z)=g(y,g(x,y,z),g(x,y,y))\;\;\text{and}\\
g_3(x,y,z)=g(t_3(x,y,z),t_3(y,z,x),t_3(z,x,y)). 
\end{gathered}
$$
Now, 
$$
\begin{gathered}
t_3\left(
\begin{bmatrix}
a \\
b \\
b
\end{bmatrix}\!,\!
\begin{bmatrix}
b \\
a \\
b
\end{bmatrix}\!,\!
\begin{bmatrix}
b \\
b \\
a
\end{bmatrix}
\right)=
g\left(
\begin{bmatrix}
b \\
a \\
b
\end{bmatrix}\!,\!
\begin{bmatrix}
a \\
a \\
a
\end{bmatrix}\!,\!
\begin{bmatrix}
a \\
d \\
b
\end{bmatrix}
\right)=
\begin{bmatrix}
d \\
b \\
a
\end{bmatrix}\quad\text{and}\\
t_3\left(
\begin{bmatrix}
b \\
a \\
a
\end{bmatrix}\!,\!
\begin{bmatrix}
a \\
b \\
a
\end{bmatrix}\!,\!
\begin{bmatrix}
a \\
a \\
b
\end{bmatrix}
\right)=
g\left(
\begin{bmatrix}
a \\
b \\
a
\end{bmatrix}\!,\!
\begin{bmatrix}
d \\
d \\
d
\end{bmatrix}\!,\!
\begin{bmatrix}
d \\
a \\
a
\end{bmatrix}
\right)=
\begin{bmatrix}
a \\
c \\
b
\end{bmatrix}.
\end{gathered}
$$
Since $g_3(a,b,b)=c$ and $g_3(b,a,a)=d$, analogously as above replacing $g_1$ with $g_3$, we get that $\{a,b\}$ is closed under some cyclic term, a contradiction. The proof of part $3)$ is finished.

By transposing $a$ with $c$, and also $b$ with $d$, in the above arguments, we prove the last part of the claim.

The term $g$ is completely determined now. We shall prove that a minimal Taylor algebra from \Cref{example5.17} is obtained in this case.

That can be done by finding any of cyclic terms $g,g',g''$ in the minimal Taylor algebra from \Cref{example5.17}. Let $g^*$ be a term defined by $g^*(x,y,z)=p(p(x,y,z),p(y,z,x),p(z,x,y))$. The verification of $g^*$ being equal to $g'$ is left to the reader. In the present paper, we denote this algebra by $\m t_{4,13}$.
\end{proof}

\end{enumerate}
\end{enumerate}

\subsection{Affine $\mathbb{Z}^{\text{aff}}_3$ quotient}\label{s:Z3}
Suppose that $\alpha=\{\{a,d\},\{b\},\{c\}\}$ is a congruence such that $\m a/\alpha$ is term-equivalent to $\mathbb{Z}^{\text{aff}}_3$. Recall the operation from \Cref{compositionactslike}; on $\mathbb{Z}^{\text{aff}}_3$ it acts like $p(x,y,z)=x-y+z$. We distinguish three cases depending on which algebra is the class $\{a,d\}$: the semilattice, the majority and the affine class.

\begin{table}[H]
\begin{subtable}{.45\linewidth}
\centering
 \begin{tabular}{c|c c c c}

    $t'$ & $a$ & $b$& $c$  & $d$\\
\hline

    $a$ & $a$ & $c$ & $b$ & $a$  \\

	$b$ & $c$ & $b$ & $\ast$ & $c$ \\
	
	$c$ & $b$ & $\ast$ & $c$ & $b$\\
	
	$d$ & $a$ & $c$ & $b$ & $d$
\end{tabular}  
\caption{General case}
\label{table:8}
\end{subtable}
\begin{subtable}{.45\linewidth}
\centering
 \begin{tabular}{c|c c c c}

    $t'$ & $a$ & $b$& $c$  & $d$\\
\hline

    $a$ & $a$ & $c$ & $b$ & $a$  \\

	$b$ & $c$ & $b$ & $d$ & $c$ \\
	
	$c$ & $b$ & $d$ & $c$ & $b$\\
	
	$d$ & $a$ & $c$ & $b$ & $d$
\end{tabular}  
\caption{Algebra $\m t_{4,14}$ ($\ast=d$)}
\label{table:9}
\end{subtable}
\caption{Algebras with $\mathbb{Z}^{\text{aff}}_3$ quotient and commutative binary operation.}
\end{table}

\begin{table}[H]
\begin{subtable}{.45\linewidth}
    \centering
    \begin{tabular}{c c|c c}
    $\mathbf{x}$ & $p(\mathbf{x})$ &$\mathbf{x}$  & $p(\mathbf{x})$\\ \hline\hline
     $(a,a,b)$ & $b$ & $(c,a,a)$ & $c$ \\
     $(a,a,c)$ & $c$ & $(c,a,c)$ & $b$\\
     $(a,b,a)$ & $c$ & $(c,a,d)$ & $c$ \\
     $(a,b,c)$ & $b$ & $(c,b,a)$ & $b$\\
     $(a,b,d)$ & $c$ & $(c,b,b)$ & $c$ \\
     $(a,c,a)$ & $b$ & $(c,b,d)$ & $b$ \\
     $(a,c,b)$ & $c$ & $(c,c,b)$ & $b$ \\
     $(a,c,d)$ & $b$ & $(c,d,a)$ & $c$ \\
     $(a,d,b)$ & $b$ & $(c,d,c)$ & $b$ \\
     $(a,d,c)$ & $c$ & $(c,d,d)$ & $c$ \\[1ex]
     $(b,a,a)$ & $b$ & $(d,a,b)$ & $b$ \\
     $(b,a,b)$ & $c$ & $(d,a,c)$ & $c$ \\
     $(b,a,d)$ & $b$ & $(d,b,a)$ & $c$ \\
     $(b,b,c)$ & $c$ & $(d,b,c)$ & $b$ \\
     $(b,c,a)$ & $c$ & $(d,b,d)$ & $c$ \\
     $(b,c,c)$ & $b$ & $(d,c,a)$ & $b$ \\
     $(b,c,d)$ & $c$ & $(d,c,b)$ & $c$ \\
     $(b,d,a)$ & $b$ & $(d,c,d)$ & $b$ \\
     $(b,d,b)$ & $c$ & $(d,d,b)$ & $b$ \\
     $(b,d,d)$ & $b$ & $(d,d,c)$ & $c$     
\end{tabular}
    \caption{Common values}
    \label{3affcommon}
\end{subtable}
\begin{subtable}{.5\linewidth}
\centering
 \begin{tabular}{|c||c|c|c|c|c|}
\hline
    Algebra & $\T_{4,15}$ & $\T_{4,16}$ & $\T_{4,17}$ & $\T_{4,18}$ \\
\hline\hline
     $p{\restriction_{\{a,d\}}}$  & maj & maj & aff & aff \\
\hline
    $p(a,b,b)$  & $d$ & $a$ & $d$ & $a$ \\
\hline
	$p(a,c,c)$ & $d$ & $d$ & $d$ & $d$ \\
\hline
	$p(b,a,c)$ & $d$ & $a$ & $a$ & $d$ \\	
\hline

    $p(b,b,a)$  & $d$ & $a$ & $d$ & $a$ \\
\hline
	$p(b,b,d)$  & $d$ & $a$ & $d$ & $a$ \\
\hline

        $p(b,c,b)$  & $d$ & $d$ & $d$ & $d$ \\
\hline
	$p(b,d,c)$  & $d$ & $d$ & $d$ & $d$ \\
\hline
 
	$p(c,a,b)$ & $d$ & $a$ & $a$ & $d$ \\
\hline
	$p(c,b,c)$ & $d$ & $a$ & $d$ & $a$ \\

\hline
	$p(c,c,a)$ & $d$ & $d$ & $d$ & $d$ \\
\hline
	$p(c,c,d)$ & $d$ & $d$ & $d$ & $d$ \\
\hline
	$p(c,d,b)$ & $d$ & $d$ & $d$ & $a$ \\
\hline
	$p(d,b,b)$ & $d$ & $a$ & $d$ & $a$ \\
\hline
	$p(d,c,c)$ & $d$ & $d$ & $d$ & $d$ \\
\hline

\end{tabular} 
\caption{Remaining values.}
\label{3affremain}
\end{subtable}
\caption{Algebras with $\mathbb{Z}^{\text{aff}}_3$ quotient and a ternary operation.}
\end{table}

\begin{figure}[ht]
    \begin{subfigure}{0.33\textwidth}
		\centering
		\begin{tikzpicture}
			\tikzstyle{every node}=[draw,circle,fill=white,minimum size=4pt,inner sep=0pt]
			
			\node (a) [label=above:\strut$a$] {};
			\node (b) [fill=gray,right=1.5cm of a,label=above:\strut$b$] {};
			\node (c) [fill=gray!20,below=1.5cm of a,label=below:\strut$c$] {};
			\node (d) [right=1.5cm of c,label=below:\strut$d$] {};

			\path[->,>=stealth,line width=1.2pt, shorten <=0.7mm, shorten >=0.7mm]
            (a) edge node[draw=none,near start,left, outer sep=3pt] {\scriptsize{s}} (d);
			
			\path[-,line width=1.2pt,densely dotted, shorten <=0.7mm, shorten >=0.7mm]
            (a) edge node[draw=none,midway,above, outer sep=1pt] {\scriptsize{a3}} (b)
			(b) edge node[draw=none,near start,right, outer sep=2pt] {\scriptsize{a3}} (c)
			(c) edge node[draw=none,midway,above,outer sep=1pt] {\scriptsize{a3}} (d)
            (b) edge node[draw=none,midway,right,outer sep=1pt] {\scriptsize{a3}} (d)
			(a) edge node[draw=none,midway,left,outer sep=1pt] {\scriptsize{a3}} (c);

		\end{tikzpicture}
		\caption{$\m t_{4,14}$}
		
	\end{subfigure}%
	\hfill
	\begin{subfigure}{0.33\textwidth}
		\centering
		\begin{tikzpicture}
			\tikzstyle{every node}=[draw,circle,fill=white,minimum size=4pt,inner sep=0pt]
			
			\node (a) [label=above:\strut$a$] {};
			\node (b) [fill=gray,right=1.5cm of a,label=above:\strut$b$] {};
			\node (c) [fill=gray!20,below=1.5cm of a,label=below:\strut$c$] {};
			\node (d) [right=1.5cm of c,label=below:\strut$d$] {};

            \path[-,line width=1.2pt,densely dashed, shorten <=0.7mm, shorten >=0.7mm]
            (a) edge node[draw=none,near start,left, outer sep=3pt] {\scriptsize{m}} (d);
			
			\path[-,line width=1.2pt,densely dotted, shorten <=0.7mm, shorten >=0.7mm]
            (a) edge node[draw=none,midway,above, outer sep=1pt] {\scriptsize{a3}} (b)
			(b) edge node[draw=none,near start,right, outer sep=2pt] {\scriptsize{a3}} (c)
			(c) edge node[draw=none,midway,above,outer sep=1pt] {\scriptsize{a3}} (d)
            (b) edge node[draw=none,midway,right,outer sep=1pt] {\scriptsize{a3}} (d)
			(a) edge node[draw=none,midway,left,outer sep=1pt] {\scriptsize{a3}} (c);

		\end{tikzpicture}
		\caption{$\m t_{4,15}$, $\m t_{4,16}$}
	\end{subfigure}
	\hfill
	\begin{subfigure}{0.33\textwidth}
		\centering
		\begin{tikzpicture}
			\tikzstyle{every node}=[draw,circle,fill=white,minimum size=4pt,inner sep=0pt]
			
			\node (a) [label=above:\strut$a$] {};
			\node (b) [fill=gray,right=1.5cm of a,label=above:\strut$b$] {};
			\node (c) [fill=gray!20,below=1.5cm of a,label=below:\strut$c$] {};
			\node (d) [right=1.5cm of c,label=below:\strut$d$] {};

            \path[-,line width=1.2pt,densely dotted, shorten <=0.7mm, shorten >=0.7mm]
            (a) edge node[draw=none,near start,left, outer sep=3pt] {\scriptsize{a}} (d)
			(a) edge node[draw=none,midway,above, outer sep=1pt] {\scriptsize{a3}} (b)
			(b) edge node[draw=none,near start,right, outer sep=2pt] {\scriptsize{a3}} (c)
			(c) edge node[draw=none,midway,above,outer sep=1pt] {\scriptsize{a3}} (d)
            (b) edge node[draw=none,midway,right,outer sep=1pt] {\scriptsize{a3}} (d)
			(a) edge node[draw=none,midway,left,outer sep=1pt] {\scriptsize{a3}} (c);;
		\end{tikzpicture}
		\caption{$\m t_{4,17}$, $\m t_{4,18}$}
	\end{subfigure}
	\caption{Directed graph of $\m t_{4,i}$ for $i=14,\dots,18$.}
\end{figure}
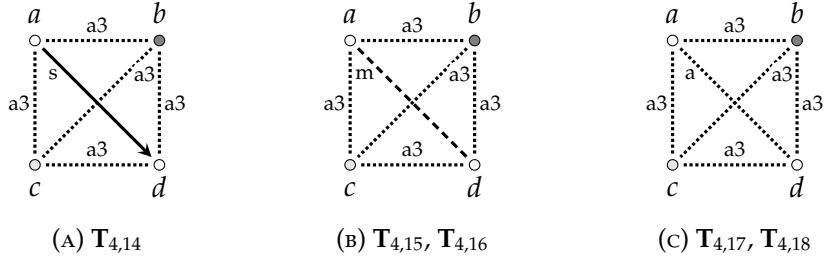

\begin{claim}\label{claim3affsemiblock}

If the two-element class of $\m a/\alpha$ is a semilattice, then  $\m a\cong\m t_{4,7}$ or $\m a\cong\m t_{4,14}$.
\end{claim}
\begin{proof}
Suppose that the class $\{a,d\}$ is the semilattice with the absorbing element $a$. Affine algebra $\mathbb{Z}^{\text{aff}}_3$ has only one commutative binary term, that is $t(x,y):=p(x,y,x)=2x+2y$. Since in this case each class of $\m a/\alpha$ also has it, such a term exists in the algebra $\m a$ as well. Denote by $t'$ one such term.
In particular, if we denote classes of $\alpha$ by $C_1,C_2,C_3$, then $t'(C_i,C_j)\subseteq C_k$, whenever $\{i,j,k\}=\{1,2,3\}$. Table of operation $t'$ is given in \Cref{table:8}, where $\ast$ represents the only undetermined value, that is $t'(b,c)=t'(c,b)\in\{a,d\}$. 
If $\ast=a$, then $\{a,b,c\}$ is the underlying set of $\mathbb{Z}^{\text{aff}}_3$. Moreover, we have $t'(d,\circ)\in\{a,b,c\}$ if $\circ\in\{a,b,c\}$, so $\{a,b,c\}\trianglelefteq_2\m a$. Therefore, $\m a$ also has a semilattice factor, which was examined before. The reader can verify that $\m a\cong\m t_{4,7}$ is true. Therefore, let $\ast=d$. This is algebra $\m t_{4,14}$ from \Cref{table:9}, which is a minimal Taylor algebra as shown in Example 5.20 of \cite{dreamteam}.
\end{proof}
\begin{rem}
In the later case of Claim~\ref{claim3affsemiblock}, we have $\mathrm{Sg}\{b,d\}=\mathrm{Sg}\{c,d\}=\m a$.
\end{rem}

\begin{claim}

If the two-element class of $\m a/\alpha$ is a majority subalgebra, then  $\m a\cong\m t_{4,15}$ or $\m a\cong\m t_{4,16}$.
\end{claim}
\begin{proof}
\Cref{3affcommon} shows the determined values of $p$, while on $\{a,d\}$ it acts as the majority because $\{a,d\}$ is assumed to be majority subalgebra. Let $$p'(x,y,z)=p(p(x,z,x),y,p(z,x,z)).$$ Hence, the following equations hold
\begin{align*}
p'(a,b,b)&=p(p(a,b,a),b,p(b,a,b))=p(c,b,c),\\
p'(b,b,a)&=p(p(b,a,b),b,p(a,b,a))=p(c,b,c),
\end{align*}
and similarly
\begin{align*}
    p'(d,b,b)&=p'(b,b,d)=p'(c,b,c)=p(c,b,c),\text{ and}\\
    p'(a,c,c)&=p'(c,c,a)=p'(d,c,c)=p'(c,c,d)=p'(b,c,b)=p(b,c,b)
\end{align*}
while for the last four tuples, we have
\begin{align*}
    p'(c,a,b)&=p(p(c,b,c),a,p(b,c,b)),\\
    p'(b,a,c)&=p(p(b,c,b),a,p(c,b,c)),\\
    p'(c,d,b)&=p(p(c,b,c),d,p(b,c,b)),\\
    p'(b,d,c)&=p(p(b,c,b),d,p(c,b,c)).
\end{align*}
It remains to determine $p(c,b,c)$ and $p(b,c,b)$. The case $p(c,b,c)=p(b,c,b)=a$ is not possible since otherwise $\{a,b,c\}$ would be a subalgebra. Therefore, we have two cases up to isomorphism.

First, let $p(c,b,c)=p(b,c,b)=d$. In this case, each of the values of $p'$ for inputs outside of \Cref{3affcommon} equals $d$. Hence, this algebra is $\m t_{4,15}$ from \Cref{3affremain}. It has the set $\{b,c,d\}$ as a subuniverse. Because of that, the values $p(c,b,c)$ and $p(b,c,b)$ must be equal to $d$ in every Taylor reduct, so $\m t_{4,15}$ is a minimal Taylor algebra. 

Suppose now that $p(c,b,c)=a$ and $p(b,c,b)=d$. Here
$$
\begin{gathered}
    p'(a,b,b)=p'(b,b,a)=p'(d,b,b)=p'(b,b,d)=\\
    p'(c,b,c)=p'(c,a,b)=p'(b,a,c)=a,\text{ and}\\
    p'(a,c,c)=p'(c,c,a)=p'(d,c,c)=p'(c,c,d)=\\
    p'(b,c,b)=p'(c,d,b)=p'(b,d,c)=d.
\end{gathered}
$$
follows. In other words, the operation $p'$ corresponds to the operation $p$ of the algebra $\m t_{4,16}$ from \ref{3affremain}. In this way, the automorphism that switches $a$ with $d$ and $b$ with $c$ is obtained. Hence, every time we have $p''(c,b,c)=a$ (or $p''(c,b,c)=d$, which is the same up to isomorphism) for some ternary term $p''$, we get $p''(b,c,b)=d$ (or, respectively, $p''(b,c,b)=a$), so $\m t_{4,16}$ is a minimal Taylor algebra.
\end{proof}

\begin{claim}
If the two-element class of $\m a/\alpha$ is a affine subalgebra, then  $\m a\cong\m t_{4,17}$ or $\m a\cong\m t_{4,18}$.
\end{claim}

\begin{proof}
Recall the same term $p'$ as before; thus, the operation $p'$ on the set $\{a,d\}$ acts as the minority operation. Similar to the proof of the previous claim, we get that $\m a$ is one of the minimal Taylor algebras $\m a\cong\m t_{4,17}$ and $\m a\cong\m t_{4,18}$. The details of the proof are left to the reader.
\end{proof}



\end{document}